\documentclass[11pt]{article}
\usepackage[utf8]{inputenc}
\usepackage[T1]{fontenc}
\usepackage{comment}
\usepackage{authblk}
\usepackage{relsize}
\usepackage{changepage}
\usepackage{mathtools}
\usepackage[english]{babel}
\usepackage{amsmath}
\usepackage{amsthm}
\usepackage{amsfonts}
\usepackage{amssymb}
\usepackage{graphicx}
\usepackage[usenames,dvipsnames]{color}
\theoremstyle{plain}
\newtheorem{theorem}{Theorem}[section]
\newtheorem{corollary}[theorem]{Corollary}

\newtheorem{lemma}[theorem]{Lemma}

\newtheorem{observation}[theorem]{Observation}
\newtheorem{remark}[theorem]{Remark}

\allowdisplaybreaks

\usepackage{empheq}

\makeatletter
\newcommand{\vast}{\bBigg@{4}}
\newcommand{\Vast}{\bBigg@{5}}
\makeatother
\definecolor{bulgarianrose}{rgb}{0.28, 0.02, 0.03}
\definecolor{gray}{rgb}{0.5, 0.5, 0.5}

\theoremstyle{definition}
\newtheorem{definition}[theorem]{Definition}

\usepackage[left=2cm,right=2cm,top=2cm,bottom=2cm]{geometry}
\usepackage{amssymb}
\usepackage{hyperref}
\makeatletter
\def\namedlabel#1#2{\begingroup
    #2%
    \def\@currentlabel{#2}%
    \phantomsection\label{#1}\endgroup
}
\usepackage[normalem]{ulem}
\usepackage{dsfont}
\usepackage{accents}
\usepackage{verbatim}
\usepackage{ulem}
\usepackage{pgf,tikz,pgfplots}
\pgfplotsset{compat = 1.16}
\usepackage[all]{xy}

\newcommand{\covar}{\text{Cov}}
\newcommand{\variance}{\text{Var}}
\newcommand{\corr}{\text{Corr}}

\usepackage{subfig}

\usepackage{enumerate}

\title{\scshape
  On the minimum bisection of random 3-regular graphs}

\author[1,2]{Lyuben Lichev}
\author[1,2,3]{Dieter Mitsche\footnote{Dieter Mitsche has been supported by grant GrHyDy ANR-20-CE40-0002 and by Fondecyt grant 1220174.}}
\affil[1]{Institut Camille Jordan, Univ.\ de Lyon, France}
\affil[2]{Univ.\ Jean Monnet, Saint-Etienne, France}
\affil[3]{Pont.\ Univ.\ Cat\'{o}lica, IMC, Santiago de Chile}
\begin{document}

\maketitle
 
\begin{abstract}
In this paper we give new bounds on the bisection width of random 3-regular graphs on $n$ vertices. The main contribution is a new lower bound of $0.103295n$ based on a first moment method together with a structural analysis of the graph, thereby improving a 27-year-old result of Kostochka and Melnikov. We also give a complementary upper bound of $0.139822n$ by combining a result of Lyons with original combinatorial insights. Developping this approach further, we obtain a non-rigorous improved upper bound with the help of Monte Carlo simulations.
\end{abstract}

\hspace{1em}Keywords: random 3-regular graph, minimal bisection, bisection width

\hspace{1em}MSC Class: 05C80, 68R10, 05D40, 05C30

\section{Introduction}
Given a graph $G=(V,E)$ with even number of vertices (denoted by $n$), a \emph{bisection of $G$} is a partition of the vertex set of $G$ into two equal parts, that is, subsets of equal sizes. A bisection is a \emph{minimum bisection} if the number of edges having endpoints in different parts is minimized. The \emph{bisection width} of a graph $G$, denoted by $bw(G)$, is the number of edges going between $A$ and $B$ in a minimum bisection.

The minimum bisection problem has received a lot of attention in mathematics, theoretical computer science and physics due to its applications in a number of graph layout and embedding problems such as the routing performance of a network~\cite{JosepSurvey}. The algorithmic problem of finding a minimum bisection is well known to be NP-complete in general~\cite{GareyJohnson} and even in the particular case of 3-regular graphs~\cite{Bui}. Moreover, even approximating the bisection width up to a constant factor is hard (see for example~\cite{Khot} and the references therein). On the positive side, $O(\log^2 n)$-approximation algorithms in polynomial time exist~\cite{Feige}, and exact polynomial time algorithms are available for graphs of bounded treewidth~\cite{Jansen}.

In the context of random regular graphs, the minimum bisection problem has also received quite a bit of attention. In this setting, one usually assumes that the number of vertices tends to infinity, and the following results all hold only with probability tending to 1 as $n \to \infty$.

Regarding the bisection width of random $3$-regular graphs, the first lower bound of 
$\frac{1}{11}n \approx 0.0909n$ was given by Bollob\'{a}s~\cite{BollobasIso}, and this was later improved by Kostochka and Melnikov to $0.10101n$~\cite{Kostochka}. Since then, during the last 27 years, to the best of our knowledge, no further improvements have been made. On the other hand, Kostochka and Melnikov~\cite{KostochkaOld} proved an upper bound of $\tfrac{1}{4}n+o(n)$ on the bisection width of any 3-regular graph, which was later improved to $\tfrac{1}{6} n$ for all sufficiently large 3-regular graphs by Monien and Preis~\cite{MonienPreis}. For random 3-regular graphs, a slightly weaker but simpler algorithmic upper bound of $0.1740 n$ was given by D\'{i}az, Do, Serna and Wormald~\cite{Diaz}. This bound was then improved by Lyons~\cite{Lyons} to $0.16226n$.

Concerning random $d$-regular graphs, the first lower bound for fixed $d\ge 3$ was given by Bollob\'{a}s~\cite{BollobasIso} who showed that the bisection width is at least $(\frac{d}{4}-\frac{\sqrt{d \log 2}}{2})n$. Independently, Clark and Entringer~\cite{Clark} observed that the bisection width is at least $(\tfrac{d}{4}+o_d(d)) n$ as $d \to \infty$. On the other hand, Alon~\cite{AlonBisection} provided an upper bound of $(\frac{d}{4}-c\sqrt{d})n$ for some sufficiently small positive constant $c$ and any $d\ge 3$. To our knowledge, the currently best known upper bound for $d$ with $5 \le d \le 12$ is due to D\'{i}az, Serna and Wormald~\cite{JosepDReg}, while for all other $d\ge 3$, Lyons~\cite{Lyons} gave an upper bound of $\frac{d}{2\pi}\arccos(\frac{2\sqrt{d-1}}{d}) n < (\tfrac{d}{4} - \tfrac{\sqrt{d}}{\pi}) n$. The case $d \to \infty$ was recently settled by Dembo, Montanari and Sen~\cite{Dembo}. Therein, the authors showed that the bisection width of a random $d$-regular graph is $(\frac{d}{4}-P_*\sqrt{\frac{d}{2}}+o_d(\sqrt{d})) n$ where $P_* \approx 0.7632$ denotes the ground state energy of the Sherington-Kirkpatrick model.

For the Erd\H{o}s-R\'{e}nyi graph $G(n,p)$ with $p=\tfrac{c}{n}$ for some constant $c > 0$, Luczak and McDiarmid~\cite{ColinMalvina} identified a phase transition: they showed that for $c < \log{4}$, the largest component has size less than $\tfrac{1}{2} n$ and the bisection width of $G(n,p)$ is 0, whereas for $c > \log{4}$, the bisection width is already $\Omega(n)$. For denser graphs with $p=\tfrac{c}{n}$ and $c \to \infty$, Dembo, Montanari and Sen~\cite{Dembo} later showed that the bisection width of $G(n,p)$ has value $\left(\frac{c}{4}-P_*\sqrt{\frac{c}{4}}+o_c(\sqrt{c})\right)n$ with $P_*$ as above.


Regarding more general results, it is well known (at least since the 70's, see Fiedler~\cite{Fiedler}) that $\tfrac{\lambda_2}{4} n$ is a lower bound for the bisection width for any graph where $\lambda_2$ is the second eigenvalue of the Laplacian of the graph. For random $d$-regular graphs, using Friedman's result~\cite{Friedman}, this translates into a lower bound of $(\frac{d}{4}-\frac{\sqrt{d-1}}{2}) n$ (in particular, giving $0.0428 n$ for random 3-regular graphs). Later, several improvements using spectral techniques have been made: for example, Bezroukov, Els\"{a}sser, Monien, Preis and Tillich~\cite{Bezroukov} gave a lower bound of $0.082n$ on the bisection width of 3-regular Ramanujan graphs. 

A different line of research focused on estimating the bisection width in a planted bisection model. More precisely, given an unknown partition of the vertex set into two sets of equal sizes (corresponding to a planted bisection), add an edge between two vertices of the same part with probability $p^+$, and add an edge between two vertices of different parts with probability $p^- < p^+$, independently for different edges. An asymptotic formula for the bisection width in this setup was found by Coja-Oghlan, Cooley, Kang and Skubch~\cite{Amin} when the difference between $p^+$ and $p^-$ is sufficiently large. This result was further extended by Sen~\cite{Sen} (for more references on the planted bisection problem, see~\cite{GV15}). Closely related to the minimum bisection problem on random $3$-regular graphs is also the prominent conjecture of Zdeborov\'{a} and Boettcher~\cite{Zdeborova} who conjectured that the bisection width of a random 3-regular graph is equal to the number of edges not crossing a maximum cut up to $o(n)$. Unfortunately, our paper does not shed light on this conjecture. Weak indications in its favor come from the fact that the upper bound on the bisection width in Theorem~\ref{thm:main} is the same as the upper bound for the edges not crossing a maximum cut from~\cite{Gamarnik}. For recent advances on the maximum cut of random $3$-regular graphs, see also~\cite{C-OLMS}.

In this paper we focus on improving the results on random 3-regular graphs. Denote by $\mathcal G_d(n)$ the set of all $d$-regular (multi)graphs. Note that graphs in $\mathcal G_d(n)$ can have loops and multiple edges; graphs without loops and without multiple edges are called \textit{simple}. Denote also by $G(n,3)$ the random 3-regular graph with $n$ vertices following the uniform distribution over the set $\mathcal G_3(n)$.

For a sequence of probability spaces $(\Omega_n, \mathcal F_n, \mathbb P_n)_{n\geq 1}$ and a sequence of events $(A_n)_{n\geq 1}$ where $A_n\in \mathcal F_n$ for every $n\geq 1$, we say that $(A_n)_{n\geq 1}$ happens \textit{asymptotically almost surely} or \textit{a.a.s.}, if $\underset{n\to +\infty}{\lim}\mathbb P_n(A_n) = 1$. The sequence of events $(A_n)_{n\geq 1}$ itself is said to be \textit{asymptotically almost sure} or again \textit{a.a.s.} 

\begin{theorem}\label{thm:main}
A.a.s.\ $0.103295 n \le bw(G(n,3)) \le 0.139822 n$.
\end{theorem}

Since the proof of the main theorem is rather cumbersome and relatively long, we provide a detailed overview with pointers to different observations in the introduction. In order to be precise, we first introduce the necessary notation.

\subsection{Notation}
We denote by $\mathbb N$ the set of all positive integers and by $\mathbb Z_{\ge 0}$ the set of all non-negative integers. For every $d\in \mathbb N$, we denote by $[d]$ the set of integers $\{i \mid 1\le i\le d\}$. 

For a graph $H=(V,E)$ and $U \subseteq V$, we denote by $H[U]$ the subgraph of $H$ induced by the vertices in $U$. For a vertex $v \in V$, we denote by $N(v)$ the neighborhood of $v$ in $H$, that is, $N(v) = \{u\in V: uv\in E\}$, and $N[v] = v\cup N(v)$. The lowercase letters $u, v, w$ will be reserved to denote vertices and the lowercase letters $e,f$ will be reserved to denote edges, possibly with some lower or upper indices. The \textit{order} of a graph $H$ is the number of vertices in $H$, and the \textit{size} of $H$ is the number of edges of $H$. For $k \in \mathbb{N}$, the \textit{$k$-core} $C_k(H)$ of a graph $H$ is the (unique) largest subgraph of $H$ with respect to inclusion in which every vertex is of degree at least $k$.

A \textit{subdivision of an edge $e = uv$} in a graph $H$ is an operation of deleting the edge $e$ and adding one new vertex $w$ together with the edges $uw$ and $vw$. Consecutive subdivisions of the edges of a graph $H$ produce a new graph $\overline{H}$ from $H$. By abuse of terminology we call the graph $\overline{H}$ itself \textit{subdivision of the graph $H$}. For a subdivision $\overline{H}$ of $H$, an edge $e$ and a vertex $w$ as above, we say that 
the vertex $w$ \textit{subdivides} the edge $e$ in $H$. For a graph $H=(V,E)$, we say that three vertices $u,v,w \in V$ form a \textit{cherry} with center $v$ and endpoints $u,w$ if $uv \in E$ and $vw \in E$. For a vertex $v$ in $H$, we say that $v$ is a \textit{leaf} of $H$ if $v$ is of degree one in $H$. Moreover, given a vertex $v\in V$ and a non-negative integer $r$, we denote by $B_H(v,r)$ the ball with center $v$ and radius $r$ consisting of all vertices at (graph) distance at most $r$ from $v$ in $H$.

For a positive integer $k$, the $k$-\textit{star} is the complete bipartite graph with parts of size 1 and $k$. A graph is a \textit{star} if it is a $k$-star for some $k\ge 1$. In a subdivision of a star, a \textit{branch} is a path that starts from its center and ends at some of the leaves. 



A \textit{cut} $(V_1, V_2)$ of a graph $H=(V,E)$ is a partition of $V$ into two non-empty sets with $V_1 \cup V_2 = V$. Finally, the \textit{size} of a cut $(V_1,V_2)$, denoted by $e(V_1, V_2)$, is the number of edges with one endvertex in $V_1$ and one endvertex in $V_2$.

\subsection{Detailed overview of the proofs.} 
The main contribution of this paper is the lower bound of Theorem~\ref{thm:main}.

The most difficult part of the proof of the lower bound is dedicated to a detailed characterization of structural properties that a minimum bisection typically possesses. On a high level, we show that if certain (forbidden) substructures appear, then one could switch vertices between the parts of the bisection and thus reduce the total number of edges going between these two parts. We use this to do a refined first moment computation. 

The structural characterization goes as follows. To begin with, we show that a.a.s.\ the random graph $G(n,3)$ is \emph{usual} meaning that its bisection width is at least $0.1n$ by the result of Kostochka and Melnikov (Theorem~\ref{Kostochka Melnikov}), and it has at most $\log n$ vertices in cycles of length at most $20$ (we remark that in the above definition, any linear lower bound on the bisection width is sufficient for our purposes.) Given a minimum bisection $(V_1, V_2)$ in a usual graph $G$, we show (via some elementary analysis of the local structure of the two parts) that for both $i=1,2$, $V_i$ contains a linear number of vertex pairs $u,v$ such that both vertices $u,v$ are of degree 2 in $G[V_i]$, and $u$ and $v$ are at (graph) distance at most 4 from each other in $G[V_i]$ (Lemma~\ref{distance}). This allows us to deduce that up to $O(1)$ vertices, $G[V_i]$ and its 2-core coincide (Lemma~\ref{2-core big}). Exchanging the vertices outside the 2-cores provides a cut $(U_1, U_2)$ which is almost a bisection (the differences of the two parts is $O(1)$), has minimum size among the cuts $(W_1, W_2)$ with $|W_1| - |W_2| = |U_1| - |U_2|$ (Lemma~\ref{new cut} and Corollary~\ref{new cut cor}), and both $G[U_1]$ and $G[U_2]$ coincide with their 2-cores. We will then deal with such minimum ``almost'' bisections with parts $U_1$ and $U_2$.


Once the cut $(U_1, U_2)$ is defined, for both $i=1,2$, we encode the graph $G[U_i]$ by a weighted 3-regular graph $G_{3,i}$ where the weight of each edge corresponds to the number of times the edge was subdivided in the construction of $G[U_i]$ (see Definition~\ref{defn:graphs1}, where we also define the subgraph $G_{3,i}^+$ of $G_{3,i}$ consisting of the edges of positive weight). By using that the cut $(U_1,U_2)$ cannot be reduced by keeping the sizes of $U_1$ and $U_2$ fixed, we show in Observations~\ref{choice 1} and~\ref{choice 2} that two cases emerge. 

In the first case, there is $i\in \{1,2\}$ such that the sum of weights of the edges in $G_{3,i}$ having weight at least three is more than seven. Then, $G_{3,3-i}^+$ has $O(1)$ vertices of degree three and $O(1)$ edges of weight at least two. This gives rise to a first type of cuts $(U_1, U_2)$ where $G_{3,3-i}^+$ essentially contains only paths and cycles of edges of weight one while the structure of $G_{3,i}$ remains unresticted. 

In the second case, both $G_{3,1}^+$ and $G_{3,2}^+$ have $O(1)$ edges of weight at least three. For both $i=1,2$, we define the weighted graph $G_{3,i}^{\le 2}$ as the subgraph of $G_{3,i}^+$ obtained by deleting these edges (again, see Definition~\ref{defn:graphs1}). We call a vertex in $G_{3,i}^{\le 2}$ \emph{critical} if it either has degree three or is incident to an edge of weight two, and define $S_i\subseteq V(G[U_i])$ as the set of all critical vertices in $G_{3,i}^{\le 2}$ (seen in $G[U_i]$) together with the vertices in $G[U_i]$ that subdivide the edges, which are incident to the critical vertices in $G_{3,i}^{\le 2}$. At the end of Section~\ref{section 2} we show that for both $i=1,2$ and every $\ell\ge 2$, either $|S_i|\le 52\ell+1091$ or one may delete $O(1)$ edges from $G_{3,3-i}^{\le 2}$ to ensure that no path of length $\ell/3$ and consisting of edges of weight 1 connects two critical vertices (Corollary~\ref{cut lemma cor}). Finally, assuming without loss of generality that $|S_1|\le |S_2|$, we study the cases $|S_1|\le \log^2 n$ (bisections of type one, which also cover the first case described in the previous paragraph) and $|S_1|\ge \log^2 n$ (bisections of type two).

For both types of bisections, we perform a first moment method. 
In Section~\ref{section 3}, we deal with bisections of type one. First, we count the number of possible skeletons for the graph $G_{pc,1} = G_{3,1}^+\setminus S_1$ of order $\beta n$ for some $\beta\in [0.1, 0.5]$. Then, we fix one possible skeleton and label its vertices. Once this labeling is constructed, we bound from above the number of extensions of $G_{pc, 1}$ to $G[U_1]$, and consequently to $G$. Finally, we optimize with respect to the parameter $\beta$ to get that bisections of type one and size at most $0.1069 n$ are a.a.s.\ not contained in $G(n,3)$.


In Section~\ref{section 4}, we deal with bisections of type two. In this (slightly harder) case, we first define
an auxiliary graph 
that very much resembles $G_{3,i}^{\le 2}$ for both $i=1,2$ and only consists of paths with at most one edge of weight two, cycles with edges of weight one, subdivisions of 3-stars with edges of weight one, and $o(n)$ further edges.
Then, we apply the first moment method to count bisections according to the form of the auxiliary graph associated to $G_{3,1}^{\le 2}$ and to $G_{3,2}^{\le 2}$, respectively (which then can be transferred to a first moment method for the original graph). Since we want to count graphs instead of bisections, we also use an additional lemma (that may be of independent interest), which allows us to reduce the first moment: in Lemma~\ref{anticlique}, we show that in a bipartite graph $H$ of maximal degree two with parts $H_1$ and $H_2$, there exists an independent set having roughly $|V(H_1)|/2$ vertices in $H_1$ and also $|V(H_2)|/2$ vertices in $H_2$. This allows us to derive that one choice of auxiliary graphs gives rise to many different bisections of the same size, and therefore, we may divide by the corresponding overcounting factor.    

The complementary upper bound of Theorem~\ref{thm:main} has two main ingredients. It is based on an idea of Lyons~\cite{Lyons} using a result of Cs\'oka, Gerencs\'er, Harangi and Vir\'ag~\cite{Gerencser}. Knowing that the sequence of random regular graphs $(G(n,3))_{n\ge 1}$ a.s.\ locally converges to $T_3$, Lyons used a sequence of Gaussian processes $(X_v^{(n)})_{v\in V(G(n,3))}$ that converges in distribution to the Gaussian wave function associated to the second largest eigenvalue of the transition operator of the (infinite) 3-regular tree $T_3$. Then, he defined a cut of the graph $G(n,3)$ (that a.a.s.\ bisects it up to $o(n)$ vertices) based on the sign of the variables $(X_v^{(n)})_{v\in V(G(n,3))}$, and used some local improvements to deduce a bisection of size approximately $0.16226n$.
We refine this strategy by using more complicated local modifications based on the structural insights from the proof of the lower bound. 
More precisely, we define an auxiliary bipartite graph whose vertices are centers of \emph{border cherries} (that is, cherries whose leaves both have the opposite sign to the center of the cherry) and whose edges cross the cut described above (see also Figure~\ref{new fig 3}). Then, using Lemma~\ref{anticlique} again, we show that we may switch certain centers of border cherries, thereby reducing the size of the constructed bisection. We conclude with a non-rigorous improvement obtained by switching vertices based on an even more complicated substructure (unfortunately we are not able to compute the associated integral, not even numerically).

\vspace{1em}


\subsection{Organization of the paper.} In Section~\ref{section PR} we introduce basic concepts and lemmas used in the proof of the lower bound. In Section~\ref{section 2} we describe several forbidden subgraphs that do not appear in a minimum bisection of a typical 3-regular graph, ending with a characterization of two types of candidates for a minimum bisection. In Section~\ref{section 3} we compute the expected number of minimum bisections of type one. In Section~\ref{section 4} we do the same for minimum bisections of type two, this time with an additional regrouping with respect to the underlying 3-regular graph they originate from. Section~\ref{section 6} is devoted to the proof of the upper bound.

\section{Preliminaries}\label{section PR}
In this section we introduce a few basic concepts that will be used in the sequel.

\begin{observation}\label{trivial 3}
In a graph of maximum degree three, two cycles are vertex-disjoint if and only if they are edge-disjoint.
\end{observation}
\begin{proof}
Two vertex-disjoint cycles are clearly edge-disjoint. On the other hand, if two cycles have a common vertex but no common edge, then the degree of this common vertex must be at least four, thus contradicting the maximum degree condition in the statement.
\end{proof}

We now introduce the probability space we will be working in until the end of this paper. For two sequences of probability measures $(\mathbb P_n)_{n\ge 1}$ and $(\mathbb Q_n)_{n\ge 1}$ defined on sequence of spaces $(\Omega_n, \mathcal F_n)_{n\ge 1}$ respectively, we say that $(\mathbb P_n)_{n\ge 1}$ is \emph{contiguous} to $(\mathbb Q_n)_{n\ge 1}$ if for every sequence of measurable properties $(A_n)_{n\ge 1}$, $\lim_{n \to \infty} \mathbb P_n(A_n)=0$ implies $\lim_{n \to \infty} \mathbb Q_n(A_n)=0$.

\vspace{1em}

\noindent
\textbf{Configuration model.}
Given positive integers $d,n$ with $dn$ even, consider $dn$ points $(P_{i,j})_{1\leq i\leq d, 1\leq j\leq n}$ regrouped into $n$ buckets according to their second index. The \emph{configuration model} is the probability space of all perfect matchings of these $dn$ points equipped with the uniform probability measure. It was introduced by Bender and Canfield~\cite{BC} and further developed by Bollobás~\cite{Bol} and Wormald~\cite{W78}. 

We call \textit{configuration} a perfect matching of $(P_{i,j})_{i\in [d], j\in [n]}$. We also call \textit{partial configuration} a matching of $(P_{i,j})_{i\in [d], j\in [n]}$ which is not necessarily perfect. In order to transition from configurations to graphs, identify the buckets with the $n$ vertices of $G(n,d)$, and connect vertices $v$ and $v'$ by $k$ edges if the configuration contains $k$ pairs consisting of one point in the bucket $v$ and one point in the bucket $v'$ (in particular, one might add loops as well). It is well known that for any fixed value of $d$, this model is contiguous to the uniform distribution on simple $d$-regular graphs, see~\cite{J95b}.


\vspace{1em}

The following lemma is a standard result in the field of random graphs.
\begin{lemma}[\cite{NW}, Theorem 2.6 and \cite{Bol1}]\label{BW}
For every $\ell\ge 1$ and $d\ge 3$, the number of cycles of length $\ell$ in a random $d$-regular graph converges in distribution to a Poisson random variable.
\end{lemma}


The next well-known observation explains how to algorithmically construct the 2-core of a graph. For the sake of completeness, we give the proof of it as well.

\begin{observation}\label{leaf_del}
The $2$-core of a graph $H$ is well-defined and may be obtained by consecutive deletions of the vertices of degree zero and one.
\end{observation}
\begin{proof}
In the end of the deletion process, one obviously obtains a subgraph $H'$ of $H$ of minumum degree at least two. On the other hand, suppose for the sake of contradiction that there is another graph $H''\not \subseteq H'$ which has minimal degree at least two. Then, $H'\cup H''$ is also a subgraph of $H$ of minimum degree at least two. Let $v$ be the first vertex of $H''\setminus H'$ that has been deleted throughout the construction of $H'$. At the moment of its deletion, since $v\in H''$, $v$ had degree at least two, which is a contradiction. Thus, every subgraph of $H$ of minimal degree at least two is contained in $H'$, which proves the observation.
\end{proof}

\noindent
\textbf{Bounds on the number of partitions.}
We first state a weak version of the Hardy-Ramanujan theorem on the number of integer partitions that will be sufficient for our purposes.

\begin{theorem}[\cite{HR}]\label{Hardy - Ramanujan}
The number of partitions of an integer $n$ is $\exp(\Theta(\sqrt{n}))$ as $n\to \infty$.
\end{theorem}

The next lemma is an application of the previous theorem.
\begin{lemma}\label{stars}
For every integer $M\ge 3$, the number of unlabeled forests on $n$ vertices in which every connected component is a subdivision of a star with at most $M$ leaves, is $\exp(o(n)))$.
\end{lemma}
\begin{proof}
To construct a forest, one may first partition its $n$ vertices into subsets that induce the trees of the forest, and then decide on the structure of each tree separately. Since the number of partitions of $n$ is $\exp(o(n))$ by Theorem~\ref{Hardy - Ramanujan}, we only need to prove that, for a fixed partition, the number of ways to form a forest of the above type is at most $\exp(o(n))$. Observe that the number of unlabeled stars of order $t$ with at most $M$ leaves is given by the number of partitions of $t-1$ into at most $M$ parts. This number is equal to the number of non-negative integer solutions of the equation $x_1+\dots+x_M = t-1$, which is at most $(t-1)^M < t^M$.

On the other hand, for every $t\ge 1$, let $c_t$ be the number of sets of size $t$ in the partition of $n$. First, consider the vertex sets of size at least $\log n$. The number of forests induced by these vertices is bounded above by $\prod_{t\ge \log n} t^{c_t M}$ where $\sum_{t\ge \log n} tc_t\le n$. We conclude that
\begin{equation}\label{eq:UB1}
    \prod_{t\ge \log n} t^{c_t M} = \exp\left(\sum_{t\ge \log n} (\log t)c_t M\right)\le \exp\left(\left(\max_{t\ge \log n} \frac{M\log t}{t}\right) \sum_{t\ge \log n} tc_t\right) = \exp(o(n)).
\end{equation}

Next, for the smaller parts, we need to refine the previous upper bound. As before, the number of stars on $t$ vertices remains at most $t^M$. We count the number of ways to partition the $c_t$ sets of size $t$ into groups where one group consists of the vertex sets inducing a particular star. As every group has size between $0$ and $c_t$, and there are at most $t^M$ groups, the number of partitions as above is at most $(c_t+1)^{t^M}$. This means that the number of (unlabeled) forests induced by the vertices in the smaller parts is at most
\begin{equation}\label{eq:UB2}
    \prod_{t=1}^{\log n} (c_t+1)^{t^M} \le \prod_{t=1}^{\log n} \left(\dfrac{n}{t}+1\right)^{t^{M}}\le (n+1)^{(\log n)^{M+1}} = \exp(o(n)). 
\end{equation}
Combining~\eqref{eq:UB1} and~\eqref{eq:UB2} finishes the proof.
\end{proof}

\section{Structural properties of minimum bisections of the random 3-regular graph}\label{section 2}
The aim of this section is to give a detailed description of the structure of minimum bisections of the random 3-regular graph. In the proof, we use the already mentioned lower bound of Kostochka and Melnikov~\cite{Kostochka}.

\begin{theorem}\label{Kostochka Melnikov}
The bisection width of $G(n,3)$ is a.a.s.\ at least $\tfrac{10}{99}n\approx 0.101n$.
\end{theorem}

First, we define some concepts used in the proof of Theorem~\ref{thm:main}. Let $G = (V, E)$ be a graph and let $(V_1, V_2)$ be a cut of $G$. We will aim to decrease the size of the cut $(V_1, V_2)$ while keeping the sizes of the sets $V_1$ and $V_2$ unchanged. For some $i\in \{1,2\}$ and $\ell\in \mathbb N$, a set $S \subseteq V_i$ is called $(i, \ell)$-winning (with respect to the cut $(V_1, V_2)$) if $e(V_1, V_2) - e(V_i\setminus S, V_{3-i}\cup S)=\ell$. For $\ell\in \mathbb N$, a subset $S$ of $V$ is $\ell$-\textit{winning} or simply \textit{winning} if there is $i\in \{1,2\}$ such that $S$ is entirely contained in $V_i$ and $S$ is an $(i, \ell)$-winning set. See Figure~\ref{fig 1}.

\begin{figure}
\centering
\begin{tikzpicture}[scale=0.7,line cap=round,line join=round,x=1cm,y=1cm]
\clip(-10,-3) rectangle (10,2.8);
\draw [rotate around={-89.37384154646209:(2.32,-0.03)},line width=0.5pt] (2.32,-0.03) ellipse (2.7303084572762595cm and 2.0261501108936297cm);
\draw [rotate around={90:(-2.68,0.03)},line width=0.5pt] (-2.68,0.03) ellipse (2.668954381026692cm and 2.200731125785604cm);
\draw [line width=0.5pt] (-2.68,1.54)-- (-2.66,-0.02);
\draw [line width=0.5pt] (-2.66,-0.02)-- (-2.68,-1.48);
\draw [line width=0.5pt] (-2.68,-1.48)-- (2.34,-1.86);
\draw [line width=0.5pt] (-2.66,-0.02)-- (2.34,-0.04);
\draw [line width=0.5pt] (-2.68,1.54)-- (2.3,1.8);
\draw [line width=0.5pt] (-2.68,1.54)-- (-3.76,1.54);
\draw [line width=0.5pt] (-2.68,-1.48)-- (-3.82,-1.54);
\draw [rotate around={90:(-2.68,0.03)},line width=0.5pt] (-2.68,0.03) ellipse (1.7897347321154529cm and 0.9607551255863177cm);
\begin{scriptsize}
\draw [fill=black] (2.34,-1.86) circle (1.5pt);
\draw [fill=black] (2.3,1.8) circle (1.5pt);
\draw [fill=black] (-2.68,1.54) circle (1.5pt);
\draw [fill=black] (-2.68,-1.48) circle (1.5pt);
\draw [fill=black] (-2.66,-0.02) circle (1.5pt);
\draw [fill=black] (2.34,-0.04) circle (1.5pt);
\draw [fill=black] (-3.76,1.54) circle (1.5pt);
\draw [fill=black] (-3.82,-1.54) circle (1.5pt);
\draw [fill=black] (-4.2,0) node {\Large{$V_1$}};
\draw [fill=black] (3.4,0) node {\Large{$V_2$}};
\draw [fill=black] (-3.2,0) node {\Large{$S$}};
\end{scriptsize}
\end{tikzpicture}
\caption{An illustration of a $1$-winning set $S$.}
\label{fig 1}
\end{figure}

A subset $S$ of $V$ is \textit{indifferent} if there is $i\in \{1,2\}$, for which $S\subseteq V_i$ and $e(V_1, V_2) = e(V_i\setminus S, V_{3-i}\cup S)$. A set $S\subseteq V_i$ is $(i, \ell)$-\textit{losing} if $e(V_1, V_2) - e(V_i\setminus S, V_{3-i}\cup S) = -\ell$, and $\ell$-\textit{losing} (or just \textit{losing}) if it is $(i, \ell)$-losing for some $i\in \{1,2\}$ and $\ell\in \mathbb N$. Finally, an \textit{improvement} of the cut $(V_1, V_2)$ is an operation of exchanging two sets $S_1$ and $S_2$, with $|S_1|=|S_2|$, $S_1\subseteq V_1$ and $S_2\subseteq V_2$, for which 
\begin{equation*}
e(V_1, V_2) - e(S_2\cup (V_1\setminus S_1), S_1\cup (V_2\setminus S_2)) \ge 1.
\end{equation*}
Thus, the operation of improvement of a given cut creates a new cut of smaller size. Moreover, it does not change the sizes of the two sets participating in the cut. In particular, if the cut is a bisection, improvements also produce a bisection.\\

In the sequel, we call a 3-regular graph $G$ of order $n$ \textit{usual} if the following two conditions are both satisfied:
\begin{itemize}
    \item the bisection width of $G$ is at least $0.10n$,
    \item there are at most $\log n$ vertices in cycles of length at most 20. 
\end{itemize}

We remark that the choice of 20 in the second point of the above definition is somehow arbitrary -- it could be replaced by every large enough positive integer.

\begin{observation}\label{typical}
A.a.s.\  $G(n,3)$  is usual.
\end{observation}
\begin{proof}
First, the bisection width of a uniformly chosen graph is a.a.s.\ at least $0.10n$ due to Theorem~\ref{Kostochka Melnikov}. Second, by Lemma~\ref{BW}, the number of cycles of fixed length converges in distribution to a Poisson random variable. Thus, the number of vertices in cycles of length at most 20 converges in distribution to a tight random variable, and is therefore less than $\log n$ a.a.s. Thus, a.a.s.\ both properties in the definition of a usual graph are satisfied.
\end{proof}

From now on, we fix a minimum bisection $(V_1, V_2)$ in a usual graph $G$.


\begin{observation}\label{balls ob}
For both $i=1$ and $i=2$, each of the following holds.
\begin{enumerate}
    \item For every vertex $u$ of degree two in $G[V_i]$ and for every $d\ge 0$, $|B_{G[V_i]}(u,d)|\le 2^{d+1}-1$. Moreover, if all vertices at distance at most $d$ from $u$ have degree three in $G[V_i]$ and none of them participates in cycles of length at most $2d$, the above bound is sharp.
    \item For every vertex $u$ of degree one in $G[V_i]$ and for every $d\ge 0$, $|B_{G[V_i]}(u,d)|\le 2^d$. Moreover, if all vertices at distance at most $d$ have degree three in $G[V_i]$ and none of them participates in cycles of length at most $2d$, the above bound is sharp.
\end{enumerate}
\end{observation}
\begin{proof}
The first point follows from a strong induction showing that for every $i\ge 1$, there are at most $2^i$ vertices at distance at most $i$ from $u$ in $G[V_i]$ with equality if the $i$-th neighborhood of $u$ in $G[V_i]$ is a binary tree of height $i$. The second point follows in a similar way.
\end{proof}

The previous lemma yields the following corollary.

\begin{corollary}\label{trivial 1}
For both $i=1$ and $i=2$ and for every large enough $n$, there are two vertices of degree at most two and at distance at most eight in $G[V_i]$. 
\end{corollary}
\begin{proof}
We argue by contradiction. By assumption, the balls of radius four around the vertices of degree one or two in $G[V_i]$ are disjoint. Since $G$ is usual, first, there are at least $0.05n$ vertices of degree one or two in $G[V_i]$, and second, the number of vertices participating in cycles of length at most 8 is at most $\log n$. Therefore, the balls of radius four around at least $0.05n - 8\log n$ of the vertices of degree one or two in $G[V_i]$ contain at least $\min(2^5-1, 2^4) = 16$ vertices. In total, this shows that the vertices in $G[V_i]$ are at least $16(0.05n - 8\log n) = 0.8n - o(n)$, which is a contradiction. The corollary is proved.
\end{proof}

\begin{corollary}\label{trivial 2}
For both $i=1$ and $i=2$ and for every large enough $n$, the number of vertices of degree $0$ or $1$ in $G[V_i]$ is at most $19$.
\end{corollary}
\begin{proof}
Fix $i\in \{1,2\}$. Since $(V_1, V_2)$ is a minimum bisection of $G$, there is no improvement of $(V_1, V_2)$ in $G$. By Corollary~\ref{trivial 1} there is a path $(v_j)_{j=0}^s$ in $G[V_{3-i}]$ where $s\le 8$ and $v_0,v_s$ are both of degree at most 2 in $G[V_{3-i}]$. Then, the vertices $(v_j)_{j=0}^s$ form a $(3-i,s-1)$-losing set.
We show that there cannot exist $s+1$ vertices in $V_i$ of degree at most one in $G[V_i]$ outside the 
neighborhood of the path $(v_j)_{j=0}^s$ in $G$. Indeed, such a set would be $(i,\ell)$-winning for some $\ell\ge s+1$, and exchanging it with the vertices $(v_j)_{j=0}^s$ in $V_{3-i}$ would lead to an improvement of $(V_1, V_2)$ in $G$. On the other hand, $(v_j)_{j=0}^s$ have at most $s+3$ neighbors in $V_i$ in $G$ (at most one for each of $(v_j)_{j=1}^{s-1}$, and at most two for $v_0$ and $v_s$), so in total the number of vertices of degree one in $G[V_i]$ is at most $s+(s+3) \le 19$.
\end{proof}

\begin{figure}[ht]
\centering
\begin{tikzpicture}[scale = 0.95, line cap=round,line join=round,x=1cm,y=1cm]
\clip(-6.5,-4.5) rectangle (7,5);
\draw [line width=0.5pt] (0.3,0.22) circle (4.689434933976587cm);
\draw [rotate around={89.63965393661681:(-2.11,0.07)},line width=0.5pt,color=red] (-2.11,0.07) ellipse (2.3113503987691937cm and 1.6775400638704614cm);
\draw [line width=0.5pt] (-2.1,1.66)-- (-2.7,1.16);
\draw [line width=0.5pt] (-2.1,1.66)-- (-2.08,0.9);
\draw [line width=0.5pt] (-2.08,0.9)-- (-1.44,1.26);
\draw [line width=0.5pt] (-2.08,0.9)-- (-1.72,0.34);
\draw [line width=0.5pt] (-1.72,0.34)-- (-2.22,-0.14);
\draw [line width=0.5pt] (-2.7,1.16)-- (-2.9,0.46);
\draw [line width=0.5pt] (-2.9,0.46)-- (-2.8,-0.28);
\draw [line width=0.5pt] (-2.7,1.16)-- (-3.12,1.44);
\draw [line width=0.5pt] (-2.8,-0.28)-- (-3.36,-0.46);
\draw [line width=0.5pt] (-2.8,-0.28)-- (-2.56,-0.86);
\draw [line width=0.5pt] (-2.56,-0.86)-- (-2.12,-1.52);
\draw [line width=0.5pt] (-2.56,-0.86)-- (-2.96,-1.18);
\draw [line width=0.5pt] (-2.12,-1.52)-- (-1.9,-0.84);
\draw [line width=0.5pt] (-1.9,-0.84)-- (-2.22,-0.14);
\draw [line width=0.5pt] (-2.22,-0.14)-- (-1.32,-0.42);
\draw [line width=0.5pt] (-1.44,1.26)-- (-0.24,1.24);
\draw [line width=0.5pt] (-0.24,1.24)-- (0.3,0.22);
\draw [line width=0.5pt] (-1.32,-0.42)-- (0,-0.52);
\draw [line width=0.5pt] (0,-0.52)-- (0.3,0.22);
\draw [line width=0.5pt] (-1.32,-0.42)-- (-1.4,-1.08);
\draw [line width=0.5pt] (-1.44,1.26)-- (-1.62,1.74);
\draw [line width=0.5pt] (-0.24,1.24)-- (-0.04,1.78);
\draw [line width=0.5pt] (0,-0.52)-- (0.16,-1.14);
\draw [line width=0.5pt] (-3.12,1.44)-- (-3.44,2.22);
\draw [line width=0.5pt] (-3.44,2.22)-- (-3.04,2.7);
\draw [line width=0.5pt] (-3.04,2.7)-- (-2.4,3.2);
\draw [line width=0.5pt] (-3.04,2.7)-- (-2.32,2.72);
\draw [line width=0.5pt] (-0.8252901228249374,0.4738536154612757) circle (3.3509753713024413cm);
\draw [line width=0.5pt] (3.02,3.04)-- (2.64,3.3);
\draw [line width=0.5pt] (2.64,3.3)-- (2.26,3.76);
\draw [line width=0.5pt] (2.64,3.3)-- (2.08,3.3);
\draw [line width=0.5pt] (3.02,3.04)-- (3.32,2.58);
\draw [line width=0.5pt] (3.32,2.58)-- (3.18,2.02);
\draw [line width=0.5pt] (3.32,2.58)-- (3.84,2.26);
\draw [line width=0.5pt] (3.8,-0.1)-- (3.78,0.44);
\draw [line width=0.5pt] (3.78,0.44)-- (3.96,0.9);
\draw [line width=0.5pt] (3.78,0.44)-- (3.44,0.76);
\draw [line width=0.5pt] (3.8,-0.1)-- (3.44,-0.58);
\draw [line width=0.5pt] (3.44,-0.58)-- (2.88,-0.64);
\draw [line width=0.5pt] (3.44,-0.58)-- (3.48,-1.18);
\draw [line width=0.5pt] (1.76,-3.34)-- (1.72,-2.34);
\draw [line width=0.5pt] (1.72,-2.34)-- (1.2,-1.78);
\draw [line width=0.5pt] (1.72,-2.34)-- (2.02,-1.88);
\draw [line width=0.5pt] (1.76,-3.34)-- (0.92,-3.1);
\draw [line width=0.5pt] (0.92,-3.1)-- (0.56,-2.38);
\draw [line width=0.5pt] (0.92,-3.1)-- (0.56,-3.42);
\draw [line width=0.5pt] (-1.26,-3.66)-- (-1.1,-3.28);
\draw [line width=0.5pt] (-1.1,-3.28)-- (-1.06,-2.72);
\draw [line width=0.5pt] (-1.1,-3.28)-- (-0.66,-3.26);
\draw [line width=0.5pt] (-1.26,-3.66)-- (-1.58,-3.3);
\draw [line width=0.5pt] (-1.58,-3.3)-- (-1.78,-2.94);
\draw [line width=0.5pt] (-1.58,-3.3)-- (-2,-3.28);
\draw [line width=0.5pt] (1.2,-1.78)-- (1.36,-0.96);
\draw [line width=0.5pt] (1.36,-0.96)-- (0.8,-0.64);
\draw [line width=0.5pt] (1.36,-0.96)-- (1.56,-0.18);
\draw [line width=0.5pt] (1.56,-0.18)-- (2.12,0.4);
\draw [line width=0.5pt] (1.56,-0.18)-- (1.14,0.7);
\draw [line width=0.5pt] (1.14,0.7)-- (1.5,1.28);
\draw [line width=0.5pt] (1.14,0.7)-- (0.62,1.34);
\draw [line width=0.5pt] (0.62,1.34)-- (0.96,1.96);
\draw [line width=0.5pt] (0.62,1.34)-- (-0.04,1.78);
\draw [line width=0.5pt] (-0.04,1.78)-- (-0.26,2.46);
\draw [line width=0.5pt] (-1.06,-2.72)-- (-0.46,-2.3);
\draw [line width=0.5pt] (-0.46,-2.3)-- (-0.58,-1.74);
\draw [line width=0.5pt] (-0.58,-1.74)-- (0.08,-1.74);
\draw [line width=0.5pt] (0.08,-1.74)-- (0.16,-1.14);
\draw [line width=0.5pt] (-0.46,-2.3)-- (-0.18,-2.5);
\draw [line width=0.5pt] (-1.06,-2.72)-- (-1.2,-2.2);
\draw [line width=0.5pt] (-0.58,-1.74)-- (-0.42,-1.28);
\draw [line width=0.5pt] (0.08,-1.74)-- (0.52,-1.74);
\draw [line width=0.5pt] (0.16,-1.14)-- (0.64,-1.2);
\draw [rotate around={-40.74760233270472:(2.944477879983923,2.8455637588635256)},line width=0.5pt] (2.944477879983923,2.8455637588635256) ellipse (1.7142314203716758cm and 0.6666612288565226cm);
\draw [rotate around={71.25613371457588:(3.4998876061880053,-0.2028569743234951)},line width=0.5pt] (3.4998876061880053,-0.2028569743234951) ellipse (1.6186865438143336cm and 0.6842370782223772cm);
\draw [rotate around={51.53083316894859:(1.340368137478371,-2.799052014630474)},line width=0.5pt] (1.340368137478371,-2.799052014630474) ellipse (1.4053879434615948cm and 1.0705783653959433cm);
\draw [line width=0.5pt] (-1.34026095582918,-3.188066363054548) circle (0.85cm);
\begin{scriptsize}
\draw [fill=black] (0.3,0.22) circle (3pt);
\draw [fill=black] (-2.12,-1.52) circle (3pt);
\draw [fill=black] (-2.1,1.66) circle (3pt);
\draw [fill=black] (-2.7,1.16) circle (1.5pt);
\draw [fill=black] (-2.08,0.9) circle (1.5pt);
\draw [fill=black] (-1.44,1.26) circle (1.5pt);
\draw [fill=black] (-1.72,0.34) circle (3pt);
\draw [fill=black] (-2.22,-0.14) circle (1.5pt);
\draw [fill=black] (-2.9,0.46) circle (3pt);
\draw [fill=black] (-2.8,-0.28) circle (1.5pt);
\draw [fill=black] (-3.12,1.44) circle (3pt);
\draw [fill=black] (-3.36,-0.46) circle (1.5pt);
\draw [fill=black] (-2.56,-0.86) circle (1.5pt);
\draw [fill=black] (-2.96,-1.18) circle (1.5pt);
\draw [fill=black] (-1.9,-0.84) circle (3pt);
\draw [fill=black] (-1.32,-0.42) circle (1.5pt);
\draw [fill=black] (-0.24,1.24) circle (1.5pt);
\draw [fill=black] (0,-0.52) circle (1.5pt);
\draw [fill=black] (-1.4,-1.08) circle (1.5pt);
\draw [fill=black] (-1.62,1.74) circle (1.5pt);
\draw [fill=black] (-0.04,1.78) circle (1.5pt);
\draw [fill=black] (0.16,-1.14) circle (1.5pt);
\draw [fill=black] (-3.44,2.22) circle (3pt);
\draw [fill=black] (-3.04,2.7) circle (1.5pt);
\draw [fill=black] (-2.4,3.2) circle (1.5pt);
\draw [fill=black] (-2.32,2.72) circle (1.5pt);
\draw [fill=black] (3.02,3.04) circle (3pt);
\draw [fill=black] (2.64,3.3) circle (1.5pt);
\draw [fill=black] (2.26,3.76) circle (1.5pt);
\draw [fill=black] (2.08,3.3) circle (1.5pt);
\draw [fill=black] (3.32,2.58) circle (1.5pt);
\draw [fill=black] (3.18,2.02) circle (1.5pt);
\draw [fill=black] (3.84,2.26) circle (1.5pt);
\draw [fill=black] (3.8,-0.1) circle (3pt);
\draw [fill=black] (3.78,0.44) circle (1.5pt);
\draw [fill=black] (3.96,0.9) circle (1.5pt);
\draw [fill=black] (3.44,0.76) circle (1.5pt);
\draw [fill=black] (3.44,-0.58) circle (1.5pt);
\draw [fill=black] (2.88,-0.64) circle (1.5pt);
\draw [fill=black] (3.48,-1.18) circle (1.5pt);
\draw [fill=black] (1.76,-3.34) circle (3pt);
\draw [fill=black] (1.72,-2.34) circle (1.5pt);
\draw [fill=black] (1.2,-1.78) circle (1.5pt);
\draw [fill=black] (2.02,-1.88) circle (1.5pt);
\draw [fill=black] (0.92,-3.1) circle (1.5pt);
\draw [fill=black] (0.56,-2.38) circle (1.5pt);
\draw [fill=black] (0.56,-3.42) circle (1.5pt);
\draw [fill=black] (-1.26,-3.66) circle (3pt);
\draw [fill=black] (-1.1,-3.28) circle (1.5pt);
\draw [fill=black] (-1.06,-2.72) circle (1.5pt);
\draw [fill=black] (-0.66,-3.26) circle (1.5pt);
\draw [fill=black] (-1.58,-3.3) circle (1.5pt);
\draw [fill=black] (-1.78,-2.94) circle (1.5pt);
\draw [fill=black] (-2,-3.28) circle (1.5pt);
\draw [fill=black] (1.36,-0.96) circle (1.5pt);
\draw [fill=black] (0.8,-0.64) circle (1.5pt);
\draw [fill=black] (1.56,-0.18) circle (1.5pt);
\draw [fill=black] (2.12,0.4) circle (1.5pt);
\draw [fill=black] (1.14,0.7) circle (1.5pt);
\draw [fill=black] (1.5,1.28) circle (1.5pt);
\draw [fill=black] (0.62,1.34) circle (1.5pt);
\draw [fill=black] (0.96,1.96) circle (1.5pt);
\draw [fill=black] (-0.26,2.46) circle (1.5pt);
\draw [fill=black] (-0.46,-2.3) circle (1.5pt);
\draw [fill=black] (-0.58,-1.74) circle (1.5pt);
\draw [fill=black] (0.08,-1.74) circle (1.5pt);
\draw [fill=black] (-0.18,-2.5) circle (1.5pt);
\draw [fill=black] (-1.2,-2.2) circle (1.5pt);
\draw [fill=black] (-0.42,-1.28) circle (1.5pt);
\draw [fill=black] (0.52,-1.74) circle (1.5pt);
\draw [fill=black] (0.64,-1.2) circle (1.5pt);
\draw [fill=black] (-3.6, 1.98) node {\Large{$u$}};
\draw [fill=black] (0.57,0.2) node {\Large{$v$}};
\draw [fill=black] (-1.58,-3.8) node {\Large{$w_1$}};
\draw [fill=black] (1.5,-3.6) node {\Large{$w_2$}};
\draw [fill=black] (3.4,-0.1) node {\Large{$w_3$}};
\draw [fill=black] (2.8,2.8) node {\Large{$w_4$}};
\draw [fill=black] (0.35,4.2) node {\Large{$G[V_i]$}};
\end{scriptsize}
\end{tikzpicture}
\caption{In the figure, the thick black vertices are the ones of degree two in $G[V_i]$ and the others are of degree one or three in $G[V_i]$. Note that some edges of $G[V_i]$ are not drawn to preserve the clarity of the figure. The vertices of $V'_i$ are the ones contained in the smallest of the three nested regions (encircled in red). The vertices $u$ and $v$ are at distance less than five from $V'_i$, but do not participate in $V'_i$. The vertices $w_1, w_2, w_3$ and $w_4$ are at distance at least five from $V'_i$ and therefore the balls of radius two in $G[V_i]$ around them are disjoint.}
\label{fig 3}
\end{figure}
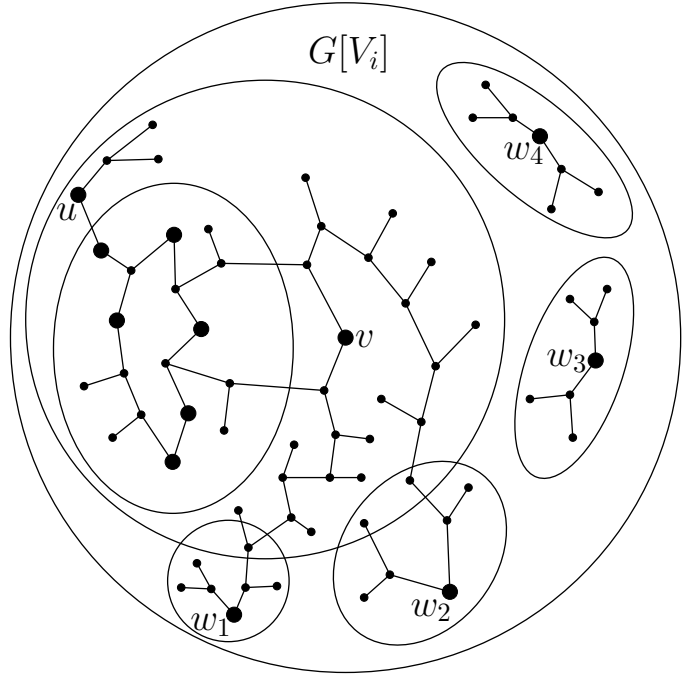

\begin{lemma}\label{distance}
For every large enough $n$ and for both $i=1$ and $i=2$, there are at least $\tfrac{n}{4000}$ disjoint pairs of vertices of degree $2$ in $G[V_i]$ at distance at most $4$.
\end{lemma}
\begin{proof}
We argue by contradiction. Call a pair of vertices in $G[V_i]$ \emph{good} if both have degree 2 and are at distance at most 4 in $G[V_i]$, and suppose that there are no $\tfrac{n}{4000}$ disjoint good pairs. Let us construct a set $V'_i\subset V_i$ as follows: as long as $V_i\setminus V'_i$ contains a good pair $u,v$, add both $u$ and $v$ to $V'_i$. In particular, $|V'_i|\le \tfrac{n}{2000}$ and for every good pair, at least one of its vertices belongs to $V'_i$.
Moreover, by maximality of $V'_i$, every vertex $v\in V_i$ of degree 2 and at distance at least 5 from $V'_i$ in $G[V_i]$ is such that $B_{G[V_i]}(v,4)$ contains only vertices of degree 3 in $G[V_i]$. We conclude that the balls of radius 2 around vertices of degree 2 in $G[V_i]$, which are at distance at least 5 from $V'_i$, are disjoint, see Figure~\ref{fig 3}. Since $G$ is a usual 3-regular graph and $(V_1, V_2)$ is a minimum bisection of $G$ by Corollary~\ref{trivial 2} the total number of vertices of degree 2 in $G[V_i]$ is at least $\tfrac{n}{10} - 3\cdot 19$ (every vertex participates in at most 3 edges between $V_1$ and $V_2$) and out of these vertices of degree 2 at least $\tfrac{n}{10} - 3\cdot 19 - (1+2)\cdot 19 = \tfrac{n}{10} - 114$ are at distance at least 3 from all leaves in $G[V_i]$. Moreover, $G$ contains at most $\log n$ vertices in cycles of length at most 4. By Observation~\ref{balls ob}, the number of vertices in $V_i \setminus V'_i$ should be at least
\begin{equation*}
\left(\dfrac{n}{10} - 114 - 4\log n - |V'_i|(2^5-1)\right)(2^3-1) \ge \dfrac{7n}{10} - \dfrac{217 n}{2000} - 28\log n - 798 \ge \dfrac{55 n}{100} 
\end{equation*}
for every large enough $n$. This is a contradiction since $|V_i| = \tfrac{n}{2}$.
\end{proof}

Next, we show that for both $i\in \{1,2\}$, almost all vertices in $V_i$ belong to the 2-core of $G[V_i]$.

\begin{lemma}\label{2-core big}
For both $i = 1$ and $i = 2$ and for all sufficiently large $n$, the graph $G[V_i]\setminus C_2(G[V_i])$ contains at most five vertices.
\end{lemma}
\begin{proof}
We argue by contradiction. Starting from the leaves of $G[V_i]$, we consecutively construct sets $S_0 = \emptyset \subset S_1\subset S_2\subset S_3\subset S_4 \subset S_5$ such that:

\begin{itemize}
    \item $\forall j\in [5],\hspace{0.2em} |S_j| = j$,
    \item for every $j\in [5]$ there is $\ell \ge j$, for which $S_j$ is an $(i,\ell)$-winning set with respect to $(V_1, V_2)$.
\end{itemize}

More precisely, for every $j\in [5]$, we construct $S_j$ from $S_{j-1}$ by adding a leaf or an isolated vertex of $G[V_i]\setminus S_{j-1}$. Notice that this construction is possible for every $j\in [5]$ by our assumption and Observation~\ref{leaf_del}, see Figure~\ref{fig 5}.

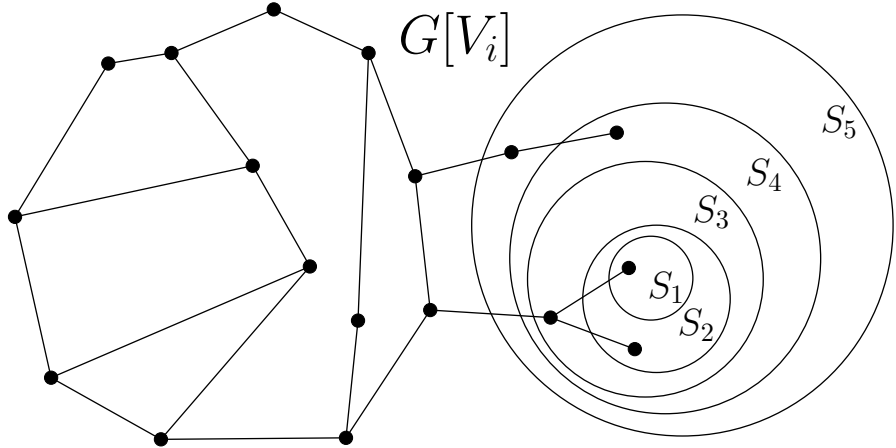
\begin{figure}
\centering
\begin{tikzpicture}[line cap=round,line join=round,x=1cm,y=1cm]
\clip(-8,-2.7) rectangle (10,4);
\draw [line width=0.5pt] (-5.36,0.64)-- (-4.88,-1.5);
\draw [line width=0.5pt] (-4.88,-1.5)-- (-3.42,-2.32);
\draw [line width=0.5pt] (-3.42,-2.32)-- (-1.44,-0.02);
\draw [line width=0.5pt] (-1.44,-0.02)-- (-2.2,1.32);
\draw [line width=0.5pt] (-2.2,1.32)-- (-5.36,0.64);
\draw [line width=0.5pt] (-5.36,0.64)-- (-4.12,2.68);
\draw [line width=0.5pt] (-4.12,2.68)-- (-3.28,2.82);
\draw [line width=0.5pt] (-3.28,2.82)-- (-2.2,1.32);
\draw [line width=0.5pt] (-4.88,-1.5)-- (-1.44,-0.02);
\draw [line width=0.5pt] (-3.28,2.82)-- (-1.92,3.4);
\draw [line width=0.5pt] (-1.92,3.4)-- (-0.66,2.82);
\draw [line width=0.5pt] (-0.66,2.82)-- (-0.04,1.18);
\draw [line width=0.5pt] (-3.42,-2.32)-- (-0.96,-2.3);
\draw [line width=0.5pt] (-0.96,-2.3)-- (-0.8,-0.74);
\draw [line width=0.5pt] (-0.8,-0.74)-- (-0.66,2.82);
\draw [line width=0.5pt] (-0.04,1.18)-- (0.16,-0.6);
\draw [line width=0.5pt] (0.16,-0.6)-- (-0.96,-2.3);
\draw [line width=0.5pt] (0.16,-0.6)-- (1.76,-0.7);
\draw [line width=0.5pt] (1.76,-0.7)-- (2.8,-0.04);
\draw [line width=0.5pt] (1.76,-0.7)-- (2.88,-1.12);
\draw [line width=0.5pt] (-0.04,1.18)-- (1.24,1.5);
\draw [line width=0.5pt] (1.24,1.5)-- (2.64,1.76);
\draw [line width=0.5pt] (3.16800540480629,-0.45074901150124186) circle (0.9800511272309143cm);
\draw [line width=0.5pt] (3.0939305248162268,-0.1754143328029211) circle (0.5576978962142771cm);
\draw [line width=0.5pt] (3.020374328667506,-0.19002621830375477) circle (1.5669760927282748cm);
\draw [line width=0.5pt] (3.283737265213575,0.0874024670086721) circle (2.0667378351171353cm);
\draw [line width=0.5pt] (3.512179352548176,0.5255275579784933) circle (2.801114600352043cm);
\begin{scriptsize}
\draw [fill=black] (-5.36,0.64) circle (1.5pt);
\draw [fill=black] (-4.88,-1.5) circle (1.5pt);
\draw [fill=black] (-3.42,-2.32) circle (1.5pt);
\draw [fill=black] (-1.44,-0.02) circle (1.5pt);
\draw [fill=black] (-2.2,1.32) circle (1.5pt);
\draw [fill=black] (-4.12,2.68) circle (1.5pt);
\draw [fill=black] (-3.28,2.82) circle (1.5pt);
\draw [fill=black] (-1.92,3.4) circle (1.5pt);
\draw [fill=black] (-0.66,2.82) circle (1.5pt);
\draw [fill=black] (-0.04,1.18) circle (1.5pt);
\draw [fill=black] (-0.96,-2.3) circle (1.5pt);
\draw [fill=black] (-0.8,-0.74) circle (1.5pt);
\draw [fill=black] (0.16,-0.6) circle (1.5pt);
\draw [fill=black] (1.76,-0.7) circle (1.5pt);
\draw [fill=black] (2.8,-0.04) circle (1.5pt);
\draw [fill=black] (2.88,-1.12) circle (1.5pt);
\draw [fill=black] (1.24,1.5) circle (1.5pt);
\draw [fill=black] (2.64,1.76) circle (1.5pt);
\draw [fill=black] (3.3,-0.3) node {\Large{$S_1$}};
\draw [fill=black] (3.7,-0.8) node {\Large{$S_2$}};
\draw [fill=black] (3.9,0.7) node {\Large{$S_3$}};
\draw [fill=black] (4.6,1.2) node {\Large{$S_4$}};
\draw [fill=black] (5.6,1.9) node {\Large{$S_5$}};
\draw [fill=black] (0.5,3) node {\huge{$G[V_i]$}};
\end{scriptsize}
\end{tikzpicture}
\caption{A possible choice of sets $S_1, S_2, S_3, S_4, S_5$ from the proof of Lemma~\ref{2-core big} for a given graph $G[V_i]$}
\label{fig 5}
\end{figure}

Now, note that there are at most $5(3^4-1) = O(1)$ vertices at distance at most 4 from $S_5$ in $G$. Thus, by Lemma~\ref{distance}, for every sufficiently large $n$ there is a pair of vertices $u,v$ of degree 2 and at distance $d\le 4$ in $G[V_{3-i}]$ such that no vertex on a shortest path between $u$ and $v$ is connected to $S_5$ in $G$. Denoting by $P$ the set of vertices on a shortest path between $u$ and $v$, we see that $P$ is a $(3-i, d-1)$-losing set without edges towards $S_{d+1}$. Hence, exchanging $P$ and $S_{d+1}$ leads to an improvement, a contradiction.
\end{proof}


Now, given a minimum bisection $(V_1, V_2)$ of $G$, let us first move all vertices in $G[V_1]\setminus C_2(G[V_1])$ to $V_2$, thus forming a set $\widetilde V_2$. Then, move the vertices in $G[\widetilde V_2]\setminus C_2(G[\widetilde V_2])$ back to $V\setminus \widetilde V_2$, thus forming a cut $(U_1, U_2)$ of $G$. 


Note that $G[U_2]$ is a graph with minimum degree 2. Moreover, $G[U_1]$ may be obtained from $C_2(G[V_1])$ by consecutively attaching vertices of degree at least 2 (corresponding to consecutive deletions of vertices of degree 0 or 1 in $G[\widetilde V_2]$). Thus, both $G[U_1]$ and $G[U_2]$ coincide with their 2-cores. Another easy observation is that $||U_1|-|U_2||\le 10$: indeed, the vertices in $\widetilde V_2$ outside the 2-core of $G[\widetilde V_2]$ either belong to $V_1$ (and were sent to $V_2$ during the first exchange) or belong to $V_2$ but remain outside the 2-core of $G[V_2]$. By Lemma~\ref{2-core big}, this number of vertices is at most 10.

The following lemma shows that, roughly speaking, the minimality assumption for the size of the bisection $(V_1, V_2)$ remains correct for the cut $(U_1, U_2)$.

\begin{lemma}\label{new cut}
For every large enough $n$, the cut $(U_1, U_2)$ obtained from $(V_1, V_2)$ has minimal size among the family of cuts $\{(W_1, W_2)\hspace{0.2em}|\hspace{0.2em}|W_1| - |W_2| = |U_1| - |U_2|\}$.
\end{lemma}
\begin{proof}
We argue by contradiction. Let $(W_1, W_2)$ be a cut with $|W_1| - |W_2| = |U_1| - |U_2|$ and of smaller size than $(U_1, U_2)$. Assume without loss of generality $|U_1| \ge |U_2|$, and note that moving a vertex from $V_1$ to $V_2$ or from $\widetilde V_2$ back to the first part (strictly) decreases the size of the cut. Hence, using that at least $(|U_1|-|U_2|)/2$ vertices changed their part during the exchange, the size of $(U_1, U_2)$ is at most $e(V_1, V_2) - (|U_1|-|U_2|)/2$. Notice that the cut $(W_1, W_2)$ contains at least $n/10 - 15$ edges since otherwise sending any set of $(|W_1|-|W_2|)/2 \le 5$ vertices from $W_1$ to $W_2$ would lead to two parts with the same number of vertices, which form a bisection of size less than $n/10$ in the usual graph $G$. We conclude that there is a set of at least $n/30-5$ vertices in $W_1$ with an edge to $W_2$ in $G$. Notice that every vertex with an edge in the cut is of degree at most two in the graph induced by its part. Therefore, by Observation~\ref{balls ob}, there is a set of at least $n/90 - 2$ non-neighboring vertices of degree at most two in $G[W_1]$. Sending any $(|W_1|-|W_2|)/2$ of them to $W_2$ produces a bisection of size $e(W_1, W_2) + (|W_1|-|W_2|)/2 < e(U_1, U_2) + (|U_1|-|U_2|)/2 \le e(V_1, V_2)$. This is a contradiction with the minimality of the size of the bisection $(V_1, V_2)$, which proves the lemma.
\end{proof}

\begin{corollary}\label{new cut cor}
The cut $(U_1, U_2)$ of the graph $G$ does not admit improvements.\qed
\end{corollary}

Since any constant difference between the sizes of the two parts will not alter subsequent ideas and calculations, we abuse terminology and call bisections these ``almost balanced'' cuts as well. However, we will keep the notation $(U_1, U_2)$ for such almost balanced cuts and $(V_1, V_2)$ for true bisections.\\


\begin{definition}\label{defn:graphs1}
For both $i=1$ and $i=2$, we define:
\begin{itemize}
    \item $G_{3,i}$ as the unique 3-core whose edges may be subdivided to obtain the graph $G[U_i]$. Moreover, we define a weight for every edge of $G_{3,i}$ equal to the number of times this edge should be subdivided in the construction of $G[U_i]$. The weight of the edge $e$ in $G_{3,i}$ will be denoted by $p(e)$, see Figure~\ref{fig 6}.
    \item $G^+_{3,i}$ as the graph obtained from $G_{3,i}$ by deleting the edges of weight 0.
    \item $G^{\le 2}_{3,i}$ as the graph obtained from $G^+_{3,i}$ by deleting all edges of weight more than 2.
\end{itemize}
\end{definition}

\begin{figure}
\centering
\begin{tikzpicture}[scale = 0.5, line cap=round,line join=round,x=1cm,y=1cm]
\clip(-19,-18) rectangle (8.72,5.5);
\draw [line width=0.5pt] (-8.64,-0.58)-- (-4.76,0.5);
\draw [line width=0.5pt] (-4.76,0.5)-- (-2.56,-1.26);
\draw [line width=0.5pt] (-4.76,0.5)-- (-4.02,2.86);
\draw [line width=0.5pt] (-4.02,2.86)-- (-7.08,2.78);
\draw [line width=0.5pt] (-4.02,2.86)-- (-1.8,1.68);
\draw [line width=0.5pt] (-7.08,2.78)-- (-6.44,4.78);
\draw [line width=0.5pt] (-6.44,4.78)-- (-3.34,5.4);
\draw [line width=0.5pt] (-1.8,1.68)-- (-1.2,3.54);
\draw [line width=0.5pt] (-1.2,3.54)-- (-3.34,5.4);
\draw [line width=0.5pt] (-6.44,4.78)-- (-3.92,4.06);
\draw [line width=0.5pt] (-3.92,4.06)-- (-3.34,5.4);
\draw [line width=0.5pt] (-3.92,4.06)-- (-1.2,3.54);
\draw [line width=0.5pt] (-8.64,-0.58)-- (-2.56,-1.26);
\draw [line width=0.5pt] (-8.64,-0.58)-- (-7.08,2.78);
\draw [line width=0.5pt] (-2.56,-1.26)-- (-1.8,1.68);

\draw [line width=0.5pt] (-8.64,-0.58-8)-- (-4.76,0.5-8);
\draw [line width=0.5pt] (-4.76,0.5-8)-- (-2.56,-1.26-8);
\draw [line width=0.5pt] (-4.76,0.5-8)-- (-4.02,2.86-8);
\draw [line width=0.5pt] (-4.02,2.86-8)-- (-7.08,2.78-8);
\draw [line width=0.5pt] (-4.02,2.86-8)-- (-1.8,1.68-8);
\draw [line width=0.5pt] (-7.08,2.78-8)-- (-6.44,4.78-8);
\draw [line width=0.5pt] (-6.44,4.78-8)-- (-3.34,5.4-8);
\draw [line width=0.5pt] (-1.8,1.68-8)-- (-1.2,3.54-8);
\draw [line width=0.5pt] (-1.2,3.54-8)-- (-3.34,5.4-8);
\draw [line width=0.5pt] (-6.44,4.78-8)-- (-3.92,4.06-8);
\draw [line width=0.5pt] (-3.92,4.06-8)-- (-3.34,5.4-8);
\draw [line width=0.5pt] (-3.92,4.06-8)-- (-1.2,3.54-8);
\draw [line width=0.5pt] (-8.64,-0.58-8)-- (-2.56,-1.26-8);
\draw [line width=0.5pt] (-8.64,-0.58-8)-- (-7.08,2.78-8);
\draw [line width=0.5pt] (-2.56,-1.26-8)-- (-1.8,1.68-8);

\draw [line width=0.5pt] (-4.76,0.5-16)-- (-2.56,-1.26-16);
\draw [line width=0.5pt] (-4.76,0.5-16)-- (-4.02,2.86-16);
\draw [line width=0.5pt] (-4.02,2.86-16)-- (-7.08,2.78-16);
\draw [line width=0.5pt] (-7.08,2.78-16)-- (-6.44,4.78-16);
\draw [line width=0.5pt] (-8.64,-0.58-16)-- (-2.56,-1.26-16);
\draw [line width=0.5pt] (-2.56,-1.26-16)-- (-1.8,1.68-16);

\draw [line width=0.2pt] (-20,-2)-- (10.5,-2);
\draw [line width=0.2pt] (-20,-10)-- (10.5,-10);

\begin{scriptsize}
\draw [fill=black] (-8.64,-0.58) circle (2.5pt);
\draw [fill=black] (-4.76,0.5) circle (2.5pt);
\draw [fill=black] (-2.56,-1.26) circle (2.5pt);
\draw [fill=black] (-4.02,2.86) circle (2.5pt);
\draw [fill=black] (-7.08,2.78) circle (2.5pt);
\draw [fill=black] (-1.8,1.68) circle (2.5pt);
\draw [fill=black] (-6.44,4.78) circle (2.5pt);
\draw [fill=black] (-3.34,5.4) circle (2.5pt);
\draw [fill=black] (-1.2,3.54) circle (2.5pt);
\draw [fill=black] (-3.92,4.06) circle (2.5pt);
\draw [fill=black] (-6.618444833924679,-0.8060949856794767) circle (2.5pt);
\draw [fill=black] (-4.89760184670628,-0.9985576881973239) circle (2.5pt);
\draw [fill=black] (-3.779512195121951,-0.2843902439024391) circle (2.5pt);
\draw [fill=black] (-2.36206307205136,-0.49429662083026066) circle (2.5pt);
\draw [fill=black] (-2.148499544527827,0.33185702511603693) circle (2.5pt);
\draw [fill=black] (-1.9833982561922525,0.97053832472997) circle (2.5pt);
\draw [fill=black] (-6.844992743105952,3.5143976777939048) circle (2.5pt);
\draw [fill=black] (-6.649927431059509,4.123976777939043) circle (2.5pt);
\draw [fill=black] (-4.368055973321127,1.7499836526515395) circle (2.5pt);
\draw [fill=black] (-5.5999658484525074,2.818693703308431) circle (2.5pt);

\draw [fill=black] (-8.64,-0.58-8) circle (2.5pt);
\draw [fill=black] (-4.76,0.5-8) circle (2.5pt);
\draw [fill=black] (-2.56,-1.26-8) circle (2.5pt);
\draw [fill=black] (-4.02,2.86-8) circle (2.5pt);
\draw [fill=black] (-7.08,2.78-8) circle (2.5pt);
\draw [fill=black] (-1.8,1.68-8) circle (2.5pt);
\draw [fill=black] (-6.44,4.78-8) circle (2.5pt);
\draw [fill=black] (-3.34,5.4-8) circle (2.5pt);
\draw [fill=black] (-1.2,3.54-8) circle (2.5pt);
\draw [fill=black] (-3.92,4.06-8) circle (2.5pt);

\draw [fill=black] (-8.64,-0.58-16) circle (2.5pt);
\draw [fill=black] (-4.76,0.5-16) circle (2.5pt);
\draw [fill=black] (-2.56,-1.26-16) circle (2.5pt);
\draw [fill=black] (-4.02,2.86-16) circle (2.5pt);
\draw [fill=black] (-7.08,2.78-16) circle (2.5pt);
\draw [fill=black] (-1.8,1.68-16) circle (2.5pt);
\draw [fill=black] (-6.44,4.78-16) circle (2.5pt);

\draw [fill=black] (-12,2) node 
{\huge{$G[U_i]$}};
\draw [fill=black] (-12,2-8) node {\huge{$G_{3, i}$}};
\draw [fill=black] (-12,2-16) node {\huge{$G^+_{3, i}$}};

\draw [fill=black] (-5.8,-17.3) node {\Large{$e_1$}};
\draw [fill=black] (-4.2,-16.5) node {\Large{$e_2$}};
\draw [fill=black] (-2.7,-15.75) node {\Large{$e_3$}};
\draw [fill=black] (-5,-14.4) node {\Large{$e_4$}};
\draw [fill=black] (-5.5,-12.8) node {\Large{$e_5$}};
\draw [fill=black] (-7.3,-12.15) node {\Large{$e_6$}};

\end{scriptsize}
\end{tikzpicture}
\caption{Top figure: An example of the graph $G[U_i]$. Middle figure: the corresponding graph $G_{3,i}$. Bottom figure: the graph $G^+_{3,i}$. In it, the weight of the edges $e_2, e_4, e_5$ is one, the weight of $e_1$ and $e_6$ is two and the weight of $e_3$ is three.}
\label{fig 6}
\end{figure}

\begin{observation}\label{choice 1}
For both $i = 1$ and $i = 2$, at most one of the following happens:
\begin{enumerate}
    \item The sum of the weights of the edges 
    \begin{equation*}
    \{e\in E(G_{3,i})\hspace{0.2em}|\hspace{0.2em} p(e)\ge 3\}
    \end{equation*}
    is more than five.
    \item $G_{3, 3-i}$ contains at least nine edges of weight at least two.
\end{enumerate}
\end{observation}
\begin{proof}
We argue by contradiction. Suppose that each of the above two events happens for some $i \in \{1,2\}$. We consider two cases: either there are two edges $e_1, e_2$ in $G_{3,i}$ with $p(e_1)\ge 3$ and $p(e_2)\ge 3$ or there is an edge $e_3$ in $G_{3,i}$ with $p(e_3)\ge 6$. We define $(w_i)_{1\le i\le 6}$ as follows:
\begin{itemize}
    \item In the first case, $(w_i)_{1\le i\le 3}$ are consecutive vertices subdividing the edge $e_1$ and $(w_i)_{4\le i\le 6}$ are consecutive vertices subdividing the edge $e_2$ in $G_{3,i}$ (see the left part of Figure~\ref{fig 7}).
    \item In the second case, $(w_i)_{1\le i\le 6}$ are consecutive vertices subdividing the edge $e_3$ in $G_{3,i}$.
\end{itemize}

By assumption there are edges $(f_j)_{1\le j\le 9}$ in $G_{3, 3-i}$, each of weight at least two. Moreover, there are at least three of these edges that are subdivided by vertices, none of which is adjacent to any of $(w_i)_{1\le i\le 6}$. Let $f'_1, f'_2, f'_3$ be one such choice of edges among $(f_j)_{1\le j\le 9}$. Then, exchanging the set of vertices $(w_i)_{1\le i\le 6}$ in $U_i$ and three pairs of neighboring vertices of degree two in $G[U_{3-i}]$ subdividing $f'_1, f'_2$ and $f'_3$ respectively, leads to an improvement of the cut $(U_1, U_2)$. This is a contradiction with Corollary~\ref{new cut cor}, which proves the observation.
\end{proof}

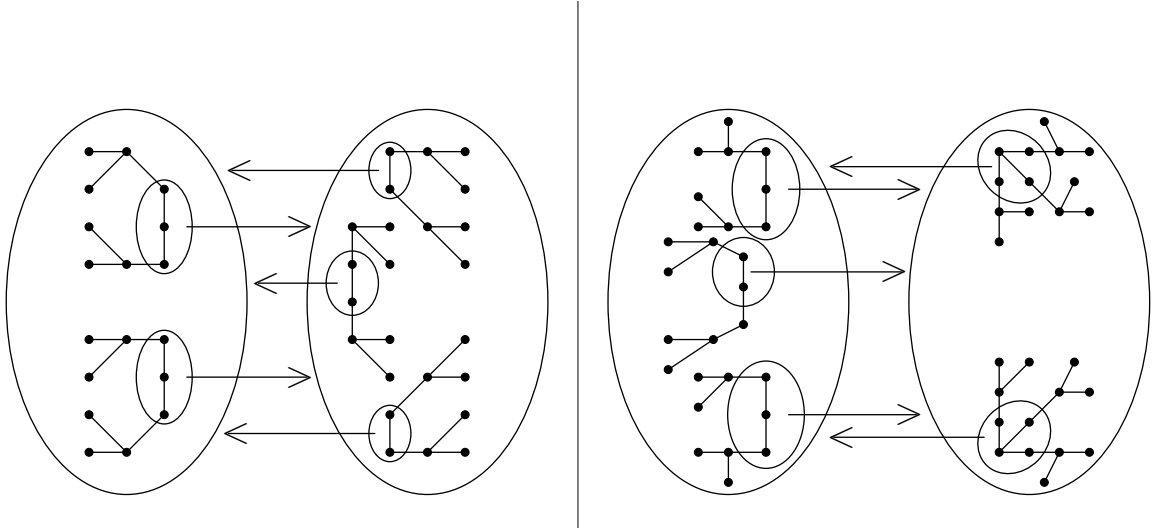
\begin{figure}
\centering
\begin{tikzpicture}[line cap=round,line join=round,x=1cm,y=1cm]
\clip(-11.291351594558828,-3) rectangle (5,4);
\draw [rotate around={90:(-9,0)},line width=0.5pt] (-9,0) ellipse (2.5615528128088116cm and 1.6004851804402291cm);
\draw [rotate around={90:(-5,0)},line width=0.5pt] (-5,0) ellipse (2.5615528128088245cm and 1.6004851804402371cm);
\draw [rotate around={90:(-1,0)},line width=0.5pt] (-1,0) ellipse (2.5615528128088307cm and 1.6004851804402414cm);
\draw [rotate around={90:(3,0)},line width=0.5pt] (3,0) ellipse (2.56155281280883cm and 1.6004851804402407cm);
\draw [line width=0.5pt] (-9,2)-- (-9.5,2);
\draw [line width=0.2pt] (-3,-20)-- (-3,20);

\draw [line width=0.5pt] (-9,2)-- (-9.5,1.5);
\draw [line width=0.5pt] (-9,2)-- (-8.5,1.5);
\draw [line width=0.5pt] (-8.5,1.5)-- (-8.5,1);
\draw [line width=0.5pt] (-8.5,1)-- (-8.5,0.5);
\draw [line width=0.5pt] (-8.5,0.5)-- (-9,0.5);
\draw [line width=0.5pt] (-9,0.5)-- (-9.5,1);
\draw [line width=0.5pt] (-9,0.5)-- (-9.5,0.5);
\draw [line width=0.5pt] (-9,-2)-- (-9.5,-2);
\draw [line width=0.5pt] (-9,-2)-- (-9.5,-1.5);
\draw [line width=0.5pt] (-9,-2)-- (-8.5,-1.5);
\draw [line width=0.5pt] (-8.5,-1.5)-- (-8.5,-1);
\draw [line width=0.5pt] (-8.5,-1)-- (-8.5,-0.5);
\draw [line width=0.5pt] (-8.5,-0.5)-- (-9,-0.5);
\draw [line width=0.5pt] (-9,-0.5)-- (-9.5,-1);
\draw [line width=0.5pt] (-9,-0.5)-- (-9.5,-0.5);
\draw [rotate around={90:(-8.5,1)},line width=0.5pt] (-8.5,1) ellipse (0.6229787289780168cm and 0.3716214428138747cm);
\draw [rotate around={90:(-8.5,-1)},line width=0.5pt] (-8.5,-1) ellipse (0.6229787289780168cm and 0.3716214428138747cm);
\draw [line width=0.5pt] (-5,2)-- (-4.5,2);
\draw [line width=0.5pt] (-5,2)-- (-4.5,1.5);
\draw [line width=0.5pt] (-5,2)-- (-5.5,2);
\draw [line width=0.5pt] (-5.5,2)-- (-5.5,1.5);
\draw [line width=0.5pt] (-5.5,1.5)-- (-5,1);
\draw [line width=0.5pt] (-5,1)-- (-4.5,1);
\draw [line width=0.5pt] (-5,1)-- (-4.5,0.5);
\draw [line width=0.5pt] (-5.5,1)-- (-6,1);
\draw [line width=0.5pt] (-6,1)-- (-5.5,0.5);
\draw [line width=0.5pt] (-6,1)-- (-6,0.5);
\draw [line width=0.5pt] (-6,0.5)-- (-6,0);
\draw [line width=0.5pt] (-6,0)-- (-6,-0.5);
\draw [line width=0.5pt] (-6,-0.5)-- (-5.5,-0.5);
\draw [line width=0.5pt] (-6,-0.5)-- (-5.5,-1);
\draw [line width=0.5pt] (-5,-2)-- (-4.5,-2);
\draw [line width=0.5pt] (-5,-2)-- (-4.5,-1.5);
\draw [line width=0.5pt] (-5.5,-1.5)-- (-5.5,-2);
\draw [line width=0.5pt] (-5.5,-2)-- (-5,-2);
\draw [line width=0.5pt] (-5.5,-1.5)-- (-5,-1);
\draw [line width=0.5pt] (-5,-1)-- (-4.5,-1);
\draw [line width=0.5pt] (-5,-1)-- (-4.5,-0.5);
\draw [rotate around={90:(-5.5,-1.75)},line width=0.5pt] (-5.5,-1.75) ellipse (0.3748488046335532cm and 0.2793056145787404cm);
\draw [rotate around={90:(-6,0.25)},line width=0.5pt] (-6,0.25) ellipse (0.4280311648918286cm and 0.3474344227601166cm);
\draw [rotate around={90:(-5.5,1.75)},line width=0.5pt] (-5.5,1.75) ellipse (0.3748488046335532cm and 0.2793056145787404cm);

\draw [line width=0.5pt] (-5.65,1.75) -- (-7.6,1.75);
\draw [line width=0.5pt] (-8.2,1) -- (-6.6,1);
\draw [line width=0.5pt] (-6.2,0.25) -- (-7.25,0.25);
\draw [line width=0.5pt] (-8.2,-1) -- (-6.6,-1);
\draw [line width=0.5pt] (-5.7,-1.75) -- (-7.65,-1.75);
\draw [line width=0.5pt] (-1,2)-- (-1,2.4);
\draw [line width=0.5pt] (-1,2)-- (-1.4,2);
\draw [line width=0.5pt] (-1,2)-- (-0.5,2);
\draw [line width=0.5pt] (-0.5,2)-- (-0.5,1.5);
\draw [line width=0.5pt] (-0.5,1.5)-- (-0.5,1);
\draw [line width=0.5pt] (-0.5,1)-- (-1,1);
\draw [line width=0.5pt] (-1,1)-- (-1.4,1);
\draw [line width=0.5pt] (-1,1)-- (-1.4,1.4);
\draw [line width=0.5pt] (-1.8,0.8)-- (-1.2,0.8);
\draw [line width=0.5pt] (-1.2,0.8)-- (-0.8,0.6);
\draw [line width=0.5pt] (-0.8,0.6)-- (-0.8,0.2);

\draw [line width=0.5pt] (-0.8,0.2)-- (-0.8,0.2-0.5);
\draw [line width=0.5pt] (-0.8,0.2-0.5)-- (-1.2,0-0.5);
\draw [line width=0.5pt] (-1.2,0-0.5)-- (-1.8,-0.4-0.5);
\draw [line width=0.5pt] (-1.2,0.8)-- (-1.8,0.4);
\draw [line width=0.5pt] (-1.2,0-0.5)-- (-1.8,0-0.5);

\draw [line width=0.5pt] (-1,-2)-- (-1.4,-2);
\draw [line width=0.5pt] (-1,-2)-- (-1,-2.4);
\draw [line width=0.5pt] (-1,-2)-- (-0.5,-2);
\draw [line width=0.5pt] (-0.5,-2)-- (-0.5,-1.5);
\draw [line width=0.5pt] (-0.5,-1.5)-- (-0.5,-1);
\draw [line width=0.5pt] (-0.5,-1)-- (-1,-1);
\draw [line width=0.5pt] (-1,-1)-- (-1.4,-1.4);
\draw [line width=0.5pt] (-1,-1)-- (-1.4,-1);
\draw [rotate around={90:(-0.5,-1.5)},line width=0.5pt] (-0.5,-1.5) ellipse (0.7138831278146033cm and 0.5095381439876332cm);
\draw [rotate around={90:(-0.8,0.4)},line width=0.5pt] (-0.8,0.4) ellipse (0.4576491222541475cm and 0.4116342054542985cm);
\draw [rotate around={90:(-0.5,1.5)},line width=0.5pt] (-0.5,1.5) ellipse (0.6720153254455281cm and 0.4490040062556888cm);
\draw [line width=0.5pt] (3,2)-- (3.4,2);
\draw [line width=0.5pt] (3.4,2)-- (3.2,2.4);
\draw [line width=0.5pt] (3.4,2)-- (3.8,2);
\draw [line width=0.5pt] (3,2)-- (2.6,2);
\draw [line width=0.5pt] (2.6,2)-- (3,1.6);
\draw [line width=0.5pt] (2.6,2)-- (2.6,1.6);
\draw [line width=0.5pt] (3,1.6)-- (3.4,1.2);
\draw [line width=0.5pt] (3.4,1.2)-- (3.6,1.6);
\draw [line width=0.5pt] (3.4,1.2)-- (3.8,1.2);
\draw [line width=0.5pt] (2.6,1.6)-- (2.6,1.2);
\draw [line width=0.5pt] (2.6,1.2)-- (2.6,0.8);
\draw [line width=0.5pt] (2.6,1.2)-- (3,1.2);
\draw [rotate around={-45:(2.8,1.8)},line width=0.5pt] (2.8,1.8) ellipse (0.5236067977499756cm and 0.44064053223686206cm);
\draw [line width=0.5pt] (3,-2)-- (3.4,-2);
\draw [line width=0.5pt] (3.4,-2)-- (3.2,-2.4);
\draw [line width=0.5pt] (3.4,-2)-- (3.8,-2);
\draw [line width=0.5pt] (3,-2)-- (2.6,-2);
\draw [line width=0.5pt] (2.6,-2)-- (2.6,-1.6);
\draw [line width=0.5pt] (2.6,-2)-- (3,-1.6);
\draw [line width=0.5pt] (2.6,-1.6)-- (2.6,-1.2);
\draw [line width=0.5pt] (3,-1.6)-- (3.4,-1.2);
\draw [line width=0.5pt] (3.4,-1.2)-- (3.8,-1.2);
\draw [line width=0.5pt] (3.4,-1.2)-- (3.6,-0.8);
\draw [line width=0.5pt] (2.6,-1.2)-- (2.6,-0.8);
\draw [line width=0.5pt] (2.6,-1.2)-- (3,-0.8);
\draw [rotate around={45:(2.8,-1.8)},line width=0.5pt] (2.8,-1.8) ellipse (0.5236067977499756cm and 0.44064053223686206cm);
\draw [line width=0.5pt] (2.4,-1.8) -- (0.4,-1.8);
\draw [line width=0.5pt] (-0.2,-1.5) -- (1.5,-1.5);
\draw [line width=0.5pt] (-0.7,0.4) -- (1.3,0.4);
\draw [line width=0.5pt] (-0.2,1.5) -- (1.5,1.5);
\draw [line width=0.5pt] (2.5,1.8) -- (0.4,1.8);
\begin{scriptsize}

\draw [fill=black] (0.5,-1.8)
 node {\Large{$<$}};
\draw [fill=black] (1.4,-1.5) node {\Large{$>$}};
\draw [fill=black] (1.2,0.4) node {\Large{$>$}};
\draw [fill=black] (1.4,1.5) node {\Large{$>$}};
\draw [fill=black] (0.5,1.8) node {\Large{$<$}};

\draw [fill=black] (-7.5,1.75) node {\Large{$<$}};
\draw [fill=black] (-6.7,1) node {\Large{$>$}};
\draw [fill=black] (-7.15,0.25) node {\Large{$<$}};
\draw [fill=black] (-6.7,-1) node {\Large{$>$}};
\draw [fill=black] (-7.55,-1.75) node {\Large{$<$}};

\draw [fill=black] (-9,2) circle (1.5pt);
\draw [fill=black] (-9,-2) circle (1.5pt);
\draw [fill=black] (-5,2) circle (1.5pt);
\draw [fill=black] (-5,-2) circle (1.5pt);
\draw [fill=black] (-1,2) circle (1.5pt);
\draw [fill=black] (-1,-2) circle (1.5pt);
\draw [fill=black] (3,2) circle (1.5pt);
\draw [fill=black] (3,-2) circle (1.5pt);
\draw [fill=black] (-9.5,2) circle (1.5pt);
\draw [fill=black] (-9.5,1.5) circle (1.5pt);
\draw [fill=black] (-8.5,1.5) circle (1.5pt);
\draw [fill=black] (-8.5,1) circle (1.5pt);
\draw [fill=black] (-8.5,0.5) circle (1.5pt);
\draw [fill=black] (-9,0.5) circle (1.5pt);
\draw [fill=black] (-9.5,1) circle (1.5pt);
\draw [fill=black] (-9.5,0.5) circle (1.5pt);
\draw [fill=black] (-9.5,-2) circle (1.5pt);
\draw [fill=black] (-9.5,-1.5) circle (1.5pt);
\draw [fill=black] (-8.5,-1.5) circle (1.5pt);
\draw [fill=black] (-8.5,-1) circle (1.5pt);
\draw [fill=black] (-8.5,-0.5) circle (1.5pt);
\draw [fill=black] (-9,-0.5) circle (1.5pt);
\draw [fill=black] (-9.5,-1) circle (1.5pt);
\draw [fill=black] (-9.5,-0.5) circle (1.5pt);
\draw [fill=black] (-4.5,2) circle (1.5pt);
\draw [fill=black] (-4.5,1.5) circle (1.5pt);
\draw [fill=black] (-5.5,2) circle (1.5pt);
\draw [fill=black] (-5.5,1.5) circle (1.5pt);
\draw [fill=black] (-5,1) circle (1.5pt);
\draw [fill=black] (-4.5,1) circle (1.5pt);
\draw [fill=black] (-4.5,0.5) circle (1.5pt);
\draw [fill=black] (-5.5,1) circle (1.5pt);
\draw [fill=black] (-6,1) circle (1.5pt);
\draw [fill=black] (-5.5,0.5) circle (1.5pt);
\draw [fill=black] (-6,0.5) circle (1.5pt);
\draw [fill=black] (-6,0) circle (1.5pt);
\draw [fill=black] (-6,-0.5) circle (1.5pt);
\draw [fill=black] (-5.5,-0.5) circle (1.5pt);
\draw [fill=black] (-5.5,-1) circle (1.5pt);
\draw [fill=black] (-4.5,-2) circle (1.5pt);
\draw [fill=black] (-4.5,-1.5) circle (1.5pt);
\draw [fill=black] (-5.5,-1.5) circle (1.5pt);
\draw [fill=black] (-5.5,-2) circle (1.5pt);
\draw [fill=black] (-5,-1) circle (1.5pt);
\draw [fill=black] (-4.5,-1) circle (1.5pt);
\draw [fill=black] (-4.5,-0.5) circle (1.5pt);
\draw [fill=black] (-1,2.4) circle (1.5pt);
\draw [fill=black] (-1.4,2) circle (1.5pt);
\draw [fill=black] (-0.5,2) circle (1.5pt);
\draw [fill=black] (-0.5,1.5) circle (1.5pt);
\draw [fill=black] (-0.5,1) circle (1.5pt);
\draw [fill=black] (-1,1) circle (1.5pt);
\draw [fill=black] (-1.4,1) circle (1.5pt);
\draw [fill=black] (-1.4,1.4) circle (1.5pt);
\draw [fill=black] (-1.8,0.8) circle (1.5pt);
\draw [fill=black] (-1.2,0.8) circle (1.5pt);
\draw [fill=black] (-0.8,0.6) circle (1.5pt);

\draw [fill=black] (-0.8,0.2) circle (1.5pt);
\draw [fill=black] (-0.8,0.2-0.5) circle (1.5pt);
\draw [fill=black] (-1.2,0-0.5) circle (1.5pt);
\draw [fill=black] (-1.8,-0.4-0.5) circle (1.5pt);
\draw [fill=black] (-1.8,0.4) circle (1.5pt);
\draw [fill=black] (-1.8,0-0.5) circle (1.5pt);

\draw [fill=black] (-1.4,-2) circle (1.5pt);
\draw [fill=black] (-1,-2.4) circle (1.5pt);
\draw [fill=black] (-0.5,-2) circle (1.5pt);
\draw [fill=black] (-0.5,-1.5) circle (1.5pt);
\draw [fill=black] (-0.5,-1) circle (1.5pt);
\draw [fill=black] (-1,-1) circle (1.5pt);
\draw [fill=black] (-1.4,-1.4) circle (1.5pt);
\draw [fill=black] (-1.4,-1) circle (1.5pt);
\draw [fill=black] (3.4,2) circle (1.5pt);
\draw [fill=black] (3.2,2.4) circle (1.5pt);
\draw [fill=black] (3.8,2) circle (1.5pt);
\draw [fill=black] (2.6,2) circle (1.5pt);
\draw [fill=black] (3,1.6) circle (1.5pt);
\draw [fill=black] (2.6,1.6) circle (1.5pt);
\draw [fill=black] (3.4,1.2) circle (1.5pt);
\draw [fill=black] (3.6,1.6) circle (1.5pt);
\draw [fill=black] (3.8,1.2) circle (1.5pt);
\draw [fill=black] (2.6,1.2) circle (1.5pt);
\draw [fill=black] (2.6,0.8) circle (1.5pt);
\draw [fill=black] (3,1.2) circle (1.5pt);
\draw [fill=black] (3.4,-2) circle (1.5pt);
\draw [fill=black] (3.2,-2.4) circle (1.5pt);
\draw [fill=black] (3.8,-2) circle (1.5pt);
\draw [fill=black] (2.6,-2) circle (1.5pt);
\draw [fill=black] (2.6,-1.6) circle (1.5pt);
\draw [fill=black] (3,-1.6) circle (1.5pt);
\draw [fill=black] (2.6,-1.2) circle (1.5pt);
\draw [fill=black] (3.4,-1.2) circle (1.5pt);
\draw [fill=black] (3.8,-1.2) circle (1.5pt);
\draw [fill=black] (3.6,-0.8) circle (1.5pt);
\draw [fill=black] (2.6,-0.8) circle (1.5pt);
\draw [fill=black] (3,-0.8) circle (1.5pt);
\end{scriptsize}
\end{tikzpicture}
\caption{To the left: the first case from the proof of Observation~\ref{choice 1}. To the right: the first case from the proof of Observation~\ref{choice 2}. Only the edges contained in the two parts of the bisection are given. Edges between the two parts are not depicted, however, in both cases there is no edge between the two sets of vertices exchanged between $U_1$ and $U_2$.} 
\label{fig 7}
\end{figure}

\begin{observation}\label{choice 2}
For both $i = 1$ and $i = 2$, at most one of the following happens:
\begin{enumerate}
    \item The sum of the weights of the edges 
    \begin{equation*}
    \{e\in E(G_{3,i})\hspace{0.2em}|\hspace{0.2em} p(e)\ge 3\}
    \end{equation*}
    is more than seven.
    \item $G^+_{3, 3-i}$ contains at least $21$ vertices of degree three.
\end{enumerate}
\end{observation}
\begin{proof}
Suppose that each of the two events above happens for some $i \in \{1,2\}$. Let $(e_j)_{1\le j\le s}$ be a set of at most three edges in $G_{3,i}$ of weights $p(e_1)\ge \dots\ge p(e_s)\ge 3$ and

\begin{equation*}
    \sum_{1\le j\le s}p(e_l) \ge 8.
\end{equation*}

There are three cases.
\begin{itemize}
    \item In the first case, $s = 3$ and there are three connected sets of vertices $(w_i)_{1\le i\le 3}$, $(w_i)_{4\le i\le 6}$ and $(w_i)_{7\le i\le 8}$ in $G[U_i]$, subdividing $e_1, e_2$ and $e_3$ respectively in $G_{3,i}$ (see the right part of Figure~\ref{fig 7}). 
    \item In the second case, $s = 2$ and there are two connected sets of vertices, regrouped either as $(w_i)_{1\le i\le 4}$ and $(w_i)_{5\le i\le 8}$ or as $(w_i)_{1\le i\le 5}$ and $(w_i)_{6\le i\le 8}$, subdividing $e_1$ and $e_2$ respectively in $G_{3,i}$.
    \item In the third case, $s=1$ and there is a connected set of vertices $(w_i)_{1\le i\le 8}$, subdividing $e_1$.
\end{itemize}

Let also $(v_i)_{1\le i\le 21}$ be vertices of degree three in $G^+_{3, 3-i}$. Out of these 21 vertices, at least five are at distance at least three from the set $(w_i)_{1\le i\le 8}$ in $G$. Indeed, every vertex of degree three in $G[U_{3-i}]$ at distance at most two from the set $(w_i)_{1\le i\le 8}$ in $G$ must be adjacent to some of the eight neighbors of the vertices $(w_i)_{1\le i\le 8}$ in $U_{3-i}$. Moreover, out of this subset of five vertices there are two vertices $v'_1$ and $v'_2$, which are not adjacent in $G^+_{3,3-i}$.\par

Let $u_1, u_2, u_3$ and $u_4, u_5, u_6$ be the vertices at distance one to $v'_1$ and $v'_2$  in $G[U_{3-i}]$, respectively. Exchanging $(w_i)_{1\le i\le 8}$ and $(u_i)_{1\le i\le 6}\cup \{v'_1, v'_2\}$ between $U_1$ and $U_2$ would lead to an improvement of $(U_1, U_2)$ -- contradiction with Corollary~\ref{new cut cor}, which proves the observation.
\end{proof}

If for $i = 1$ or for $i = 2$, the first point from the statement of Observation~\ref{choice 2} holds, then $G^+_{3, 3-i}$ must contain predominantly chains and cycles with edges of weight one. We deal with this case in Section~\ref{section 3}. There, we show that the proportion of graphs possessing a bisection of this type of size between $0.1n$ and $0.1069n$ tends to zero as $n$ tends to infinity.\par

From now on we concentrate on the setting, in which the sum of the weights of the edges $\{e\in E(G_{3,i})\hspace{0.2em}|\hspace{0.2em} p(e)\ge 3\}$ is at most seven for both $i = 1$ and $i = 2$. In particular, there are at most two edges of weight at least three in both $G^{+}_{3,1}$ and $G^{+}_{3,2}$.

For both $i=1$ and $i=2$, recall the graph $G^{\le 2}_{3,i}$ from Definition~\ref{defn:graphs1}. We call a vertex in $G^{\le 2}_{3,i}$ \textit{critical} if it is either of degree three in $G^{\le 2}_{3,i}$, or if it is incident to an edge of weight two. For example, in Figure~\ref{fig 6} the graph $G^{\le 2}_{3,i}$ is obtained from $G^+_{3,i}$ by deleting the leaf in $G^+_{3,i}$ incident to the edge $e_3$, and the critical vertices in $G^{\le 2}_{3,i}$ are the endvertices of the edges $e_1$ and $e_6$.

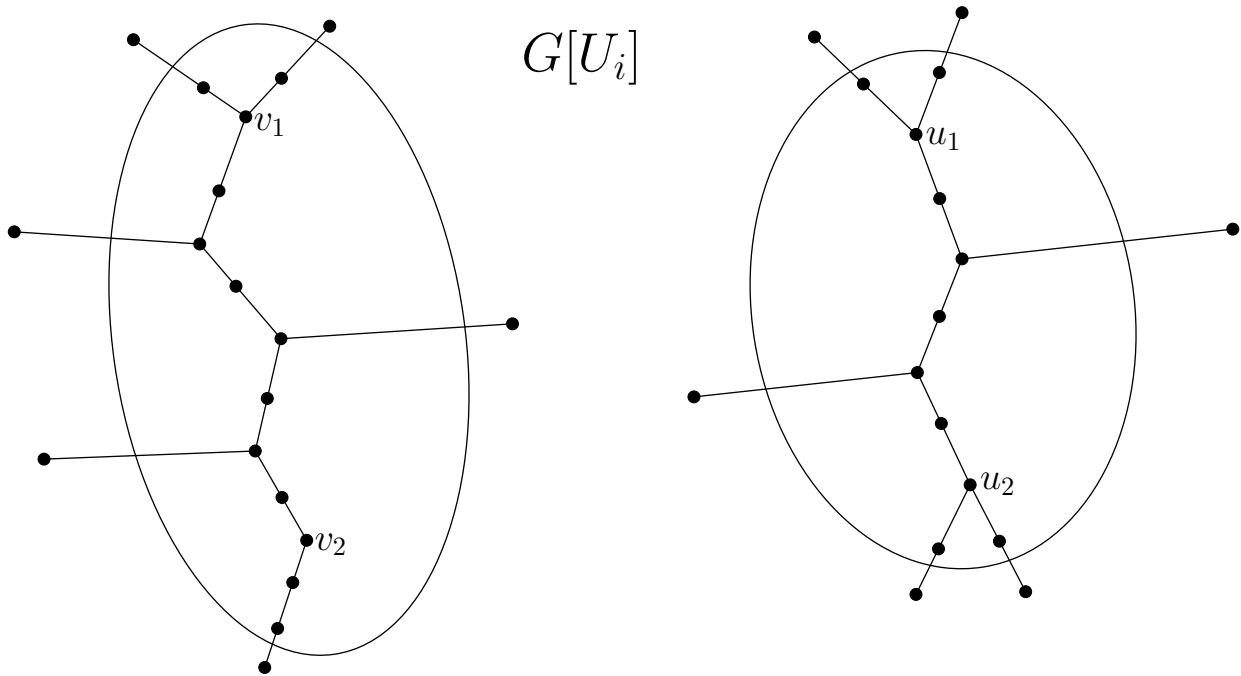
\begin{figure}
\centering
\begin{tikzpicture}[scale = 0.5, line cap=round,line join=round,x=1cm,y=1cm]
\clip(-10,-6.04) rectangle (12.8,6.04);
\draw [line width=0.5pt] (-7.12,4.26)-- (-5.46,3.12);
\draw [line width=0.5pt] (-5.46,3.12)-- (-4.22,4.46);
\draw [line width=0.5pt] (-5.46,3.12)-- (-6.14,1.24);
\draw [line width=0.5pt] (-6.14,1.24)-- (-4.94,-0.16);
\draw [line width=0.5pt] (-4.94,-0.16)-- (-5.32,-1.82);
\draw [line width=0.5pt] (-5.32,-1.82)-- (-4.56,-3.14);
\draw [line width=0.5pt] (-4.56,-3.14)-- (-5.18,-5.02);

\draw [line width=0.5pt] (-8.44,-1.94)-- (-5.32,-1.82);
\draw [line width=0.5pt] (-4.94,-0.16)-- (-1.52,0.06);
\draw [line width=0.5pt] (-6.14,1.24)-- (-8.88,1.42);
\draw [line width=0.5pt] (4.44,2.86)-- (5.12,1.02);
\draw [line width=0.5pt] (5.12,1.02)-- (4.46,-0.66);
\draw [line width=0.5pt] (4.46,-0.66)-- (5.24,-2.32);
\draw [line width=0.5pt] (5.24,-2.32)-- (4.44,-3.94);
\draw [line width=0.5pt] (4.44,2.86)-- (5.12,4.66);
\draw [rotate around={-81.81864919840815:(-5.01,-0.2)},line width=0.5pt] (-5.01,-0.01) ellipse (4.7cm and 2.6cm);
\draw [line width=0.5pt] (4.44,2.86)-- (2.94,4.3);
\draw [line width=0.5pt] (5.24,-2.32)-- (6.06,-3.9);
\draw [rotate around={-81.22059388883224:(4.84,0.27)},line width=0.5pt] (4.84,0.27) ellipse (3.8510502233186634cm and 2.8217880541462947cm);
\draw [line width=0.5pt] (4.46,-0.66)-- (1.16,-1.02);
\draw [line width=0.5pt] (5.12,1.02)-- (9.12,1.46);
\begin{scriptsize}
\draw [fill=black] (-7.12,4.26) circle (2.5pt);
\draw [fill=black] (-5.46,3.12) circle (2.5pt);
\draw [fill=black] (-4.22,4.46) circle (2.5pt);
\draw [fill=black] (-6.14,1.24) circle (2.5pt);
\draw [fill=black] (-4.94,-0.16) circle (2.5pt);
\draw [fill=black] (-5.32,-1.82) circle (2.5pt);
\draw [fill=black] (-4.56,-3.14) circle (2.5pt);
\draw [fill=black] (-5.18,-5.02) circle (2.5pt);
\draw [fill=black] (-4.9898295396549965,-4.443354087986118) circle (2.5pt);
\draw [fill=black] (-4.765991630090844,-3.764619781565785) circle (2.5pt);
\draw [fill=black] (-4.924472792505481,-2.5069683077536373) circle (2.5pt);
\draw [fill=black] (-5.1420863936262915,-1.042798456367484) circle (2.5pt);
\draw [fill=black] (-5.604941176470588,0.615764705882352) circle (2.5pt);
\draw [fill=black] (-5.855804643714972,2.0257165732586073) circle (2.5pt);
\draw [fill=black] (-4.932482899315971,3.690058802352096) circle (2.5pt);
\draw [fill=black] (-6.085815742750049,3.5497770763464196) circle (2.5pt);

\draw [fill=black] (-8.44,-1.94) circle (2.5pt);
\draw [fill=black] (-1.52,0.06) circle (2.5pt);
\draw [fill=black] (-8.88,1.42) circle (2.5pt);
\draw [fill=black] (4.44,2.86) circle (2.5pt);
\draw [fill=black] (5.12,1.02) circle (2.5pt);
\draw [fill=black] (4.46,-0.66) circle (2.5pt);
\draw [fill=black] (5.24,-2.32) circle (2.5pt);
\draw [fill=black] (4.44,-3.94) circle (2.5pt);
\draw [fill=black] (5.12,4.66) circle (2.5pt);
\draw [fill=black] (4.785730337078652,3.775168539325842) circle (2.5pt);
\draw [fill=black] (4.790602910602911,1.9113097713097715) circle (2.5pt);
\draw [fill=black] (4.785988950276242,0.16979005524861887) circle (2.5pt);
\draw [fill=black] (4.814014268727705,-1.41341498216409) circle (2.5pt);
\draw [fill=black] (4.772508271045216,-3.26667075113344) circle (2.5pt);
\draw [fill=black] (2.94,4.3) circle (2.5pt);
\draw [fill=black] (6.06,-3.9) circle (2.5pt);
\draw [fill=black] (5.672565008836153,-3.153478919464782) circle (2.5pt);
\draw [fill=black] (3.6639800166527894,3.6049791840133225) circle (2.5pt);
\draw [fill=black] (1.16,-1.02) circle (2.5pt);
\draw [fill=black] (9.12,1.46) circle (2.5pt);
\draw [fill=black] (-0.5,4) node {\huge{$G[U_i]$}};
\draw [fill=black] (-4.85,2.9) node {\Large{$v_1$}};
\draw [fill=black] (-4,-3.2) node {\Large{$v_2$}};
\draw [fill=black] (5,2.8) node {\Large{$u_1$}};
\draw [fill=black] (5.8,-2.3) node {\Large{$u_2$}};

\end{scriptsize}
\end{tikzpicture}
\caption{The set $S$ from the proof of Observation~\ref{stars and ds_ob} is the union of the vertices in the two encircled regions.}
\label{fig 8}
\end{figure}

\begin{observation}\label{stars and ds_ob}
For both $i=1$ and $i=2$ and for every $\ell\ge 2$, at most one of the following happens:
\begin{enumerate}
    \item In $G^{\le 2}_{3,i}$, there are two (not necessarily disjoint) pairs of critical vertices, $(v_1, v_2)$ and $(u_1, u_2)$, connected by two paths $p_1$ and $p_2$ satisfying the following conditions:
    \begin{itemize}
        \item The sum of the lengths of $p_1$ and $p_2$ is at most $\ell$.
        \item Both paths contain only edges of weight one in $G^{\le 2}_{3,i}$.
    \end{itemize}
    \item There are at least $3\ell + 27$ edges in $G^{\le 2}_{3, 3-i}$ of weight two. 
\end{enumerate}
\end{observation}
\begin{proof}
We argue by contradiction. Suppose that each of the above two events happens for some $i \in \{1,2\}$. By choosing the two paths in $G^{\le 2}_{3,i}$ satisfying the above conditions and with the smallest sum of lengths, one may assume that $p_1$ and $p_2$ intersect in at most one vertex. Moreover, if $p_1$ and $p_2$ have a common vertex, it must be an endvertex for each of them.

Let $S'$ be the set of vertices in $G[U_i]$ that subdivide the edges in $G^{\le 2}_{3,i}$, which are incident to some of the vertices $v_1, v_2, u_1, u_2$ from the first statement. Define $S\subseteq V(G[U_i])$ as the union of $S'$ and the set of vertices contained in the subdivisions of $p_1$ and $p_2$ in $G[U_i]$. See Figure~\ref{fig 8}. Of course, the subdivisions of the two paths contain both vertices of degrees two and three in $G[U_i]$. Then $S$ is an $(i, \ell')$-winning set of size $|S|\le 2\ell + 18$, where $\ell'\ge 2$. Indeed, the subdivisions of the paths $p_1$ and $p_2$ contain at most $2\ell + 2$ vertices and every endvertex is incident in $G^{\le 2}_{3,i}$ to at most two edges of weight at most two. Then, among the $3\ell + 27$ edges of weight two in $G^{\le 2}_{3, 3-i}$ there are $\ell + 9$, for which the pairs of vertices, which subdivide these edges in $G[U_{3-i}]$, contain no vertex incident to a vertex in $S$. Then, an improvement of the bisection $(U_1, U_2)$ is given by exchanging $S$ in $G[U_i]$ with 
\begin{itemize}
    \item $|S|/2$ of the above pairs of vertices in $G[U_{3-i}]$, if $|S|$ is even.
    \item $(|S|-1)/2$ of the above pairs of vertices in $G[U_{3-i}]$ together with some additional vertex of degree two in $G[U_{3-i}]$, not connected to the set $S$ by an edge in $G$, if $|S|$ is odd.
\end{itemize}
This is a contradiction, which proves the observation.
\end{proof}

\begin{observation}\label{cycles}
For any two cycles $c_1$ and $c_2$ in a graph $H$ containing a common edge, there are cycles $c'_1$ and $c'_2$ contained in $c_1\cup c_2$, for which $c'_1\cap c'_2$ is a path.
\end{observation}
\begin{proof}
If some of the cycles $c_1$ and $c_2$ has length two, the claim is trivial. Otherwise, let $p$ be a shortest path with endvertices $u,v\in c_1\cap c_2$ contained in $c_2$ and sharing no common edges with $c_1$. Then, $c_1\cup p$ consists of three disjoint paths $p, p_1, p_2$ between $u$ and $v$. Choosing $c'_1 = p\cup p_1$ and $c'_2 = p\cup p_2$ finishes the proof. 
\end{proof}

\begin{observation}\label{ob 2 same}
For both $i=1$ and $i=2$ and for every $\ell \ge 2$, at most one of the following happens:
\begin{enumerate}
    \item There are at least two cycles $c_1$ and $c_2$ in $G^{\le 2}_{3,i}$, whose union contains at least two vertices of degree three in $G^{\le 2}_{3,i}$, and the sum of whose lengths is at most $\ell$.
    \item There are at least $3\ell + 27$ edges in $G^{\le 2}_{3, 3-i}$ of weight two. 
\end{enumerate}
\end{observation}
\begin{proof}
Suppose that each of the above two events happens for some $i \in \{1,2\}$ and some $\ell\ge 2$. If the two cycles have a common edge, by Observation~\ref{cycles} one can find two cycles $c'_1$ and $c'_2$ in $G^{\le 2}_{3,i}$ whose intersection is a path $p$ with endvertices $u$ and $v$, see the left part of Figure~\ref{fig 9}. Then, one may find without difficulty two paths (possibly some of them of length zero, for example when all edges in $c'_1\cup c'_2$ have weight two) between critical vertices in $G^{\le 2}_{3,i}$ without common interior vertices and containing only edges of weight one, one starting from $v$ and contained in $c'_1\setminus p$ and one starting from $u$ and contained in $c'_2\setminus p$. We obtain a contradiction by Observation~\ref{stars and ds_ob}.

If the cycles are edge-disjoint, by Observation~\ref{trivial 3} they are vertex-disjoint as well. If one cycle, say $c_1$, contains at least two vertices of degree three, we directly apply Observation~\ref{stars and ds_ob} (again, paths of length zero may occur) for the subdivision of $c_1$ in $G^{\le 2}_{3,i}$. If both cycles contain exactly one vertex of degree three and one of them, say $c_1$, contains an edge of weight two in $G^{\le 2}_{3,i}$, then, again, we directly apply Observation~\ref{stars and ds_ob} for the subdivision of $c_1$. It remains the case when $c_1$ and $c_2$ are disjoint and each contains exactly one vertex of degree three and no edges of weight two. Then, the set $S$ of vertices in $G[U_i]$ contained in the subdivisions of the two cycles is $2$-winning and has size $2\ell + 2$. Thus, there are $\ell + 1$ pairs of vertices of degree two in $G^{\le 2}_{3, 3-i}$, none of which is adjacent to $S$. Exchanging $S$ with this set of pairs leads to an improvement of $(U_1, U_2)$ in $G$ -- contradiction.
\end{proof}

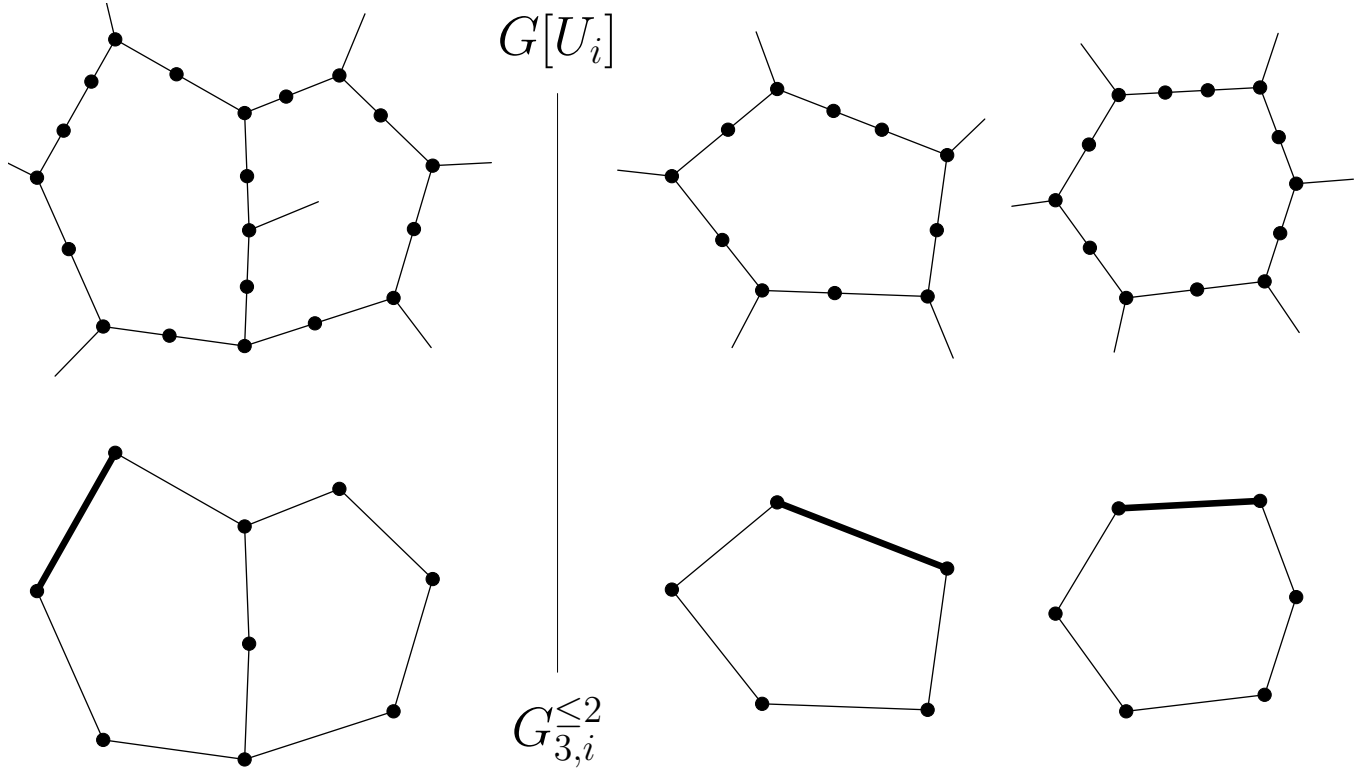
\begin{figure}
\centering
\begin{tikzpicture}[line cap=round,line join=round,x=1cm,y=1cm]
\clip(-7.5,-7.4) rectangle (15.4,3.2);
\draw [line width=0.5pt] (-4.36,1.74)-- (-6.08,2.72);
\draw [line width=0.5pt] (-6.08,2.72)-- (-7.12,0.88);
\draw [line width=0.5pt] (-7.12,0.88)-- (-6.24,-1.1);
\draw [line width=0.5pt] (-6.24,-1.1)-- (-4.36,-1.36);
\draw [line width=0.5pt] (-4.36,-1.36)-- (-4.3,0.18);
\draw [line width=0.5pt] (-4.3,0.18)-- (-4.36,1.74);
\draw [line width=0.5pt] (-4.36,1.74)-- (-3.1,2.24);
\draw [line width=0.5pt] (-3.1,2.24)-- (-1.86,1.04);
\draw [line width=0.5pt] (-1.86,1.04)-- (-2.38,-0.72);
\draw [line width=0.5pt] (-2.38,-0.72)-- (-4.36,-1.36);
\draw [line width=0.5pt] (2.72,2.06)-- (1.32,0.9);
\draw [line width=0.5pt] (1.32,0.9)-- (2.52,-0.62);
\draw [line width=0.5pt] (2.52,-0.62)-- (4.72,-0.7);
\draw [line width=0.5pt] (4.72,-0.7)-- (4.98,1.18);
\draw [line width=0.5pt] (4.98,1.18)-- (2.72,2.06);
\draw [line width=0.5pt] (7.26,1.98)-- (6.42,0.58);
\draw [line width=0.5pt] (6.42,0.58)-- (7.36,-0.72);
\draw [line width=0.5pt] (7.36,-0.72)-- (9.2,-0.5);
\draw [line width=0.5pt] (9.2,-0.5)-- (9.62,0.8);
\draw [line width=0.5pt] (9.62,0.8)-- (9.14,2.08);
\draw [line width=0.5pt] (9.14,2.08)-- (7.26,1.98);
\draw [line width=0.5pt] (-4.3,0.18)-- (-3.38,0.56);
\draw [line width=0.5pt] (-3.1,2.24)-- (-2.76,3.06);
\draw [line width=0.5pt] (-1.86,1.04)-- (-1.08,1.08);
\draw [line width=0.5pt] (-2.38,-0.72)-- (-1.88,-1.38);
\draw [line width=0.5pt] (-6.24,-1.1)-- (-6.88,-1.76);
\draw [line width=0.5pt] (-7.12,0.88)-- (-7.96,1.3);
\draw [line width=0.5pt] (-6.08,2.72)-- (-6.28,3.56);
\draw [line width=0.5pt] (2.72,2.06)-- (2.44,2.82);
\draw [line width=0.5pt] (1.32,0.9)-- (0.6,0.98);
\draw [line width=0.5pt] (2.52,-0.62)-- (2.12,-1.38);
\draw [line width=0.5pt] (4.72,-0.7)-- (5.06,-1.52);
\draw [line width=0.5pt] (4.98,1.18)-- (5.48,1.66);
\draw [line width=0.5pt] (6.42,0.58)-- (5.84,0.5);
\draw [line width=0.5pt] (7.36,-0.72)-- (7.2,-1.44);
\draw [line width=0.5pt] (9.2,-0.5)-- (9.66,-1.18);
\draw [line width=0.5pt] (9.62,0.8)-- (10.38,0.86);
\draw [line width=0.5pt] (9.14,2.08)-- (9.38,2.8);
\draw [line width=0.5pt] (7.26,1.98)-- (6.76,2.66);

\draw [line width=0.5pt] (-4.36,1.74-5.5)-- (-6.08,2.72-5.5);
\draw [line width=2.5pt] (-6.08,2.72-5.5)-- (-7.12,0.88-5.5);
\draw [line width=0.5pt] (-7.12,0.88-5.5)-- (-6.24,-1.1-5.5);
\draw [line width=0.5pt] (-6.24,-1.1-5.5)-- (-4.36,-1.36-5.5);
\draw [line width=0.5pt] (-4.36,-1.36-5.5)-- (-4.3,0.18-5.5);
\draw [line width=0.5pt] (-4.3,0.18-5.5)-- (-4.36,1.74-5.5);
\draw [line width=0.5pt] (-4.36,1.74-5.5)-- (-3.1,2.24-5.5);
\draw [line width=0.5pt] (-3.1,2.24-5.5)-- (-1.86,1.04-5.5);
\draw [line width=0.5pt] (-1.86,1.04-5.5)-- (-2.38,-0.72-5.5);
\draw [line width=0.5pt] (-2.38,-0.72-5.5)-- (-4.36,-1.36-5.5);
\draw [line width=0.5pt] (2.72,2.06-5.5)-- (1.32,0.9-5.5);
\draw [line width=0.5pt] (1.32,0.9-5.5)-- (2.52,-0.62-5.5);
\draw [line width=0.5pt] (2.52,-0.62-5.5)-- (4.72,-0.7-5.5);
\draw [line width=0.5pt] (4.72,-0.7-5.5)-- (4.98,1.18-5.5);
\draw [line width=2.5pt] (4.98,1.18-5.5)-- (2.72,2.06-5.5);
\draw [line width=0.5pt] (7.26,1.98-5.5)-- (6.42,0.58-5.5);
\draw [line width=0.5pt] (6.42,0.58-5.5)-- (7.36,-0.72-5.5);
\draw [line width=0.5pt] (7.36,-0.72-5.5)-- (9.2,-0.5-5.5);
\draw [line width=0.5pt] (9.2,-0.5-5.5)-- (9.62,0.8-5.5);
\draw [line width=0.5pt] (9.62,0.8-5.5)-- (9.14,2.08-5.5);
\draw [line width=2.5pt] (9.14,2.08-5.5)-- (7.26,1.98-5.5);

\draw [line width=0.2pt] (-0.2,2)-- (-0.2,-5.7);
\begin{scriptsize}

\draw [fill=black] (-4.36,1.74) circle (2.5pt);
\draw [fill=black] (-6.08,2.72) circle (2.5pt);
\draw [fill=black] (-7.12,0.88) circle (2.5pt);
\draw [fill=black] (-6.24,-1.1) circle (2.5pt);
\draw [fill=black] (-4.36,-1.36) circle (2.5pt);
\draw [fill=black] (-4.3,0.18) circle (2.5pt);
\draw [fill=black] (-3.1,2.24) circle (2.5pt);
\draw [fill=black] (-1.86,1.04) circle (2.5pt);
\draw [fill=black] (-2.38,-0.72) circle (2.5pt);
\draw [fill=black] (-5.265207716647954,2.255757885066857) circle (2.5pt);
\draw [fill=black] (-6.397363896848137,2.158510028653295) circle (2.5pt);
\draw [fill=black] (-6.766131805157594,1.5060744985673349) circle (2.5pt);
\draw [fill=black] (-6.697731958762887,-0.0701030927835049) circle (2.5pt);
\draw [fill=black] (-5.3590857142857145,-1.2218285714285715) circle (2.5pt);
\draw [fill=black] (-3.807618632999564,1.9591989551589033) circle (2.5pt);
\draw [fill=black] (-4.327710487444609,0.9004726735598227) circle (2.5pt);
\draw [fill=black] (-4.329235825801151,-0.5703861955628595) circle (2.5pt);
\draw [fill=black] (-3.425402412497528,-1.0579078505042516) circle (2.5pt);
\draw [fill=black] (-2.1091306413301663,0.1967885985748219) circle (2.5pt);
\draw [fill=black] (-2.550628694250403,1.7083503492745837) circle (2.5pt);
\draw [fill=black] (2.72,2.06) circle (2.5pt);
\draw [fill=black] (1.32,0.9) circle (2.5pt);
\draw [fill=black] (2.52,-0.62) circle (2.5pt);
\draw [fill=black] (4.72,-0.7) circle (2.5pt);
\draw [fill=black] (4.98,1.18) circle (2.5pt);
\draw [fill=black] (7.26,1.98) circle (2.5pt);
\draw [fill=black] (6.42,0.58) circle (2.5pt);
\draw [fill=black] (7.36,-0.72) circle (2.5pt);
\draw [fill=black] (9.2,-0.5) circle (2.5pt);
\draw [fill=black] (9.62,0.8) circle (2.5pt);
\draw [fill=black] (9.14,2.08) circle (2.5pt);
\draw [fill=black] (2.066610551630843,1.51862017135127) circle (2.5pt);
\draw [fill=black] (3.468838390250096,1.7684169099911131) circle (2.5pt);
\draw [fill=black] (4.111519614066268,1.5181693538149044) circle (2.5pt);
\draw [fill=black] (1.9899444885011892,0.05140364789849372) circle (2.5pt);
\draw [fill=black] (3.485043957485894,-0.6550925075449415) circle (2.5pt);
\draw [fill=black] (4.84195771339076,0.18184808144087683) circle (2.5pt);
\draw [fill=black] (6.877116751269035,-0.05218274111675092) circle (2.5pt);
\draw [fill=black] (7.878797091870456,2.0129147389292794) circle (2.5pt);
\draw [fill=black] (8.444163912756114,2.0429874421678784) circle (2.5pt);
\draw [fill=black] (9.38759616945301,1.4197435481253042) circle (2.5pt);
\draw [fill=black] (9.407141713370695,0.14115292233787036) circle (2.5pt);
\draw [fill=black] (8.302439024390244,-0.6073170731707317) circle (2.5pt);
\draw [fill=black] (6.864611464968155,1.321019108280258) circle (2.5pt);

\draw [fill=black] (-4.36,1.74-5.5) circle (2.5pt);
\draw [fill=black] (-6.08,2.72-5.5) circle (2.5pt);
\draw [fill=black] (-7.12,0.88-5.5) circle (2.5pt);
\draw [fill=black] (-6.24,-1.1-5.5) circle (2.5pt);
\draw [fill=black] (-4.36,-1.36-5.5) circle (2.5pt);
\draw [fill=black] (-4.3,0.18-5.5) circle (2.5pt);
\draw [fill=black] (-3.1,2.24-5.5) circle (2.5pt);
\draw [fill=black] (-1.86,1.04-5.5) circle (2.5pt);
\draw [fill=black] (-2.38,-0.72-5.5) circle (2.5pt);

\draw [fill=black] (2.72,2.06-5.5) circle (2.5pt);
\draw [fill=black] (1.32,0.9-5.5) circle (2.5pt);
\draw [fill=black] (2.52,-0.62-5.5) circle (2.5pt);
\draw [fill=black] (4.72,-0.7-5.5) circle (2.5pt);
\draw [fill=black] (4.98,1.18-5.5) circle (2.5pt);
\draw [fill=black] (7.26,1.98-5.5) circle (2.5pt);
\draw [fill=black] (6.42,0.58-5.5) circle (2.5pt);
\draw [fill=black] (7.36,-0.72-5.5) circle (2.5pt);
\draw [fill=black] (9.2,-0.5-5.5) circle (2.5pt);
\draw [fill=black] (9.62,0.8-5.5) circle (2.5pt);
\draw [fill=black] (9.14,2.08-5.5) circle (2.5pt);

\draw [fill=black] (-0.2,2.7) node {\huge{$G[U_i]$}};
\draw [fill=black] (-0.2,-6.5) node {\huge{$G^{\le 2}_{3,i}$}};

\draw [fill=black] (-4.4,2.1) node {\huge{$u$}};
\draw [fill=black] (-4.4,-1.7) node {\huge{$v$}};

\draw [fill=black] (-4.4,2.1-5.5) node {\huge{$u$}};
\draw [fill=black] (-4.4,-1.7-5.5) node {\huge{$v$}};

\end{scriptsize}
\end{tikzpicture}
\caption{The cycles in the proofs of Observation~\ref{ob 2 same} (on the left) and Observation~\ref{ob 1 same} (on the right). The edges of weight two in $G^{\le 2}_{3,i}$ are thickened.}
\label{fig 9}
\end{figure}

\begin{observation}\label{ob 1 same}
For both $i=1$ and $i=2$ and for every $\ell\ge 2$, at most one of the following happens:
\begin{enumerate}
    \item There are two cycles $c_1$ and $c_2$ in $G^{\le 2}_{3,i}$, whose union contains at least two edges of weight two, and the sum of whose lengths is at most $\ell$.
    \item There are at least $3\ell + 27$ edges in $G^{\le 2}_{3, 3-i}$ of weight two. 
\end{enumerate}
\end{observation}
\begin{proof}
We argue by contradiction. If the two cycles have a common vertex, they also have a common edge by Observation~\ref{trivial 3} and therefore they contain at least two vertices of degree three in $G^{\le 2}_{3,i}$. One may directly apply Observation~\ref{ob 2 same} in this case.\par
If the cycles are edge-disjoint and some of them contains at least two edges of weight two, then we directly apply Observation~\ref{stars and ds_ob} (again, paths of length zero may occur). If both of them contain exactly one edge of weight two, then the set $S$ of vertices in $G[U_i]$ contained in the subdivisions of the two cycles is winning and has even size, which is at most $2\ell + 2$, see the right part of Figure~\ref{fig 9}. Thus, there are $\ell + 1$ pairs of vertices of degree two in $G^{\le 2}_{3, 3-i}$, none of which is adjacent to $S$. Exchanging $S$ with this set of pairs leads to an improvement of $(U_1, U_2)$ in $G$ -- contradiction.
\end{proof}

\begin{remark}\label{stars and ds rq}
The same conclusion holds in the case of two cycles containing at least one vertex of degree three in $G^{\le 2}_{3,i}$ and at least one edge of weight two.\qed
\end{remark}

\begin{corollary}\label{short cycles cor}
For both $i=1$ and $i=2$ and for every $\ell\ge 2$, at most one of the following happens:
\begin{enumerate}
    \item There are two cycles $c_1$ and $c_2$ in $G^{\le 2}_{3,i}$, each of length at most $\ell/2$ and each containing either an edge of weight two or a vertex of degree three, or both.
    \item There are at least $3\ell + 27$ edges in $G^{\le 2}_{3, 3-i}$ of weight two. 
\end{enumerate}
\end{corollary}

\begin{corollary}\label{cut cor 1}
Suppose that there are at least $3\ell + 27$ edges in $G^{\le 2}_{3, 3-i}$ of weight two. Then, by deleting at most six edges in $G^{\le 2}_{3,i}$, one may construct a graph $G''_{3,i}$, which does not contain two critical vertices connected by a path of length at most $\ell/2$, which contains only edges of weight one.
\end{corollary}
\begin{proof}
If there is a cycle of length at most $\ell/2$, containing either a a vertex of degree three, or an edge of weight two, then: 
\begin{enumerate}
    \item in the first case, delete one of the edges incident to the vertex of degree three and participating in the cycle.
    \item in the second case, delete the edge of weight two participating in the cycle.
\end{enumerate}

By Corollary~\ref{short cycles cor} one may conclude that after the deletion, in each of the two cases, no cycle of length at most $\ell/2$ contains a vertex of degree three or an edge of weight two.

If after the deletion there is no path of length at most $\ell/2$ between two critical vertices containing only edges of weight one, then we are done. In any other case, let $p$ be a path of minimal length between some pair of critical vertices $v_1, v_2$. By minimality of $p$, $p$ contains only vertices of degree two in $G^{\le 2}_{3,i}$ and only edges of weight one, and has length at most $\ell/2$. Deleting all edges in $G^{\le 2}_{3,i}$ outside $p$ that are incident to $v_1$ and $v_2$, and one edge from $p$, ensures that there remains no pair of critical vertices at distance at most $\ell/2$. Indeed, deleting the edges incident to $v_1$ and $v_2$ outside $p$ disconnects the path $p$ from the rest of $G^{\le 2}_{3,i}$. Deleting further one edge in $p$ means that if there is a path of length at most $\ell/2$ between two critical vertices in the new graph, it would be disjoint from $p$ in $G^{\le 2}_{3,i}$, which would contradict Observation~\ref{stars and ds_ob}. The corollary is proved.
\end{proof} 

Note that Observations~\ref{stars and ds_ob},~\ref{ob 2 same},~\ref{ob 1 same}, Remark~\ref{stars and ds rq} and Corollary~\ref{short cycles cor} all deal with edges of weight two in $G^{\le 2}_{3,i}$, whose subdivisions in either $G[U_1]$ or in $G[U_2]$ play the role of minimal indifferent sets in the bisection $(U_1, U_2)$. These results have natural counterparts for the other minimal indifferent sets in $G[U_1]$ and $G[U_2]$ -- the 3-stars. We continue by presenting these analogous results.

\begin{observation}\label{deg 3 disjoint}
In a graph of maximal degree three that contains at least $m$ vertices of degree three, one may find an independent set $I$ of at least $m/4$ vertices of degree three. 
\end{observation}
\begin{proof}
Such an independent set can be constructed by consecutively adding a vertex of degree three to $I$, which does not yet have a neighbor in $I$. At every step, this decreases the number of vertices that could be added to $I$ by at most four.
\end{proof}

For a graph $H$ and a vertex $v$ in $H$, we define the \textit{(closed) degree two neighborhood} $N^{\le 2}_H[v]$ as the set of vertices that may be attained from $v$ by a path in $H$, containing no vertex of degree more than two except possibly $v$ itself. 

\begin{observation}\label{stars and ds_ob 1}
For both $i=1$ and $i=2$ and for every $\ell\ge 2$, at most one of the following happens:
\begin{enumerate}
    \item There are three pairs of critical vertices in $G^{\le 2}_{3,i}$, $(v^{1}_j, v^{2}_j)_{1\le j\le 3}$, connected by three paths $(p_j)_{1\le j\le 3}$ such that:
    \begin{itemize}
        \item the sum of their lengths is at most $\ell$,
        \item containing only edges of weight one in $G^{\le 2}_{3,i}$.
    \end{itemize}
    \item There is an independent set $I$ of at least $5\ell/2 + 55$ vertices of degree three in $G^{\le 2}_{3, 3-i}$. 
\end{enumerate}
\end{observation}
\begin{proof}
We argue by contradiction. Suppose that each of the above two events happens for some $i \in \{1,2\}$ and some $\ell\ge 2$. Notice that up to choosing the paths $p_1, p_2$ and $p_3$ such that the sum of their lengths is minimal, we may assume that $p_1$, $p_2$ and $p_3$ do not share common interior vertices since otherwise one may always shorten at least one of the three paths. We remark that the subdivision of $p_1\cup p_2\cup p_3$ together with $\bigcup_{1\le j\le 3} \left(N^{\le 2}_{G[U_i]}[v^1_j]\cup N^{\le 2}_{G[U_i]}[v^2_j]\right)$ in $G[U_i]$ forms an $\ell'$-winning set $S_i$ for some $\ell'\ge 3$ -- this follows directly by an elementary case by case analysis of the positions of the endvertices of the paths $p_1, p_2$ and $p_3$: Indeed, together $p_1, p_2$ and $p_3$ may form:
\begin{itemize}
    \item two cycles, which intersect in a common path;
    \item one cycle containing two of the paths and intersecting the third path;
    \item one cycle containing two of the paths that is disjoint from the third path;
    \item a subdivision of a 3-star;
    \item a longer path consisting of $p_1,p_2$ and $p_3$, composed one after the other in some order in a sequential manner;
    \item a path consisting of two of $p_1, p_2$ and $p_3$, composed one after the other in some order in a sequential manner, and a third disjoint path; this is the case from Figure~\ref{fig 10};
    \item three disjoint paths.
\end{itemize}

Now, on the one hand, the number of vertices in $S_i$ is at most $(\ell + 6\cdot 2\cdot 2) + (\ell + 3) = 2\ell + 27$. On the other hand, there are at most $2(\ell + 6\cdot 2\cdot 2) = 2\ell + 48$ vertices in $I$, which may have a neighbor in $G[U_{3-i}]$, adjacent to a vertex in $S_i$. Indeed, every vertex in $G[U_{3-i}]$ with a neighbor in $S_i$ must have degree at most two in $G[U_{3-i}]$. There remain at least $\ell/2 + 7$ vertices in $I$ at distance at least three from $S_i$ in $G$, for which the balls of radius one in $G[U_{3-i}]$ are two by two disjoint. Thus, one may choose $\lfloor |S_i|/4\rfloor$ of these balls together with one, two or three vertices from another ball in $G[U_{3-i}]$ with center in $I$ to form an indifferent, a $1$-losing or a $2$-losing set $S_{3-i}\subseteq U_{3-i}$ with $|S_i| = |S_{3-i}|$ and no edge between $S_i$ and $S_{3-i}$ in $G$. Exchanging $S_i$ and $S_{3-i}$ between $U_i$ and $U_{3-i}$ leads to an improvement of the bisection $(U_1, U_2)$ -- contradiction. The observation is proved.
\end{proof}

\begin{figure}
\centering
\begin{tikzpicture}[line cap=round,line join=round,x=1cm,y=1cm]
\clip(-10,-15) rectangle (10,5);
\draw [line width=0.5pt] (-4.96,3.14)-- (-5.68,1.66);
\draw [line width=0.5pt] (-5.68,1.66)-- (-5.22,0.16);
\draw [line width=0.5pt] (-5.22,0.16)-- (-6.08,-1.64);
\draw [line width=0.5pt] (-6.08,-1.64)-- (-7.72,-2.82);
\draw [line width=0.5pt] (-6.08,-1.64)-- (-4.5,-3.24);
\draw [line width=0.5pt] (-4.96,3.14)-- (-5.84,4.82);
\draw [line width=0.5pt] (-4.96,3.14)-- (-2.94,1.9);
\draw [line width=0.5pt] (-2.94,1.9)-- (-2.64,-0.04);
\draw [line width=0.5pt] (-2.64,-0.04)-- (-0.84,-1.58);
\draw [line width=0.5pt] (-0.84,-1.58)-- (-1.94,-3.3);
\draw [line width=0.5pt] (1.4,3.82)-- (2.66,2.42);
\draw [line width=0.5pt] (2.66,2.42)-- (2.52,0.52);
\draw [line width=0.5pt] (2.52,0.52)-- (3.6,-0.3);
\draw [line width=0.5pt] (3.6,-0.3)-- (2.68,-2.58);
\draw [line width=0.5pt] (3.6,-0.3)-- (1.5,-1.62);
\draw [line width=0.5pt] (1.4,3.82)-- (3.74,4.78);
\draw [line width=0.5pt] (-2.94,1.9)-- (-1.02,3.62);
\draw [line width=0.5pt] (-2.64,-0.04)-- (-0.64,1.22);
\draw [line width=0.5pt] (-5.68,1.66)-- (-7.52,1.22);
\draw [line width=0.5pt] (-5.22,0.16)-- (-3.3,-1.6);
\draw [line width=0.5pt] (2.66,2.42)-- (4.48,2.86);
\draw [line width=0.5pt] (2.52,0.52)-- (1.06,0.1);

\draw [line width=0.5pt] (-4.96,3.14-10)-- (-5.68,1.66-10);
\draw [line width=0.5pt] (-5.68,1.66-10)-- (-5.22,0.16-10);
\draw [line width=0.5pt] (-5.22,0.16-10)-- (-6.08,-1.64-10);
\draw [line width=0.5pt] (-6.08,-1.64-10)-- (-7.72,-2.82-10);
\draw [line width=0.5pt] (-6.08,-1.64-10)-- (-4.5,-3.24-10);
\draw [line width=0.5pt] (-4.96,3.14-10)-- (-5.84,4.82-10);
\draw [line width=0.5pt] (-4.96,3.14-10)-- (-2.94,1.9-10);
\draw [line width=0.5pt] (-2.94,1.9-10)-- (-2.64,-0.04-10);
\draw [line width=0.5pt] (-2.64,-0.04-10)-- (-0.84,-1.58-10);
\draw [line width=2.5pt] (-0.84,-1.58-10)-- (-1.94,-3.3-10);
\draw [line width=0.5pt] (1.4,3.82-10)-- (2.66,2.42-10);
\draw [line width=0.5pt] (2.66,2.42-10)-- (2.52,0.52-10);
\draw [line width=0.5pt] (2.52,0.52-10)-- (3.6,-0.3-10);
\draw [line width=0.5pt] (3.6,-0.3-10)-- (2.68,-2.58-10);
\draw [line width=0.5pt] (3.6,-0.3-10)-- (1.5,-1.62-10);
\draw [line width=2.5pt] (1.4,3.82-10)-- (3.74,4.78-10);

\draw [line width=0.2pt] (-20,-4)-- (7,-4);

\draw [line width=0.5pt] (-2.94,1.9)-- (-1.02,3.62);
\draw [line width=0.5pt] (-2.64,-0.04)-- (-0.64,1.22);
\draw [line width=0.5pt] (-5.68,1.66)-- (-7.52,1.22);
\draw [line width=0.5pt] (-5.22,0.16)-- (-3.3,-1.6);
\draw [line width=0.5pt] (2.66,2.42)-- (4.48,2.86);
\draw [line width=0.5pt] (2.52,0.52)-- (1.06,0.1);

\draw [rotate around={76.81294009217092:(-5.52,0.75)},line width=0.5pt] (-5.52,0.75) ellipse (2.8008554825873224cm and 1.3487369774494524cm);
\draw [rotate around={-48.88289856036367:(-2.9,0.78)},line width=0.5pt] (-2.9,0.78) ellipse (3.308677879051606cm and 1.064964462940171cm);
\draw [rotate around={-61.898648694038215:(2.5,1.76)},line width=0.5pt] (2.5,1.76) ellipse (2.611333839667906cm and 1.1685308820030083cm);
\begin{scriptsize}

\draw [fill=black] (-4.96,3.14) circle (2.5pt);
\draw [fill=black] (-5.68,1.66) circle (2.5pt);
\draw [fill=black] (-5.22,0.16) circle (2.5pt);
\draw [fill=black] (-6.08,-1.64) circle (2.5pt);
\draw [fill=black] (-7.72,-2.82) circle (2.5pt);
\draw [fill=black] (-4.5,-3.24) circle (2.5pt);
\draw [fill=black] (-5.84,4.82) circle (2.5pt);
\draw [fill=black] (-2.94,1.9) circle (2.5pt);
\draw [fill=black] (-2.64,-0.04) circle (2.5pt);
\draw [fill=black] (-0.84,-1.58) circle (2.5pt);
\draw [fill=black] (-1.94,-3.3) circle (2.5pt);
\draw [fill=black] (1.4,3.82) circle (2.5pt);
\draw [fill=black] (2.66,2.42) circle (2.5pt);
\draw [fill=black] (2.52,0.52) circle (2.5pt);
\draw [fill=black] (3.6,-0.3) circle (2.5pt);
\draw [fill=black] (2.68,-2.58) circle (2.5pt);
\draw [fill=black] (1.5,-1.62) circle (2.5pt);
\draw [fill=black] (3.74,4.78) circle (2.5pt);
\draw [fill=black] (-5.389067548417572,3.9591289560699106) circle (2.5pt);
\draw [fill=black] (-3.960918918918919,2.526702702702703) circle (2.5pt);
\draw [fill=black] (-2.775963062604523,0.8392278048425796) circle (2.5pt);
\draw [fill=black] (-1.7495552367288378,-0.8018249641319942) circle (2.5pt);
\draw [fill=black] (-1.326824680932732,-2.3412167738220897) circle (2.5pt);
\draw [fill=black] (-1.646659629594089,-2.841322329910757) circle (2.5pt);
\draw [fill=black] (-5.331729929459187,2.375888478333893) circle (2.5pt);
\draw [fill=black] (-5.4478025477707,0.9028343949044586) circle (2.5pt);
\draw [fill=black] (-5.6205604017975155,-0.6783822363203811) circle (2.5pt);
\draw [fill=black] (-6.997176470588233,-2.2999196556671433) circle (2.5pt);
\draw [fill=black] (-5.268356164383562,-2.4619178082191784) circle (2.5pt);
\draw [fill=black] (3.089883866356979,4.5132856887618376) circle (2.5pt);
\draw [fill=black] (2.3381740217973923,4.204891906378418) circle (2.5pt);
\draw [fill=black] (2.0433963378617834,3.105115180153574) circle (2.5pt);
\draw [fill=black] (2.5832107023411375,1.3778595317725761) circle (2.5pt);
\draw [fill=black] (3.0320780487804884,0.1312) circle (2.5pt);
\draw [fill=black] (2.3854879299095906,-1.0634075869139716) circle (2.5pt);
\draw [fill=black] (3.154014654839954,-1.4052680293096795) circle (2.5pt);

\draw [fill=black] (-4.96,3.14-10) circle (2.5pt);
\draw [fill=black] (-5.68,1.66-10) circle (2.5pt);
\draw [fill=black] (-5.22,0.16-10) circle (2.5pt);
\draw [fill=black] (-6.08,-1.64-10) circle (2.5pt);
\draw [fill=black] (-7.72,-2.82-10) circle (2.5pt);
\draw [fill=black] (-4.5,-3.24-10) circle (2.5pt);
\draw [fill=black] (-5.84,4.82-10) circle (2.5pt);
\draw [fill=black] (-2.94,1.9-10) circle (2.5pt);
\draw [fill=black] (-2.64,-0.04-10) circle (2.5pt);
\draw [fill=black] (-0.84,-1.58-10) circle (2.5pt);
\draw [fill=black] (-1.94,-3.3-10) circle (2.5pt);
\draw [fill=black] (1.4,3.82-10) circle (2.5pt);
\draw [fill=black] (2.66,2.42-10) circle (2.5pt);
\draw [fill=black] (2.52,0.52-10) circle (2.5pt);
\draw [fill=black] (3.6,-0.3-10) circle (2.5pt);
\draw [fill=black] (2.68,-2.58-10) circle (2.5pt);
\draw [fill=black] (1.5,-1.62-10) circle (2.5pt);
\draw [fill=black] (3.74,4.78-10) circle (2.5pt);

\draw [fill=black] (-9,4) node {\huge{$G[U_i]$}};
\draw [fill=black] (-9,4-10) node {\huge{$G^{\le 2}_{3,i}$}};

\draw [fill=black] (-6.4,-11.3) node {\Large{$v^1_1$}};
\draw [fill=black] (-5.9,-9) node {\Large{$p_1$}};
\draw [fill=black] (-4.2,3.5-10) node {\Large{$v^2_1\equiv v^1_2$}};
\draw [fill=black] (-2.4,-9) node {\Large{$p_2$}};
\draw [fill=black] (-0.45,-11.5) node {\Large{$v^2_2$}};
\draw [fill=black] (1,3.85-10) node {\Large{$v^1_3$}};
\draw [fill=black] (2.2,-8.6) node {\Large{$p_3$}};
\draw [fill=black] (4,-10.3) node {\Large{$v^2_3$}};
\end{scriptsize}
\end{tikzpicture}
\caption{The three paths $p_1, p_2$ and $p_3$ from the proof of Observation~\ref{stars and ds_ob 1}. The edges of weight two are thickened. 
\label{fig 10}}
\end{figure}

We directly deduce the following corollary of Observation~\ref{stars and ds_ob 1}:

\begin{corollary}\label{stars and ds cor 1}
For both $i=1$ and $i=2$ and for every $\ell\ge 2$, at most one of the following happens:
\begin{enumerate}
    \item There are three connected components in $G^{\le 2}_{3,i}$, each of them containing a path:
    \begin{itemize}
        \item of length at most $\ell/3$,
        \item containing only edges of weight one in $G^{\le 2}_{3,i}$,
        \item starting and ending with critical vertices in $G^{\le 2}_{3,i}$.
    \end{itemize}
    \item There is an independent set $I$ containing at least $5\ell/2 + 55$ vertices of degree three in $G^{\le 2}_{3, 3-i}$.\qed
\end{enumerate}
\end{corollary}

The next observation is in the same spirit.

\begin{observation}\label{stars and ds ob 3}
For both $i=1$ and $i=2$ and for every $\ell\ge 2$, at most one of the following happens:
\begin{enumerate}
    \item There are three vertex-disjoint cycles in $G^{\le 2}_{3,i}$, each of length at most $\ell/3$ and each containing a critical vertex in $G^{\le 2}_{3,i}$.
    \item There is an independent set $I$ containing at least $5\ell/2 + 55$ vertices of degree three in $G^{\le 2}_{3, 3-i}$.
\end{enumerate}
\end{observation}
\begin{proof}
The proof is analogous to the one of Observation~\ref{stars and ds_ob 1} by choosing $v^1_i = v^2_i$ for every $i\in [3]$.
\end{proof}

\begin{corollary}\label{cor 2.30}
For both $i=1$ and $i=2$ and for every $\ell\ge 2$, at most one of the following happens:
\begin{enumerate}
    \item There are two cycles in $G^{\le 2}_{3,i}$, each of length at most $\ell/3$, which share a common edge.
    \item There is an independent set $I$ containing at least $5\ell/2 + 55$ vertices of degree three in $G^{\le 2}_{3, 3-i}$.
\end{enumerate}
\end{corollary}
\begin{proof}
We argue by contradiction. Suppose that each of the above two events happens for some $i \in \{1,2\}$ and some $\ell\ge 2$. Let $c_1$ and $c_2$ be two cycles as described in the first assertion. By applying Observation~\ref{cycles} one deduces that there are two cycles $c'_1$ and $c'_2$ such that $c'_1\cup c'_2 \subseteq c_1\cup c_2$ and $p = c'_1\cap c'_2$ is a path. Then each of the paths $p, c'_1\setminus p$ and $c'_2\setminus p$ in $G^{\le 2}_{3,i}$ is of length at most $\ell/3$, and they share common endvertices, which are therefore critical in $G^{\le 2}_{3,i}$ (in case the cycles contain edges of weight two, one may shorten the paths so that all three of them contain only edges of weight one. Note that this may possibly lead to paths of length zero.). This is a contradiction by Observation~\ref{stars and ds_ob 1}. The observation is proved.
\end{proof}

\begin{corollary}\label{cut cor 2}
Suppose that there are at least $10\ell + 220$ vertices of degree three in $G^{\le 2}_{3, 3-i}$. Then, by deleting at most $20$ edges in $G^{\le 2}_{3,i}$ one may construct a graph $G''_{3,i}$, which does not contain two critical vertices, connected by a path of length at most $\ell/3$ of edges of weight one.
\end{corollary}
\begin{proof}
We argue by contradiction. By Observation~\ref{deg 3 disjoint} there is an independent set of at least $5\ell/2 +55$ vertices of degree three in $G^{\le 2}_{3, 3-i}$. Then, by Observation~\ref{stars and ds ob 3} and Corollary~\ref{cor 2.30} there are at most two cycles of length at most $\ell/3$ in $G^{\le 2}_{3,i}$, which contain a critical vertex in $G^{\le 2}_{3,i}$, and, if present, they must be (vertex-)disjoint. Moreover, in each of the cycles there are at most:
\begin{itemize}
    \item two edges of weight two and no vertex of degree three in $G^{\le 2}_{3,i}$, or
    \item one edge of weight two and one vertex of degree three in $G^{\le 2}_{3,i}$, or
    \item two vertices of degree three in $G^{\le 2}_{3,i}$.
\end{itemize}

Indeed, if a cycle contains at least, say, two edges of weight two and a vertex of degree three in $G^{\le 2}_{3,i}$, one may find three edge-disjoint paths between critical vertices (some of which may be reduced to a single vertex), containing only edges of weight one in $G^{\le 2}_{3,i}$ -- contradiction with Observation~\ref{stars and ds_ob 1}. All other cases are treated analogously. We conclude that one may delete at most eight edges -- at most four edges to disconnect the cycles containing critical vertices from the rest of $G^{\le 2}_{3,i}$, and at most four more edges to disconnect all paths of positive length between critical vertices in $G^{\le 2}_{3,i}$ containing only edges of weight one. See Figure~\ref{fig 11}.\par

It remains to deal with the paths of length at most $\ell/3$ between critical vertices, which consist of edges of weight one. There are at most two paths $p_1$ and $p_2$ of this type by Observation~\ref{stars and ds_ob 1}. Notice also that the paths $p_1$ and $p_2$ cannot have common interior vertices and may only share common endvertices -- otherwise one would be able to decompose $p_1$ and $p_2$ into at least three paths of length at most $\ell/3$ of edges of weight one, contradicting Observation~\ref{stars and ds_ob 1}. For the same reason all interior vertices of these paths are of degree two in $G^{\le 2}_{3,i}$. Suppose that the paths $p_1$ and $p_2$ have endpoints the critical vertices $(v_1, v_2)$ and $(u_1, u_2)$. Then, by deleting all (at most twelve) edges $G^{\le 2}_{3,i}$, incident to $v_1, v_2, u_1$ and $u_2$ we obtain a graph without critical vertices at distance $\ell/3$. Indeed, if there remains some path of length at most $\ell/3$ between two critical vertices in the obtained graph, containing only edges of weight one, it will be disjoint from the first two and this would contradict Observation~\ref{stars and ds_ob 1}. The observation is proved since we deleted in total at most $8+12=20$ edges in $G^{\le 2}_{3,i}$.
\end{proof}

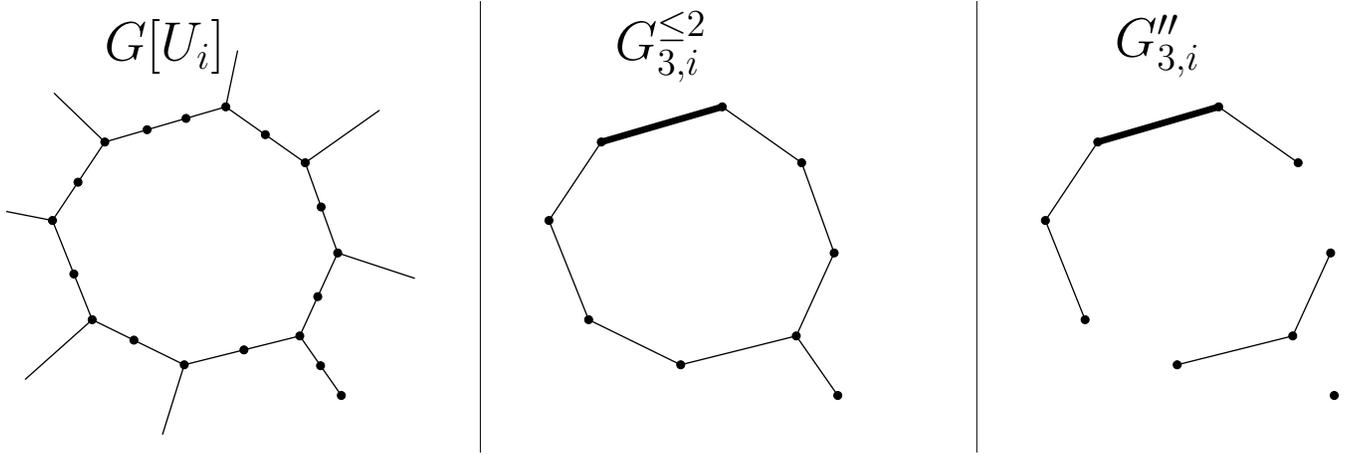
\begin{figure}
\centering
\begin{tikzpicture}[scale = 0.6, line cap=round,line join=round,x=1cm,y=1cm]
\clip(-7.5,-4) rectangle (22,6);

\draw [line width=0.2pt] (3,-20)-- (3,20);
\draw [line width=0.2pt] (14,-20)-- (14,20);

\draw [line width=0.5pt] (-5.32,2.88)-- (-6.48,1.14);
\draw [line width=0.5pt] (-6.48,1.14)-- (-5.6,-1.06);
\draw [line width=0.5pt] (-5.6,-1.06)-- (-3.56,-2.06);
\draw [line width=0.5pt] (-3.56,-2.06)-- (-1,-1.42);
\draw [line width=0.5pt] (-1,-1.42)-- (-0.16,0.42);
\draw [line width=0.5pt] (-0.16,0.42)-- (-0.88,2.42);
\draw [line width=0.5pt] (-0.88,2.42)-- (-2.64,3.66);
\draw [line width=0.5pt] (-2.64,3.66)-- (-5.32,2.88);
\draw [line width=0.5pt] (-5.32,2.88)-- (-6.44,3.96);
\draw [line width=0.5pt] (-2.64,3.66)-- (-2.38,4.9);
\draw [line width=0.5pt] (-0.88,2.42)-- (0.76,3.58);
\draw [line width=0.5pt] (-0.16,0.42)-- (1.54,-0.14);
\draw [line width=0.5pt] (-1,-1.42)-- (-0.08,-2.74);
\draw [line width=0.5pt] (-3.56,-2.06)-- (-4.04,-3.6);
\draw [line width=0.5pt] (-5.6,-1.06)-- (-7.08,-2.38);
\draw [line width=0.5pt] (-6.48,1.14)-- (-8.32,1.5);

\draw [line width=0.5pt] (-5.32+11,2.88)-- (-6.48+11,1.14);
\draw [line width=0.5pt] (-6.48+11,1.14)-- (-5.6+11,-1.06);
\draw [line width=0.5pt] (-5.6+11,-1.06)-- (-3.56+11,-2.06);
\draw [line width=0.5pt] (-3.56+11,-2.06)-- (-1+11,-1.42);
\draw [line width=0.5pt] (-1+11,-1.42)-- (-0.16+11,0.42);
\draw [line width=0.5pt] (-0.16+11,0.42)-- (-0.88+11,2.42);
\draw [line width=0.5pt] (-0.88+11,2.42)-- (-2.64+11,3.66);
\draw [line width=2.5pt] (-2.64+11,3.66)-- (-5.32+11,2.88);

\draw [line width=0.5pt] (-1+11,-1.42)-- (-0.08+11,-2.74);

\draw [line width=0.5pt] (-5.32+22,2.88)-- (-6.48+22,1.14);
\draw [line width=0.5pt] (-6.48+22,1.14)-- (-5.6+22,-1.06);

\draw [line width=0.5pt] (-3.56+22,-2.06)-- (-1+22,-1.42);
\draw [line width=0.5pt] (-1+22,-1.42)-- (-0.16+22,0.42);

\draw [line width=0.5pt] (-0.88+22,2.42)-- (-2.64+22,3.66);
\draw [line width=2.5pt] (-2.64+22,3.66)-- (-5.32+22,2.88);

\begin{scriptsize}
\draw [fill=black] (-5.32,2.88) circle (2.5pt);
\draw [fill=black] (-6.48,1.14) circle (2.5pt);
\draw [fill=black] (-5.6,-1.06) circle (2.5pt);
\draw [fill=black] (-3.56,-2.06) circle (2.5pt);
\draw [fill=black] (-1,-1.42) circle (2.5pt);
\draw [fill=black] (-0.16,0.42) circle (2.5pt);
\draw [fill=black] (-0.88,2.42) circle (2.5pt);
\draw [fill=black] (-2.64,3.66) circle (2.5pt);
\draw [fill=black] (-4.383644298403245,3.152521435539354) circle (2.5pt);
\draw [fill=black] (-3.521041228115213,3.403577553011244) circle (2.5pt);
\draw [fill=black] (-5.913846153846154,1.9892307692307696) circle (2.5pt);
\draw [fill=black] (-6.0055172413793105,-0.046206896551724164) circle (2.5pt);
\draw [fill=black] (-4.676119032858029,-1.5128828270303778) circle (2.5pt);
\draw [fill=black] (-2.2376470588235295,-1.7294117647058824) circle (2.5pt);
\draw [fill=black] (-0.6020101681658193,-0.5482127493156043) circle (2.5pt);
\draw [fill=black] (-0.5268838526912181,1.4391218130311614) circle (2.5pt);
\draw [fill=black] (-1.7639489126682775,3.0427821884708317) circle (2.5pt);
\draw [fill=black] (-0.08,-2.74) circle (2.5pt);
\draw [fill=black] (-0.54,-2.08) circle (2.5pt);

\draw [fill=black] (-5.32+11,2.88) circle (2.5pt);
\draw [fill=black] (-6.48+11,1.14) circle (2.5pt);
\draw [fill=black] (-5.6+11,-1.06) circle (2.5pt);
\draw [fill=black] (-3.56+11,-2.06) circle (2.5pt);
\draw [fill=black] (-1+11,-1.42) circle (2.5pt);
\draw [fill=black] (-0.16+11,0.42) circle (2.5pt);
\draw [fill=black] (-0.88+11,2.42) circle (2.5pt);
\draw [fill=black] (-2.64+11,3.66) circle (2.5pt);
\draw [fill=black] (-0.08+11,-2.74) circle (2.5pt);

\draw [fill=black] (-5.32+22,2.88) circle (2.5pt);
\draw [fill=black] (-6.48+22,1.14) circle (2.5pt);
\draw [fill=black] (-5.6+22,-1.06) circle (2.5pt);
\draw [fill=black] (-3.56+22,-2.06) circle (2.5pt);
\draw [fill=black] (-1+22,-1.42) circle (2.5pt);
\draw [fill=black] (-0.16+22,0.42) circle (2.5pt);
\draw [fill=black] (-0.88+22,2.42) circle (2.5pt);
\draw [fill=black] (-2.64+22,3.66) circle (2.5pt);

\draw [fill=black] (-0.08+22,-2.74) circle (2.5pt);

\draw [fill=black] (-4,5) node {\huge{$G[U_i]$}};
\draw [fill=black] (-4+11,5) node {\huge{$G^{\le 2}_{3,i}$}};
\draw [fill=black] (-4+22,5) node {\huge{$G''_{3,i}$}};

\end{scriptsize}
\end{tikzpicture}
\caption{Deleting edges incident to cycles, containing critical vertices, in the proof of Corollary~\ref{cut cor 2}. Edges of weight two in $G^{\le 2}_{3,i}$ and $G''_{3,i}$ are thickened. In the figure, three edges are sufficient to disconnect the paths of edges of weight one between critical vertices.}
\label{fig 11}
\end{figure}

\begin{lemma}\label{cut lemma}
Suppose that there are at least $13\ell + 273$ critical vertices in $G^{\le 2}_{3,3-i}$. Then, one may delete at most $20$ edges in $G^{\le 2}_{3,i}$ to obtain a graph $G''_{3,i}$, which does not contain a path of length at most $\ell/3$ between two critical vertices in $G^{\le 2}_{3,i}$, which contains only edges of weight one.
\end{lemma}
\begin{proof}
Out of these at least $13\ell + 273$ critical vertices, either at least $6\ell + 54$ are incident to edges of weight two and thus the number of these edges is at least $3\ell + 27$ in $G^{\le 2}_{3,3-i}$, or there are at least $10\ell + 220$ vertices of degree three in $G^{\le 2}_{3,i}$. In the first case, we apply Corollary~\ref{cut cor 1}. In the second case, we apply Corollary~\ref{cut cor 2}. The lemma follows.
\end{proof}

The results in this section suggest the following idea. For both $i=1$ and $i=2$, let 
\begin{equation}\label{def:S}
    S_i = \bigcup_{v \text{ is critical in }G^{\le 2}_{3,i}} N^{\le 2}_{G[U_i]}[v].
\end{equation}

In words, $S_i$ is the union of the set of critical vertices in $G^{\le 2}_{3,i}$ together with the vertices in $G[U_i]$, which subdivide the edges, incident to the critical vertices in $G^{\le 2}_{3,i}$. See Figure~\ref{fig 16}.

\begin{figure}[ht]
\centering
\begin{tikzpicture}[line cap=round,line join=round,x=1cm,y=1cm]
\clip(-8,-3.5) rectangle (12.8,6.5);
\draw [line width=0.8pt] (-5.82,0.54)-- (-3.78,-1.64);
\draw [line width=0.8pt] (-3.78,-1.64)-- (-0.96,0.16);
\draw [line width=0.8pt] (-0.96,0.16)-- (-0.38,2.54);
\draw [line width=0.8pt] (-0.96,0.16)-- (0.96,-1.8);
\draw [line width=0.8pt] (0.96,-1.8)-- (2.88,-0.22);
\draw [line width=0.8pt] (2.96,4.46)-- (4.48,3.38);
\draw [line width=0.8pt] (4.48,3.38)-- (4.54,1.66);
\draw [line width=0.8pt] (4.48,3.38)-- (5.9,4.62);
\draw [line width=0.8pt] (-6.52,2.66)-- (-5.28,3.58);
\draw [line width=0.8pt] (-5.28,3.58)-- (-3.46,3.98);
\draw [line width=0.8pt] (-3.46,3.98)-- (-1.9,4.98);
\draw [line width=0.8pt] (2.88,-0.22)-- (4.16,-1.8);
\draw [line width=0.8pt] (4.16,-1.8)-- (5.7,-0.1);
\draw [line width=0.8pt] (-0.96,0.16) circle (1.8539687160251652cm);
\draw [rotate around={47.82712457816126:(4.93,-0.95)},line width=0.8pt] (4.93,-0.95) ellipse (1.9733594496843034cm and 1.6058479123685223cm);
\draw [line width=0.8pt] (4.48,3.38) circle (1.44346804606129cm);
\draw [rotate around={32.660912721673924:(-2.68,4.48)},line width=0.8pt] (-2.68,4.48) ellipse (2.0460723815231865cm and 1.8242840213168474cm);
\begin{scriptsize}
\draw [fill=white] (-5.82,0.54) circle (3.5pt);
\draw [fill=white] (-3.78,-1.64) circle (3.5pt);
\draw [fill=black] (-0.96,0.16) circle (3.5pt);
\draw [fill=white] (-0.38,2.54) circle (3.5pt);
\draw [fill=white] (0.96,-1.8) circle (3.5pt);
\draw [fill=white] (2.88,-0.22) circle (3.5pt);
\draw [fill=white] (2.96,4.46) circle (3.5pt);
\draw [fill=black] (4.48,3.38) circle (3.5pt);
\draw [fill=white] (4.54,1.66) circle (3.5pt);
\draw [fill=white] (5.9,4.62) circle (3.5pt);
\draw [fill=white] (-6.52,2.66) circle (3.5pt);
\draw [fill=white] (-5.28,3.58) circle (3.5pt);
\draw [fill=black] (-3.46,3.98) circle (3.5pt);
\draw [fill=black] (-1.9,4.98) circle (3.5pt);
\draw [fill=black] (4.16,-1.8) circle (3.5pt);
\draw [fill=black] (5.7,-0.1) circle (3.5pt);
\draw [fill=black] (-2.908620689655173,4.333448275862069) circle (2pt);
\draw [fill=black] (-2.3997670083876987,4.659636533084809) circle (2pt);
\draw [fill=black] (-4.2990323695426795,3.7955972814191914) circle (2pt);
\draw [fill=black] (-5.91289932885906,3.1104295302013423) circle (2pt);
\draw [fill=black] (-4.900283598833296,-0.4428341934036353) circle (2pt);
\draw [fill=black] (-2.278842071405597,-0.6818140881312322) circle (2pt);
\draw [fill=black] (-0.682062391681109,1.3005025996533797) circle (2pt);
\draw [fill=black] (0.00020403825717341117,-0.8202082890541977) circle (2pt);
\draw [fill=black] (1.861805007439995,-1.0578896292941709) circle (2pt);
\draw [fill=black] (3.4778368965850825,-0.9579549192222111) circle (2pt);
\draw [fill=black] (4.722777862247225,-1.1787517105063097) circle (2pt);
\draw [fill=black] (5.149983275049416,-0.7071613197506454) circle (2pt);
\draw [fill=black] (4.52394868332208,2.1201377447670495) circle (2pt);
\draw [fill=black] (4.500929101958136,2.780032410533423) circle (2pt);
\draw [fill=black] (5.120558244231851,3.939360720315137) circle (2pt);
\draw [fill=black] (3.683976069949379,3.9455959502991256) circle (2pt);
\end{scriptsize}
\end{tikzpicture}
\caption{The figure depicts an example of the subdivision of the graph $G^{\le 2}_{3,i}$ included in $G[U_i]$. The big black vertices are the critical ones in $V(G^{\le 2}_{3,i})\subseteq V(G[U_i])$. The big white vertices are the non-critical ones in $V(G^{\le 2}_{3,i})\subseteq V(G[U_i])$. The small black vertices are the ones that subdivide the edges of $G^{\le 2}_{3,i}$ in $G[U_i]$. Finally, the set $S_i$ consists of the vertices in the union of the encircled regions.}
\label{fig 16}
\end{figure}
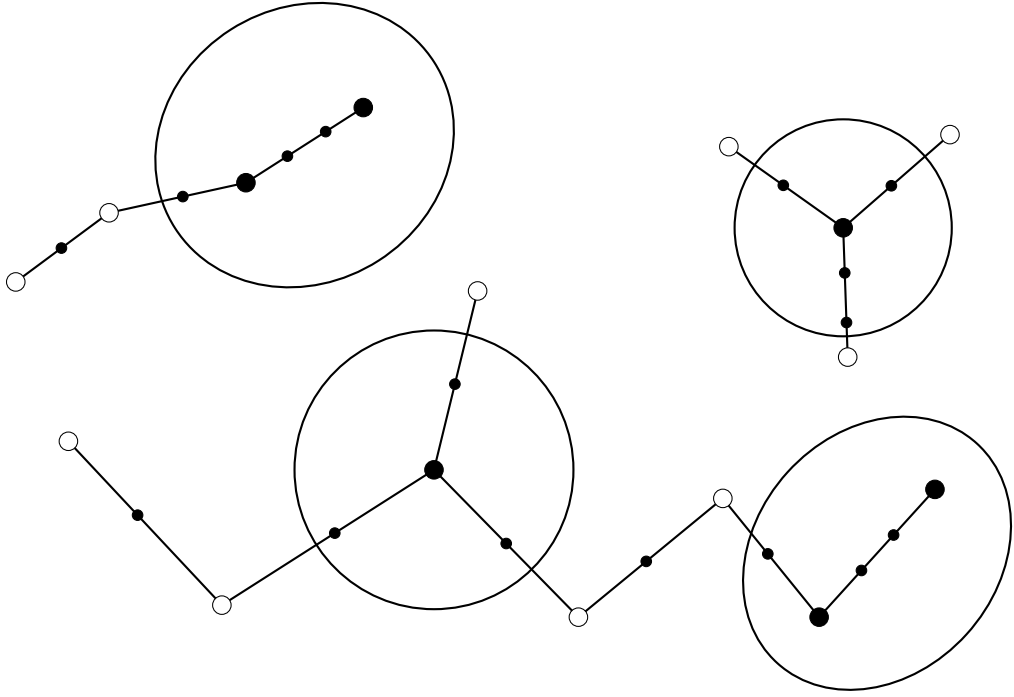

\begin{corollary}\label{cut lemma cor}
Suppose that $|S_{3-i}|\ge 52\ell + 1092$. Then, one may delete at most $20$ edges in $G^{\le 2}_{3,i}$ to obtain a graph $G''_{3,i}$, which does not contain a path of length at most $\ell/3$ between two critical vertices in $G^{\le 2}_{3,i}$, which contains only edges of weight one.
\end{corollary}
\begin{proof}
Every vertex in $S_{3-i}$, which subdivides an edge in $G^{\le 2}_{3-i}$, is at distance one to at least one critical vertex. On the other hand, every critical vertex in $G^{\le 2}_{3-i}$ is at distance one in $G[U_{3-i}]$ to at most three vertices in $S_{3-i}$. Thus, the number of critical vertices in $G^{\le 2}_{3, 3-i}$ is at least $|S_{3-i}|/4\ge 13\ell+273$. It remains to apply Lemma~\ref{cut lemma}.
\end{proof}

In the rest of the paper we assume without loss of generality that $|S_1|\le |S_2|$.

\begin{lemma}\label{l 3.1}
For both $i=1$ and $i=2$, at most one of the following happens:

\begin{enumerate}
    \item The sum of the weights of the edges
    \begin{equation*}
        \{e\in E(G_{3,i})\hspace{0.2em}|\hspace{0.2em} p(e)\ge 3\}
    \end{equation*}
    is at least $34$.
    \item $|S_{3-i}|\ge 129$.
\end{enumerate}
\end{lemma}
\begin{proof}
We argue by contradiction. Suppose that each of the above events happens for some $i\in \{1,2\}$. Then, by Observation~\ref{choice 1} and Observation~\ref{choice 2} we know that there are at most eight edges of weight two and at most 20 vertices of degree three.\par 
Suppose that there is an edge $e$ in $G^{+}_{3,3-i}$ of weight at least three and let $u,v,w$ be three consecutive vertices subdividing this edge in $G[U_{3-i}]$. Then, by the pigeonhole principle, either there are four edges in $G_{3,i}$ of weight at least three or there is one edge in $G_{3,i}$ of weight at least twelve. In both cases, one may find three consecutive vertices $u', v', w'$ in $G[U_i]$, subdividing an edge of $G^+_{3,i}$ and none of them being adjacent to any of $u, v$ and $w$. Therefore, one may exchange $u, v, w$ in $U_{3-i}$ with $u', v', w'$ in $U_i$ and thus improve the bisection $(U_1, U_2)$, which is a contradiction. See Figure~\ref{new fig 1}. Therefore, all edges in $G_{3,3-i}$ have weight at most two. In total, this would mean that $|S_{3-i}|\le 8\times 2\times 3 + 20\times 4 = 128$, which is a contradiction. The observation is proved.
\end{proof}

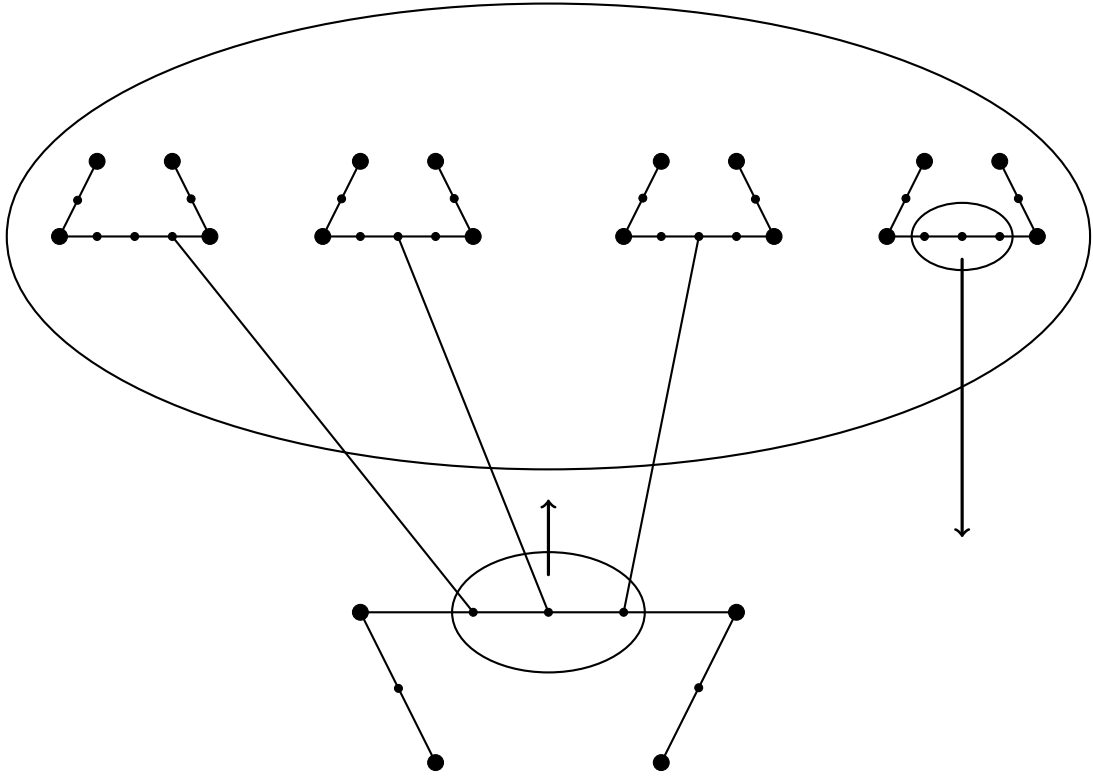
\begin{figure}[ht]
\centering
\begin{tikzpicture}[line cap=round,line join=round,x=1cm,y=1cm]
\clip(0.5,-5.535349064456274) rectangle (21.082566848929947,5.434659726710099);
\draw [rotate around={0:(9.5,2)},line width=0.8pt] (9.5,2) ellipse (7.200796519048111cm and 3.098623970206038cm);
\draw [line width=0.8pt] (3,2)-- (5,2);
\draw [line width=0.8pt] (3,2)-- (3.5,3);
\draw [line width=0.8pt] (5,2)-- (4.5,3);
\draw [line width=0.8pt] (16,2)-- (14,2);
\draw [line width=0.8pt] (14,2)-- (14.5,3);
\draw [line width=0.8pt] (16,2)-- (15.5,3);
\draw [line width=0.8pt] (6.5,2)-- (8.5,2);
\draw [line width=0.8pt] (12.5,2)-- (10.5,2);
\draw [line width=0.8pt] (6.5,2)-- (7,3);
\draw [line width=0.8pt] (8.5,2)-- (8,3);
\draw [line width=0.8pt] (10.5,2)-- (11,3);
\draw [line width=0.8pt] (12.5,2)-- (12,3);
\draw [line width=0.8pt] (7,-3)-- (8,-5);
\draw [line width=0.8pt] (7,-3)-- (12,-3);
\draw [line width=0.8pt] (12,-3)-- (11,-5);
\draw [line width=0.8pt] (4.5,2)-- (8.5,-3);
\draw [line width=0.8pt] (7.5,2)-- (9.5,-3);
\draw [line width=0.8pt] (11.5,2)-- (10.5,-3);
\draw [rotate around={0:(9.5,-3)},line width=0.8pt] (9.5,-3) ellipse (1.2807764064044274cm and 0.8002425902201281cm);
\draw [->,line width=1.2pt] (9.5,-2.5) -- (9.5,-1.5);
\draw [rotate around={0:(15,2)},line width=0.8pt] (15,2) ellipse (0.6708203932500033cm and 0.44721359550000195cm);
\draw [->,line width=1.2pt] (14.999583811631911,1.697036680664164) -- (15,-2);
\begin{scriptsize}
\draw [fill=black] (3,2) circle (3pt);
\draw [fill=black] (16,2) circle (3pt);
\draw [fill=black] (5,2) circle (3pt);
\draw [fill=black] (3.5,3) circle (3pt);
\draw [fill=black] (4.5,3) circle (3pt);
\draw [fill=black] (3.240926664802584,2.481853329605168) circle (1.5pt);
\draw [fill=black] (4.749653111629519,2.500693776740961) circle (1.5pt);
\draw [fill=black] (3.5,2) circle (1.5pt);
\draw [fill=black] (4,2) circle (1.5pt);
\draw [fill=black] (4.5,2) circle (1.5pt);
\draw [fill=black] (14,2) circle (3pt);
\draw [fill=black] (14.5,3) circle (3pt);
\draw [fill=black] (15.5,3) circle (3pt);
\draw [fill=black] (14.5,2) circle (1.5pt);
\draw [fill=black] (15,2) circle (1.5pt);
\draw [fill=black] (15.5,2) circle (1.5pt);
\draw [fill=black] (15.747997917789535,2.504004164420931) circle (1.5pt);
\draw [fill=black] (14.252877495403908,2.5057549908078167) circle (1.5pt);
\draw [fill=black] (6.5,2) circle (3pt);
\draw [fill=black] (8.5,2) circle (3pt);
\draw [fill=black] (12.5,2) circle (3pt);
\draw [fill=black] (10.5,2) circle (3pt);
\draw [fill=black] (7,3) circle (3pt);
\draw [fill=black] (8,3) circle (3pt);
\draw [fill=black] (11,3) circle (3pt);
\draw [fill=black] (12,3) circle (3pt);
\draw [fill=black] (6.7503836535279635,2.500767307055928) circle (1.5pt);
\draw [fill=black] (7,2) circle (1.5pt);
\draw [fill=black] (7.5,2) circle (1.5pt);
\draw [fill=black] (8,2) circle (1.5pt);
\draw [fill=black] (8.24744779369092,2.5051044126181603) circle (1.5pt);
\draw [fill=black] (10.755082813342506,2.5101656266850116) circle (1.5pt);
\draw [fill=black] (11,2) circle (1.5pt);
\draw [fill=black] (11.5,2) circle (1.5pt);
\draw [fill=black] (12,2) circle (1.5pt);
\draw [fill=black] (12.252146953505463,2.4957060929890758) circle (1.5pt);
\draw [fill=black] (7,-3) circle (3pt);
\draw [fill=black] (8,-5) circle (3pt);
\draw [fill=black] (12,-3) circle (3pt);
\draw [fill=black] (11,-5) circle (3pt);
\draw [fill=black] (8.5,-3) circle (1.5pt);
\draw [fill=black] (9.5,-3) circle (1.5pt);
\draw [fill=black] (10.5,-3) circle (1.5pt);
\draw [fill=black] (7.50623484625975,-4.012469692519501) circle (1.5pt);
\draw [fill=black] (11.49802466902774,-4.003950661944519) circle (1.5pt);
\end{scriptsize}
\end{tikzpicture}
\caption{The improvement from the proof of Lemma~\ref{l 3.1}. The large vertices are the ones in $G^{\le 2}_{3,1}\cup G^{\le 2}_{3,2}$. The small vertices are the subdivision vertices of the edges of $G^{\le 2}_{3,1}\cup G^{\le 2}_{3,2}$ in $G[U_1]\cup G[U_2]$.}
\label{new fig 1}
\end{figure}

\begin{corollary}\label{cor 3.2}
If the sum of the weights of the edges
\begin{equation*}
    \{e\in E(G_{3,1})\hspace{0.2em}|\hspace{0.2em} p(e)\ge 3\}
\end{equation*}
is at least $34$, then $|S_1|\le 128$.
\end{corollary}
\begin{proof}
This follows from Lemma~\ref{l 3.1} and the assumption that $|S_1|\le |S_2|$. 
\end{proof}

We partially characterized the structure of a minimum bisection. In the next two sections we do first moment computations for two types of bisections $(U_1, U_2)$ depending on $|S_1|$:

\begin{enumerate}
    \item Bisections $(U_1, U_2)$ of type one: $|S_1|\le \log^2 n$. This case is treated in Section~\ref{section 3}.
    \item Bisections $(U_1, U_2)$ of type two: $\log^2 n\le |S_1|$. This case will be considered in Section~\ref{section 4}.
\end{enumerate}

\section{Bisections of type one}\label{section 3}

In this case, we count bisections with $|S_1|\le \log^2 n$ with $S_1$ defined in~\eqref{def:S}.

We define the \textit{skeleton $Sk(H)$} of a labeled graph $H$ to be the unlabeled graph obtained from $H$ by deleting the labels of the vertices of $H$.\par

Recall that by Observation~\ref{typical}, a.a.s.\ $G(n,3)$ is a usual graph, so the results from Section~\ref{section 2} hold a.a.s.\ for $G(n,3)$. Thus, our aim in this section is to count the number of bisections $(U_1, U_2)$ of type one in usual 3-regular graphs. We begin by counting the number of possible skeletons of $G^+_{3,1}$. Then, we count the ways to give labels to the vertices in $G[U_1]$ in the subdivisions of these skeletons. Finally, we count the extensions of these subdivisions of $G^+_{3,1}$ to $G[U_1]$, and consequently to $G$.

Let $\beta \ge 0.1$ be such that $e(U_1, U_2) = \beta n$. Then, $G[U_1]$ contains $\beta n$ vertices of degree two. Since $|S_1|\le \log^2 n $, all but at most $\log^2 n $ of the edges of $G^+_{3,1}$ have weight one, and at most $\log^2 n $ vertices in $G^+_{3,1}$ have degree three.\par

\begin{observation}\label{ob 4.1}
By deleting the vertices of $S_1$ in $G^+_{3,1}$, we obtain a graph whose connected components are paths and cycles with edges of weight one.
\end{observation}
\begin{proof}
The maximum degree in this graph is two, and all vertices, incident to edges of weight more than one, have been deleted. The observation follows.
\end{proof}

\begin{observation}\label{ob 4.2}
The number of unlabeled graphs containing at most $\beta n$ edges and of maximal degree two is $\exp(o(n))$ as $n\to \infty$.
\end{observation}
\begin{proof}
Any graph of the above type consists of paths and cycles. Thus, for any $t\le \beta n$ and $k,\ell\in \mathbb N$ with $k+\ell = t$, the number of graphs with $k$ edges in paths and $\ell$ edges in cycles is given by $\exp(o(k)+o(l))$ by Theorem~\ref{Hardy - Ramanujan}. Summing over all pairs $(k,\ell)$ with $k+\ell\le \beta n$, we obtain an upper bound of $n^2\exp(o(n)) = \exp(o(n))$ on the number of unlabeled graphs of maximal degree two and at most $\beta n$ edges. The lemma is proved.
\end{proof}

The main object in this section will be the graph
\begin{equation}\label{eq:graph2}
G_{pc, 1} = G^{+}_{3,1}\setminus S_1.
\end{equation}
Let also $G_{pc, 1}$ have $\beta'n \ge \beta n - 3\log^2 n$ edges. Our goal is to bound from above the number of possibilities for the graph $G$ in which $(U_1, U_2)$ is a minimal bisection of size $\beta n$. To do this, we first construct the skeleton of $G_{pc, 1}$ and label its vertices. Then, we bound from above the number of extensions of $G_{pc, 1}$ to $G[U_1]$ and consequently to $G$. After that, we optimize with respect to the parameter $\beta$.

Let $t_i = t_i(n)$ be the number of paths of length $i$ in $G_{pc, 1}$ and let $c_i = c_i(n)$ be the number of cycles of length $i$ edges in $G_{pc, 1}$.

\begin{lemma}\label{lem:autobisection1}
The number of automorphisms of the graph $G_{pc, 1}$ are 
\begin{equation*}
    \left(\prod_{i\ge 1} 2^{t_i} t_i!\right) \times c_1! \times 2^{c_2}c_2!\times \prod_{i\ge 3} (2i)^{c_i} c_i!.
\end{equation*}
\end{lemma}
\begin{proof}
First, the paths (respectively the cycles) of the same length are indistinguishable. Second, a path has two symmetries and a cycle has one symmetry, if it is of length one, two symmetries, if it is of length two, and $2i$ symmetries, if it is of length $i\ge 3$. This proves the lemma.
\end{proof}

Now, we count the number of bisections $(U_1, U_2)$ of size $\beta n$. By Observation~\ref{ob 4.2}, the number of possible (unlabeled) graphs for $G_{pc, 1}$ is subexponential. We conclude that the sum over all possibilities for $(t_i)_{i\ge 1}$ and $(c_i)_{i\ge 1}$ must be dominated, up to a subexponential factor, by the number of extensions of the unlabeled graph $G_{pc, 1}$, which has the largest number of extensions among all unlabeled graphs of maximum degree two on at most $\beta n$ vertices. Thus, maximizing over $(t_i)_{i\ge 1}$ and $(c_i)_{i\ge 1}$ will give us the correct exponential order of growth of the number of extensions in general. This is what we do in the sequel. For the same reason, we ignore the fact that the size of $U_1$ is between $n/2 - 5$ and $n/2 + 5$, since summing over all possible sizes of $U_1$ does not make a difference on an exponential scale.\par

We now explain our counting procedure step by step. Since $G$ is a 3-regular graph on $n$ vertices, $n$ is even and therefore $n/2\in \mathbb N$.

\begin{enumerate}
    \item Choose in $\binom{n}{n/2}$ ways the labels of the vertices participating in $U_1$ and the labels of the vertices participating in $U_2$.
    \item Choose the $\beta n$ labels of the vertices of degree two in $G[U_1]$ in $\binom{0.5n}{\beta n} = \exp(o(n))\binom{0.5 n}{\beta' n}$ ways. Out of these, choose in $\binom{\beta n}{\beta' n} = \exp(o(n))$ ways which of these vertices must subdivide the edges of $G_{pc, 1}$.
    \item Assign the labels to the vertices in $(\beta' n)!(\beta n - \beta' n)! = \exp(o(n)) (\beta' n)!$ ways. 
    \item Choose $Tn := \sum_{i\ge 1} (i+1) t_i$ labels for the vertices of degree three in $U_1$ participating in the paths in $G_{pc, 1}$ and $Cn := \sum_{i\ge 1} i c_i$ labels for the vertices of degree three in $U_1$ participating in the cycles in $G_{pc, 1}$ in
    \begin{equation*}
        \binom{(0.5 - \beta)n}{Tn, Cn, (0.5 - T - C - \beta)n} = \exp(o(n)) \binom{(0.5 - \beta')n}{Tn, Cn, (0.5 - T - C - \beta')n}
    \end{equation*}
    ways. Here, $T = T(n)$ and $C = C(n)$ are functions of $n$.
    \item Assign the labels to the vertices in $(Tn)!(Cn)!$ ways.
    \item Divide by the size of the automorhpism group of $G_{pc, 1}$, which by Lemma~\ref{lem:autobisection1} is 
    \begin{equation*}
        \left(\prod_{i\ge 1} 2^{t_i} t_i!\right) \times c_1! \times 2^{c_2}c_2!\times \prod_{i\ge 3} (2i)^{c_i} c_i!, 
    \end{equation*}
    since different ways to distribute the labels might lead to the same final (labeled) graph.
    \item Extend the labeled copy of the optimal skeleton in 
    \begin{equation*}
        \dfrac{((1.5-4\beta'-\beta)n)!!}{2^{2(T + C - \beta')n + (\beta - \beta')n}6^{(0.5-(\beta+T+C))n}} = \exp(o(n)) \dfrac{((1.5-5\beta')n)!!}{2^{2(T + C - \beta')n}6^{(0.5-(\beta'+T+C))n}}
    \end{equation*}
    ways to a graph $G[U_1]$. The exponent of 2 in the formula comes from the fact that the number of paths in $G_{pc, 1}$ is exactly $(T+C-\beta')n$ since, first, every cycle contains the same number of edges and vertices, and second, every path contains one vertex more than edges. Moreover, every path of length at least $1$ contains two vertices of degree one.
    \item Choose $\beta n$ labels for the vertices of degree two in $U_2$ in $\binom{0.5n}{\beta n} = \exp(o(n))\binom{0.5 n}{\beta' n}$ ways.
    \item Form the matching between the vertices of degree two in $G[U_1]$ and $G[U_2]$ in $(\beta n)! = \exp(o(n)) (\beta' n)!$ ways. 
    \item Construct the graph $G[U_2]$ in
    \begin{equation*}
        \dfrac{((1.5-\beta)n)!!}{2^{\beta n}6^{(0.5-\beta)n}} = \exp(o(n)) \dfrac{((1.5-\beta')n)!!}{2^{\beta' n}6^{(0.5-\beta')n}}.
    \end{equation*}
    \item Multiply by $6^n$ to count configurations instead of graphs.
    \item Finally, divide by the total number of 3-regular configurations $(3n-1)!!$ to find an upper bound on the proportion of bisections of type one.
\end{enumerate}

The final formula is
\begin{equation*}
     \dfrac{\mathrm{e}^{o(n)}\binom{n}{0.5n}\binom{0.5n}{\beta' n} (\beta' n)! \binom{(0.5 - \beta')n}{Tn, Cn, (0.5 - \beta' - T - C)n} (Tn)!(Cn)! \dfrac{((1.5-5\beta')n)!!}{2^{2(T + C - \beta')n}6^{(0.5-(\beta'+T+C))n}} \binom{0.5n}{\beta' n} (\beta' n)! \dfrac{((1.5-\beta')n)!!}{2^{\beta' n}6^{(0.5-\beta')n}} 6^n}{\left(\prod_{i\ge 1} 2^{t_i} t_i!\right) \times c_1! \times 2^{c_2}c_2!\times \left(\prod_{i\ge 3} (2i)^{c_i} c_i!\right) (3n-1)!!}.
\end{equation*}

Before proceeding with explicit optimization computation, we observe that the numerator depends only on $T+C$ as a function of $T$ and $C$. We define a \textit{cycle transformation} of some unlabeled graph $H$ of degree at most two to be the unlabeled graph containing the exact same multiset of paths as $H$ and in which each of the remaining vertices participates in one common cycle. Remark that neither the number of vertices in $H$ nor the number of edges changes by applying this transformation and therefore both $T+C$ and $\beta'$ remain unchanged. Then, the denominator becomes 
\begin{equation*}
    \left(\prod_{i\ge 1} 2^{t_i} t_i!\right) \times 2(\beta' n - \sum_{i\ge 1} it_i)\times (3n-1)!!.
\end{equation*}

Since the term $2(\beta' n - \sum_{i\ge 1}it_i) = \exp(o(n))$, we may include it in the $\exp(o(n))$ term in the beginning of the formula and thus simplify the expression, leaving only the terms in $(t_i)_{i\ge 1}$ and $T$. Moreover, taking one edge out of the large cycle (which decreases $\beta' n$ by one) and transforming it into a path changes the number of extensions by at most $\Theta(n)$. Therefore, one may suppose in this optimization part that we optimize over the unlabeled graphs $H$, which consist of a multiset of paths. The final formula simplifies to 
\begin{equation}\label{eq sec 4}
     \dfrac{\exp(o(n))\binom{n}{0.5n}\binom{0.5n}{\beta' n} (\beta' n)! \binom{(0.5 - \beta')n}{Tn} (Tn)! \dfrac{((1.5-5\beta')n)!!}{2^{2(T - \beta')n}6^{(0.5-(\beta'+T))n}} \binom{0.5n}{\beta' n} (\beta' n)! \dfrac{((1.5-\beta')n)!!}{2^{\beta' n}6^{(0.5-\beta')n}} 6^n}{\left(\prod_{i\ge 1} 2^{t_i} t_i!\right) (3n-1)!!}.
\end{equation}

Now we maximize the above formula over the parameters $(t_i)_{i\ge 1}$ under the conditions
\begin{enumerate}
    \item\label{c 1}
    \begin{equation*}
        Tn - \sum_{i\ge 1} t_i = \beta' n \iff \sum_{i\ge 1} t_i = (T - \beta')n.
    \end{equation*}
    \item\label{c 2}
    \begin{equation*}
        \sum_{i\ge 1} (i+1)t_i = Tn \iff \sum_{i\ge 1} it_i = \beta' n,
    \end{equation*}
\end{enumerate}
where the second equivalence uses the left equality in the first one. Our first main goal is the minimize the product 
\begin{equation*}
    \prod_{i\ge 1} t_i!
\end{equation*}
itself for fixed $T$ and $\beta'$ since $\prod_{i\ge 1} 2^{t_i} = 2^{(T-\beta')n}$. 

\begin{lemma}\label{decrease}
There is a sequence $(t^{(n)}_i)_{i\ge 1}$, which minimizes the function $(t_i)_{i\ge 1}\mapsto \prod_{i\ge 1} t_i!$ under the conditions~\ref{c 1} and~\ref{c 2}, such that for all but at most one $i\ge 1$ we have $t^{(n)}_i\ge t^{(n)}_{i+1}$. Moreover, for this exceptional $i$ we may only have $t^{(n)}_i = 0$, $t^{(n)}_{i+1} = 1$ and $t^{(n)}_{i+2} = 0$.
\end{lemma}
\begin{proof}
Let $(t_i)_{i\ge 0}$ be a minimizing sequence for $\prod_{i\ge 1} t_i!$ under the two conditions above. Let $(t'_i)_{i\ge 0}$ be a sequence such that $\{t'_i\}_{i\ge 1} \equiv \{t_i\}_{i\ge 1}$ as multisets of non-negative integers and $(t'_i)_{i\ge 1}$ is decreasing.\par
Then, first, $\sum_{i\ge 1} t_i = \sum_{i\ge 1} t'_i = (T-\beta')n$ and $\prod_{i\ge 1} t_i! = \prod_{i\ge 1} t'_i!$. Moreover, let $m = \max\{i\in \mathbb N, t'_i\ge 1\}$. Now, define $(t^{(n)}_i)_{i\ge 1}$ as follows: 
\begin{equation*}
    t^{(n)}_i = 
\begin{cases}
t'_i & \text{if }i\notin \{m, m + \beta' n - \sum_{i\ge 1} it'_i\} \\    
t'_i - 1 & \text{if } i = m, \\
1 & \text{if } i = m + \beta' n - \sum_{i\ge 1} it'_i.
\end{cases}
\end{equation*}
See Figure~\ref{fig 12}. One may easily verify that the product $\prod_{i\ge 1} t^{(n)}_i!$ can only decrease and this time $(t^{(n)}_i)_{i\ge 1}$ satisfies both conditions given above.
\end{proof}

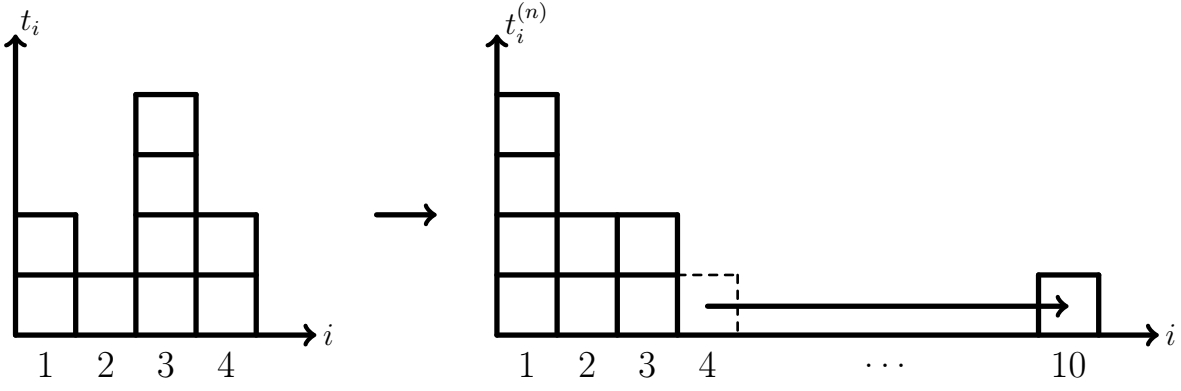
\begin{figure}
\centering
\begin{tikzpicture}[scale=0.8,line cap=round,line join=round,x=1cm,y=1cm]
\clip(-9.5,-5) rectangle (15.94,2);
\draw [->,line width=2pt] (-9,-4) -- (-4,-4);
\draw [->,line width=2pt] (-9,-4) -- (-9,1);
\draw [line width=2pt] (-9,-2)-- (-8,-2);
\draw [line width=2pt] (-9,-3)-- (-8,-3);
\draw [line width=2pt] (-8,-2)-- (-8,-4);
\draw [line width=2pt] (-8,-3)-- (-7,-3);
\draw [line width=2pt] (-7,-4)-- (-7,0);
\draw [line width=2pt] (-7,0)-- (-6,0);
\draw [line width=2pt] (-6,0)-- (-6,-4);
\draw [line width=2pt] (-7,-3)-- (-6,-3);
\draw [line width=2pt] (-7,-2)-- (-6,-2);
\draw [line width=2pt] (-7,-1)-- (-6,-1);
\draw [line width=2pt] (-6,-2)-- (-5,-2);
\draw [line width=2pt] (-5,-2)-- (-5,-4);
\draw [line width=2pt] (-6,-3)-- (-5,-3);
\draw [->,line width=2pt] (-3,-2) -- (-2,-2);
\draw [->,line width=2pt] (-1,-4) -- (-1,1);
\draw [->,line width=2pt] (-1,-4) -- (10,-4);
\draw [line width=2pt] (0,0)-- (0,-4);
\draw [line width=2pt] (-1,0)-- (0,0);
\draw [line width=2pt] (-1,-1)-- (0,-1);
\draw [line width=2pt] (-1,-2)-- (0,-2);
\draw [line width=2pt] (-1,-3)-- (2,-3);
\draw [line width=2pt] (0,-2)-- (2,-2);
\draw [line width=2pt] (2,-2)-- (2,-4);
\draw [line width=2pt] (1,-2)-- (1,-4);
\draw [style = dashed, line width=1pt] (2,-3)-- (3,-3);
\draw [style = dashed, line width=1pt] (3,-3)-- (3,-4);
\draw [->,line width=2pt] (2.5,-3.52) -- (8.5,-3.52);
\draw [line width=2pt] (8,-4)-- (8,-3);
\draw [line width=2pt] (8,-3)-- (9,-3);
\draw [line width=2pt] (9,-3)-- (9,-4);
\begin{scriptsize}
\draw [fill=black] (-8.5,-4.5) node {\Large{1}};
\draw [fill=black] (-7.5,-4.5) node {\Large{2}};
\draw [fill=black] (-6.5,-4.5) node {\Large{3}};
\draw [fill=black] (-5.5,-4.5) node {\Large{4}};

\draw [fill=black] (-8.5+8,-4.5) node {\Large{1}};
\draw [fill=black] (-7.5+8,-4.5) node {\Large{2}};
\draw [fill=black] (-6.5+8,-4.5) node {\Large{3}};
\draw [fill=black] (-5.5+8,-4.5) node {\Large{4}};

\draw [fill=black] (-5.5+1.7,-4) node {\large{$i$}};
\draw [fill=black] (-5.5+1.7+14,-4) node {\large{$i$}};

\draw [fill=black] (-5.5+1.75-5,1.2) node {\large{$t_i$}};
\draw [fill=black] (-5.5+2+3,1.2) node {\large{$t^{(n)}_i$}};

\draw [fill=black] (-5.5+11,-4.5) node {\Large{\dots}};

\draw [fill=black] (-5.5+14,-4.5) node {\Large{10}};

\end{scriptsize}
\end{tikzpicture}
\caption{An example of a sequence $(t_i)_{i\ge 1}$ being transformed into $(t^{(n)}_i)_{i\ge 1}$ (in this example $\beta' n-\sum_{i\ge 1} i t'_i=6$).}
\label{fig 12}
\end{figure}

\begin{observation}\label{ob_isolated_i}
In the sequence $(t^{(n)}_i)_{i\ge 1}$, the second-largest $i\ge 1$, for which $t^{(n)}_i \ge 1$, is less than $\sqrt{2\beta' n}$. 
\end{observation}
\begin{proof}
Notice that all indices $j \le i$ to the left of this second-largest $i$ satisfy $t_j^{(n)} \ge 1.$ For this second-largest $i$ we have 
\begin{equation*}
    \beta' n\ge \sum_{1\le j\le i} jt^{(n)}_j \ge \dfrac{i(i+1)}{2},
\end{equation*}
so $2\beta' n \ge i^2$.
\end{proof}

\begin{corollary}\label{cor 4.8}
$t^{(n)}_1 \ge \tfrac{(T - \beta') n - 1}{\sqrt{2\beta' n}}$.
\end{corollary}
\begin{proof}
Let $i$ be the second-largest index, for which $t_i^{(n)}\ge 1$. We have that 
\begin{equation*}
    \sqrt{2\beta' n} t_1^{(n)}\ge it_1^{(n)}\ge \sum_{1\le j\le i} t_j^{(n)} = (T-\beta') n - 1.
\end{equation*}
\end{proof}

We distinguish two cases. First, let us treat the sequences $(t^{(n)})_{i\ge 1}$, for which $t^{(n)}_1=1$ (and therefore for every $i\ge 1$ one has $t^{(n)}_i\in \{0,1\}$). In this case, by the proof of Corollary~\ref{cor 4.8}, we have that $\sqrt{2\beta' n}+1\ge (T-\beta')n$. Then, one may rewrite~\eqref{eq sec 4} as
\begin{equation*}
     \dfrac{\exp(o(n))\binom{n}{0.5n}\binom{0.5n}{\beta' n} (\beta' n)! \binom{(0.5 - \beta')n}{\beta' n} (\beta' n)! \dfrac{((1.5-5\beta')n)!!}{6^{(0.5-2\beta')n}} \binom{0.5n}{\beta' n} (\beta' n)! \dfrac{((1.5-\beta')n)!!}{2^{\beta' n}6^{(0.5-\beta')n}} 6^n}{(3n-1)!!}.
\end{equation*}

By Stirling's formula we deduce that the above expression can be rewritten as
\begin{equation*}
    \exp(o(n))\left(\dfrac{(1.5 - 5\beta')^{\frac{1.5-5\beta'}{2}} (1.5 - \beta')^{\frac{1.5-\beta'}{2}} 6^{3\beta'}}{2^{\beta'} 3^{1.5} (0.5-\beta')^{0.5-\beta'} (0.5-2\beta')^{0.5-2\beta'}}\right)^n.
\end{equation*}

One may easily check that the maximum of the function
\begin{equation*}
    \beta'\in [0, 0.25)\mapsto \dfrac{(1.5 - 5\beta')^{\frac{1.5-5\beta'}{2}} (1.5 - \beta')^{\frac{1.5-\beta'}{2}} 6^{3\beta'}}{2^{\beta'} 3^{1.5} (0.5-\beta')^{0.5-\beta'} (0.5-2\beta')^{0.5-2\beta'}}
\end{equation*}
over the interval $[0, 0.25)$ is strictly less than one. We deduce that at most an exponentially small fraction of the configurations have bisections with $t^{(n)}_1 = 1$.\par

We now treat the second case, in which $t^{(n)}_1\ge 2$.

\begin{observation}\label{ob 4.9}
For every $j\ge \lceil 2\sqrt{2n}+2\rceil$, we have that $t^{(n)}_j = 0$.
\end{observation}
\begin{proof}
We argue by contradiction. Suppose that $t^{(n)}_j \ge 1$ for some $j\ge \lceil 2\sqrt{2n}+2\rceil$. Then, define $(s^{(n)}_{\ell})_{\ell\ge 0}$ as follows:
\begin{equation*}
    s^{(n)}_{\ell} = 
\begin{cases}
t^{(n)}_{\ell} & \text{if }\ell\notin \{1, 1 + \lceil \sqrt{2n}\rceil, j - \lceil \sqrt{2n}\rceil, j\} \\   
t^{(n)}_1 - 1 & \text{if } \ell = 1, \\
1 & \text{if } \ell = 1 + \lceil \sqrt{2n}\rceil, \\
1 & \text{if } \ell = j - \lceil \sqrt{2n}\rceil, \\
0 & \text{if } \ell = j.
\end{cases}
\end{equation*}
Using Lemma~\ref{decrease} and Observation~\ref{ob_isolated_i}, one may easily verify that
\begin{equation*}
    \dfrac{\prod_{\ell\ge 1} t^{(n)}_{\ell}!}{\prod_{\ell\ge 1} s^{(n)}_{\ell}!} = t^{(n)}_1 \ge 2,
\end{equation*}
which is a contradiction, since $(s^{(n)}_{\ell})_{\ell\ge 1}$ satisfies both conditions~\ref{c 1} and~\ref{c 2} and $(t^{(n)}_{\ell})_{\ell\ge 1}$ is a sequence minimizing the function $(t_i)_{i\ge 1}\mapsto \prod_{i\ge 1} t_i!$ with these properties.
\end{proof}

Since for our purposes a subexponential factor in the formula (\ref{eq sec 4}) does not matter, by abuse we forget about this largest isolated positive term of the sequence $(t^{(n)}_i)_{i\ge 1}$, if it exists. Indeed, by Observation~\ref{ob 4.9}, it contributes at most $\lceil 2\sqrt{2n} + 2\rceil$ to the sum $\sum_{i\ge 1} it^{(n)}_i$ and at most one to the sum $\sum_{i\ge 1} t^{(n)}_i$. From now on, we consider the sequence $(t^{(n)}_i)_{i\ge 1}$ to be decreasing.\par
Let $m = m(n) = \max\{i\in \mathbb N, t^{(n)}_i \ge 1\}$. By Observation~\ref{ob 4.9} we have that $m(n) \le 2\sqrt{2n} + 2$. For the terms $(t^{(n)}_i)_{1\le i\le m}$, we bound $t^{(n)}_i!$ from below by 
\begin{equation*}
    \left(\dfrac{t^{(n)}_i}{e}\right)^{t^{(n)}_i}\sqrt{2\pi t^{(n)}_i}.
\end{equation*}\

We remark that the product of the terms $\sqrt{2\pi t^{(n)}_i}$ for $i\in [m]$ is at most $(\sqrt{2\pi n})^m \le (\sqrt{2\pi n})^{2\sqrt{2n}+2} = \exp(o(n))$, so it is absorbed by the $\exp(o(n))$ term in the beginning of the formula. Since $\prod_{1\le i\le m} \exp(-t^{(n)}_i) = \exp((\beta' - T)n)$ depends only on $\sum_{i=1}^m t^{(n)}_i$ but not on any of the individual terms $t^{(n)}_i$, we need to minimize the quantity
\begin{equation*}
    \prod_{1\le i\le m} {t^{(n)}_i}^{t^{(n)}_i}.
\end{equation*}
under the constraints~\ref{c 1} and~\ref{c 2}. We rewrite this as 
\begin{equation}\label{eq 1 sec 4}
    n^{(\sum_{i\ge 1}t^{(n)}_i)}  \left(\prod_{1\le i\le m} \left(\dfrac{t^{(n)}_i}{n}\right)^{\frac{t^{(n)}_i}{n}}\right)^n = n^{(\beta' - T)n} \left(\prod_{1\le i\le m} \left(\dfrac{t^{(n)}_i}{n}\right)^{\frac{t^{(n)}_i}{n}}\right)^n.
\end{equation}

Define the function
\begin{equation*}
    f_m: (t_i)_{1\le i\le m}\in [0,1]^m \mapsto \sum_{1\le i\le m} t_i\ln(t_i)\in \mathbb R.
\end{equation*}
By extending $t\in (0, 1]\mapsto t\log t\in \mathbb R$ at zero by continuity to the value zero, $f_m$ may be seen as a projection of the function
\begin{equation*}
    f: (t_i)_{i\ge 1}\in [0,1]^{\mathbb N} \mapsto \sum_{1\le i\le m} t_i\ln(t_i)\in \mathbb R\cup \{-\infty\}
\end{equation*}
onto its first $m$ coordinates.\par

Clearly under the conditions 
\begin{equation}\label{c'_1}
    \sum_{i\ge 1} t_i = T-\beta'
\end{equation}
and
\begin{equation}\label{c'_2}
    \sum_{i\ge 1} it_i = \beta',
\end{equation}

the minimum of the function $f_m = f_{m(n)}$ is larger than the minimum of the function $f$. On the other hand, $f^{|\mathbb R}_{|(0,1]^{\mathbb N}}$ is a convex and infinitely differentiable function on an infinite-dimensional Banach space, and therefore by (\cite{Lag1}, Theorem 1) and \cite{Lag2}  we know that if there is some critical point in the interior of the domain, it must be unique and it must be a global minimum for the function $f$. This allows us to apply the method of Lagrange multipliers for the function $f$ under the constraints (\ref{c'_1}) and (\ref{c'_2}) for any fixed $n$ in the infinite-dimensional setting. 

Let
\begin{equation*}
    F((t_i)_{i\ge 1}, \lambda_1, \lambda_2) := f((t_i)_{i\ge 1}) - \lambda_1\sum_{i\ge 1} t_i - \lambda_2\sum_{i\ge 1} it_i.
\end{equation*}
Differentiating with respect to $t_i$ and setting the derivative to zero gives
\begin{equation*}
    1 + \ln(t_i) - \lambda_1 - i\lambda_2 = 0 \iff t_i = \exp(\lambda_1 - 1 + i\lambda_2).
\end{equation*}
Therefore, we solve the following system to find $\lambda_1$ and $\lambda_2$:
\begin{equation*}
\begin{cases}    
\dfrac{\exp(\lambda_1 + \lambda_2 - 1)}{1 - \exp(\lambda_2)} = T - \beta', \\
\dfrac{\exp(\lambda_1 + \lambda_2 - 1)}{(1 - \exp(\lambda_2))^2} = \beta'.
\end{cases}
\end{equation*}
Solving this system gives
\begin{equation*}
\begin{cases}    
\exp(\lambda_1) = \dfrac{e(T-\beta')^2}{2\beta' - T}, \\
\exp(\lambda_2) = \dfrac{2\beta' - T}{\beta'}.
\end{cases}
\end{equation*}

Thus, we have that 
\begin{equation*}
    \forall i\ge 1, t_i = \dfrac{(T-\beta')^2}{2\beta' - T} \left(\dfrac{2\beta' - T}{\beta'}\right)^i
\end{equation*}
is the argument where the absolute minimum of $f$ is attained. The value of this minimum is
\begin{equation*}
    f((t_i)_{i\ge 1}) = \left(\log\left(\dfrac{(T-\beta')^2}{\beta'}\right) - \dfrac{2\beta - T}{\beta'}\log\left(\dfrac{(T-\beta')^2}{2\beta' - T}\right)\right)\beta'.
\end{equation*}

Plugging in this minimum consecutively into~\eqref{eq 1 sec 4} and then into~\eqref{eq sec 4} leads to the formula
\begin{equation*}
    \exp(o(n)) \left(\dfrac{(1.5 - 5\beta')^{\frac{1.5 - 5\beta'}{2}} 2^{4\beta' - 2T} 3^{2\beta' + T - 1.5} (1.5-\beta')^{\frac{1.5-\beta'}{2}}}{(0.5-\beta)^{0.5-\beta'} (0.5 - \beta' - T)^{0.5-\beta' - T} \exp\left( \beta'\left(\log\left(\dfrac{(T-\beta')^2}{\beta'}\right) - \dfrac{2\beta' - T}{\beta'}\log\left(\dfrac{(T-\beta')^2}{2\beta' - T}\right)\right)\right)}\right)^n.
\end{equation*}

Taking the logarithm of the $n-$th root of the entire formula and letting $n\to \infty$, we obtain the following expression, which remains to be maximized as a function of $\beta'$ and $T$:

\begin{equation}\label{eq 2 sec 4}
\left. \begin{array}{r} 
0.5(1.5 - 5\beta')\ln(1.5 - 5\beta') + (4\beta' - 2T)\ln(2) + (2\beta' + T - 1.5)\ln(3) \\\\
+ 0.5(1.5-\beta')\ln(1.5-\beta') - (0.5-\beta')\ln(0.5-\beta') - (0.5 - \beta' - T)\ln(0.5-\beta' - T) \\\\
- \beta\left(\log\left(\dfrac{(T-\beta')^2}{\beta'}\right) - \dfrac{2\beta' - T}{\beta'}\log\left(\dfrac{(T-\beta')^2}{2\beta' - T}\right)\right)
\end{array} \right\} 
\end{equation}

One may easily observe that in the range $\beta'\in [0.1,0.1069]$ the above function is well defined for every $T\in (\beta', 2\beta')$ and may be extended by continuity at the values $T = \beta'$ and $T = 2\beta'$. Maximizing this function in the range $\{(\beta', T): \beta \in [0.10, 0.1069], \, \beta'\le T\le 2\beta'\}$ gives a maximum of $-3.713\times 10^{-5}$, which is attained at the point $(\beta', T) = (0.1069, 0.1802)$. This calculation is confirmed by the graphing calculator Desmos - see Figure~\ref{fig 13}. It proves that the proportion of 3-regular graphs containing a bisection of size $\beta\le 0.1069 n$ of type one is exponentially small.

\begin{figure}
\centering
\includegraphics{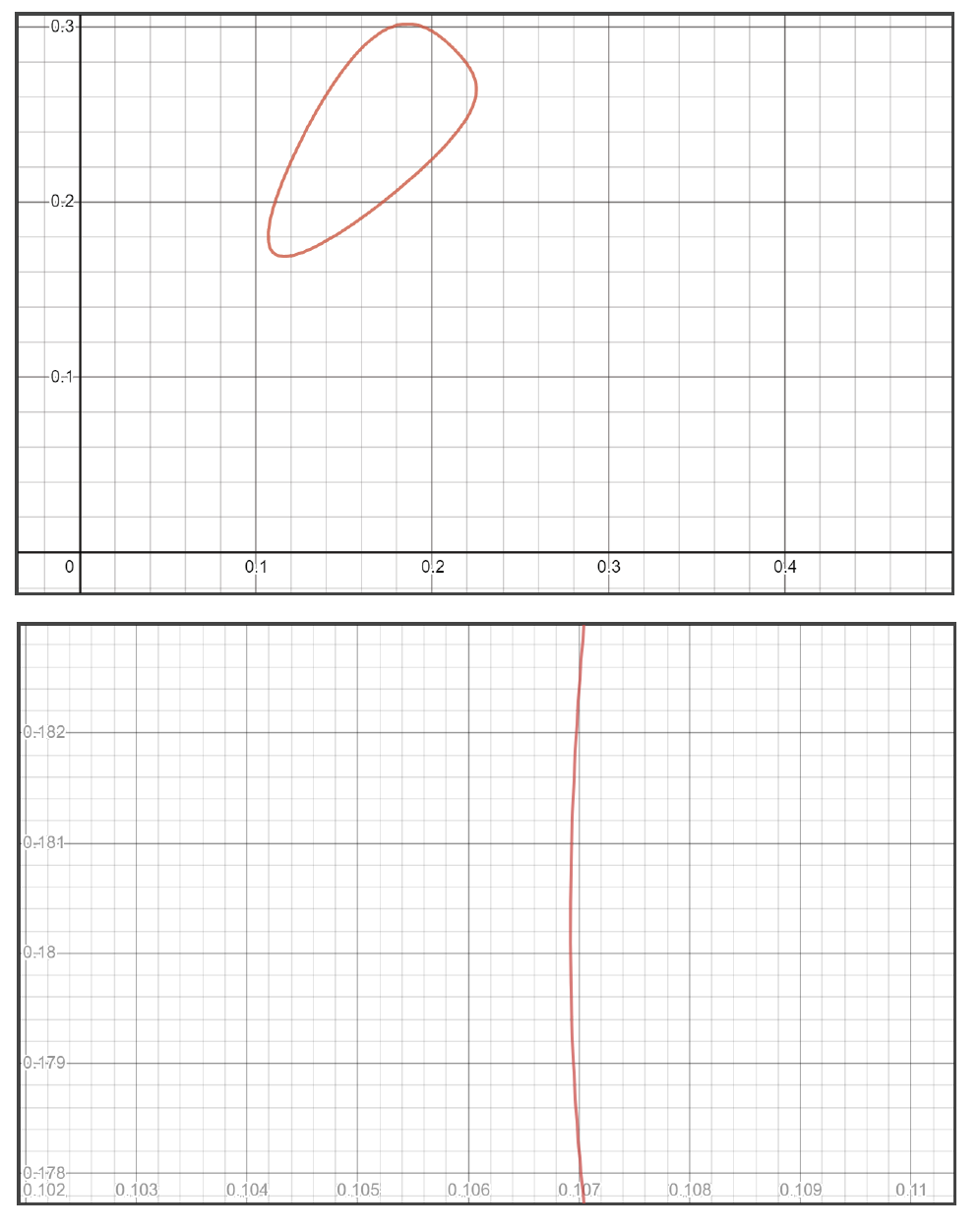}
\caption{In both figures, the horizontal axis represents the $\beta'$-coordinate, and the vertical axis stands for the $T$-coordinate. In the top figure, the curve represents the set of coordinates $(\beta', T)$, for which the expression (\ref{eq 2 sec 4}) is equal to zero. The exterior of the encircled region is the set where (\ref{eq 2 sec 4}) is negative, and its interior corresponds to the set where (\ref{eq 2 sec 4}) is positive. The bottom figure is a zoomed copy of the top one around the point on the curve with minimal $\beta'$.}
\label{fig 13}
\end{figure}

\section{Bisections of type two}\label{section 4}
This section is dedicated to counting the bisections of type two. Recall $S_1$ and $S_2$ from~\eqref{def:S}. In this section, we suppose that $\log^2 n \le |S_1|\le |S_2|$. 

\begin{lemma}\label{trivial lem 5.1}
For both $i = 1$ and $i = 2$, the number of edges in $E(G^+_{3,i})\setminus E(G^{\le 2}_{3,i})$ is at most $12$. Moreover, the number of extensions of the skeleton of the graph $G^{\le 2}_{3,i}$ to the skeleton of the graph $G^+_{3,i}$ is $\exp(o(n))$.
\end{lemma}
\begin{proof}
By Lemma~\ref{l 3.1} one may conclude that, for both $i=1$ and $i=2$ and for every large enough $n$, the sum of the weights of the edges
\begin{equation*}
    \{e\in E(G_{3,i})\hspace{0.2em}|\hspace{0.2em} p(e)\ge 3\}
\end{equation*}
is at most 34. This means in particular that, for both $i=1$ and $i=2$, $E(G^+_{3,i})\setminus E(G^{\le 2}_{3,i})$ contains less than $\lceil 34/3\rceil = 12$ edges and therefore the number of ways to extend $G^{\le 2}_{3,i}$ to $G^+_{3,i}$ is less than
\begin{equation*}
    \sum_{0\le i\le 12} n^{2i} = \exp(o(n)).
\end{equation*}
Indeed, knowing $G^{\le 2}_{3,i}$, the number of ways to add a new edge is always at most $n^2$. The lemma is proved.
\end{proof}

Since our counting strategy will be similar to that in Section~\ref{section 3}, polynomial factors will not be of any importance for us. Hence, we may and do assume that $|U_1| = |U_2| = n/2$ (recall that $||U_1| - |U_2||\le 10$ in general).\par

By Corollary~\ref{cut lemma cor} we know that, for both $i=1$ and $i=2$, by deleting at most 20 edges from $G^{\le 2}_{3, i}$ we can obtain a graph in which the minimal length of a path of edges of weight one between critical vertices is at least $\tfrac{\log^2 n  - 1092}{3\cdot 52}$. We call this graph $G_{ld, i}$. This graph inherits the weights of the edges that come from $G^{\le 2}_{3,i}$.

\begin{figure}
\centering
\begin{tikzpicture}[line cap=round,line join=round,x=1cm,y=1cm]
\clip(-5,-0.5) rectangle (5.5,3);
\draw [line width=2.5pt] (-4.32,2.06)-- (-4.32,-0.14);
\draw [->,line width=2pt] (-1.92,1.04) -- (1.06,1.04);
\draw [line width=0.5pt] (3,2.06)-- (3,-0.06);
\draw [line width=0.5pt] (3,1.08)-- (4,2.06);
\draw [line width=0.5pt] (3,1.08)-- (4,-0.06);
\begin{scriptsize}
\draw [fill=black] (-4.32,2.06) circle (2.5pt);
\draw [fill=black] (-4.32,-0.14) circle (2.5pt);
\draw [fill=black] (3,2.06) circle (2.5pt);
\draw [fill=black] (3,-0.06) circle (2.5pt);
\draw [fill=black] (3,1.08) circle (2.5pt);
\draw [fill=black] (4,2.06) circle (2.5pt);
\draw [fill=black] (4,-0.06) circle (2.5pt);

\draw [fill=black] (-4.32-0.3,2.06) node {\Large{$u$}};
\draw [fill=black] (-4.32-0.3,-0.14) node {\Large{$v$}};
\draw [fill=black] (3-0.3,2.06) node {\Large{$u$}};
\draw [fill=black] (3-0.3,-0.06) node {\Large{$v$}};
\draw [fill=black] (3-0.43,1.08) node {\Large{$w_e$}};
\draw [fill=black] (4+0.55,2.06) node {\Large{$w_{1,e}$}};
\draw [fill=black] (4+0.55,-0.06) node {\Large{$w_{2,e}$}};
\end{scriptsize}
\end{tikzpicture}    
\caption{The transformation of an edge of weight two into a $4$-star}
\label{fig 14}
\end{figure}
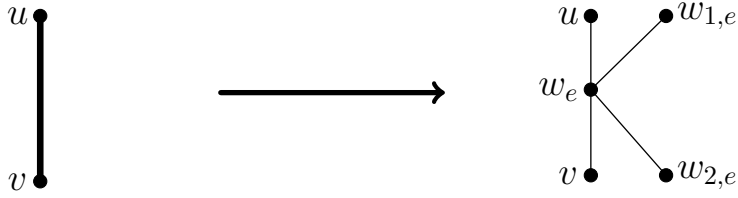

\begin{definition}[The graphs $G'_{ld, i}$ and $G''_{ld, i}$]\label{def:graph3}
For $i \in \{1,2\}$, define the graph $G'_{ld, i}$ from $G_{ld, i}$ as follows: for every edge $e = uv$ of $G_{ld, i}$ of weight two, delete $e$ and add three vertices $w_e, w_{1,e}, w_{2,e}$ together with the edges $w_e v, w_e u, w_e w_{1,e}, w_e w_{2,e}$.

By definition, the distance between every pair of vertices of degree three or four in $G'_{ld, i}$ is at least $\tfrac{\log^2 n  - 1092}{3\cdot 52}$. Construct the graph $G''_{ld, i}$ by deleting every edge of weight one in $G'_{ld, i}$ with endvertices at distance exactly $\left\lfloor\tfrac{\log^2 n  - 1092}{2\cdot 3\cdot 52}\right\rfloor - 2$ and $\left\lfloor\tfrac{\log^2 n  - 1092}{2\cdot 3\cdot 52}\right\rfloor - 1$ from a vertex of degree three or four in $G'_{ld,i}$.
\end{definition}

Roughly speaking, the construction of $G'_{ld, i}$ replaces every edge of weight two in $G_{ld, i}$ by a $4$-star (see Figure~\ref{fig 14}), while the construction of $G''_{ld, i}$ ensures that every connected component contains at most one vertex of degree more than two.

\begin{observation}\label{equivalence}
Given the skeleton of the graph $G'_{ld, i}$, there is a unique weighted unlabeled graph $H = Sk(G_{ld, i})$ such that $Sk(G'_{ld, i})$ is obtained from $H$ by the operation from Definition~\ref{def:graph3}. 
\end{observation}
\begin{proof}
Consider a vertex of degree four in $G'_{ld, i}$ and denote its four neighbors by $w_1, w_2, w_3, w_4$. We consider three cases.
\begin{itemize}
    \item If two of $w_1, w_2, w_3, w_4$, say $w_1$ and $w_2$, are of degree at least two, then the star came from an edge $w_1w_2$ of weight two.
    \item If only one of $w_1, w_2, w_3, w_4$ is of degree at least two, say $w_1$, then the star came from an edge of weight two with endvertices $w_1$ and some leaf in $G_{ld, i}$. Since skeletons are unlabeled graphs, this leaf may be an arbitrary vertex among $w_2, w_3$ or $w_4$.
    \item If none of $w_1, w_2, w_3, w_4$ is of degree at least two, then the star came from an isolated edge in $G_{ld, i}$ of weight two.
\end{itemize}
Finally, there is a unique way to reconstruct $G'_{ld,i}$ in each of the cases, as desired.
\end{proof}

\begin{observation}\label{5.3}
For every large enough $n$, the number of deleted edges of $G'_{ld, i}$ in the construction of $G''_{ld, i}$ is at most $\tfrac{624 n}{\log^2 n}$.
\end{observation}
\begin{proof}
For every edge in $G'_{ld, i}$ to be deleted in the construction of $G''_{ld, i}$, associate to it the shortest path from some of its endvertices to a vertex of degree three or four. Notice that these paths are well defined and edge-disjoint since, first, in $G'_{ld, i}$ all pairs of vertices of degree three or four are at distance at least $\tfrac{\log^2 n  - 1092}{3\cdot 52}$, and second, each of these paths has length exactly $\left\lfloor\tfrac{\log^2 n  - 1092}{2\cdot 3\cdot 52}\right\rfloor - 1$ (the edge to be deleted is counted as an edge of the path). We conclude that for every large enough $n$, the total number of deleted edges must be at most 
\begin{equation*}
    \dfrac{3n/2}{\left\lfloor\tfrac{\log^2 n  - 1092}{2\cdot 3\cdot 52}\right\rfloor-1}\le \dfrac{2\cdot 3\cdot 52\cdot 2n}{\log^2 n} = \dfrac{624 n}{\log^2 n}.
\end{equation*}
The observation is proved.
\end{proof}

\begin{corollary}\label{cor 5.3.75}
The number of possible skeletons for the graph $G''_{ld, i}$ is $\exp(o(n))$.
\end{corollary}
\begin{proof}
First, recall that all connected components in $G''_{ld, i}$ contain at most one vertex of degree three or four. Moreover, by construction $Sk(G''_{ld,i})$ is the union of a forest $F$ on $N\le n$ vertices, and a 2-regular graph $H$.
By Lemma~\ref{stars} with $M = 4$, there are $\exp(o(N)) = \exp(o(n))$ choices for $F$. Moreover, Theorem~\ref{Hardy - Ramanujan} implies that there are $\exp(\Theta(\sqrt{n-N})) = \exp(o(n))$ choices for $H$. Since $Sk(G''_{ld,i})$ is determined by the pair $(F,H)$ (which can also be chosen in $\exp(o(n))$ ways), the corollary follows.
\end{proof}

By applying the reverse transformation of the graph $G''_{ld,i}$ of $4$-stars into edges of weight two (see Figure~\ref{fig 14}), we conclude that the connected components in the skeleton of the graph $G_{ld, i}$ can be:
\begin{itemize}
    \item paths with at most one edge of weight two,
    \item cycles with edges of weight one,
    \item subdivisions of 3-stars with edges of weight one, and
    \item at most $\tfrac{624 n}{\log^2 n} + 12$ more edges due to Lemma~\ref{trivial lem 5.1} and Observation~\ref{5.3}.
\end{itemize}

We will base our first moment computation on this decomposition. Since we are in fact interested in counting graphs containing a bisection of size $\beta n$ and not bisections themselves, we will divide by a factor, which ensures that the same graph is not counted separately for too many bisections. In the next lemma (that might be of independent interest), let $H$ be a bipartite graph with parts $H_1$ and $H_2$ of maximal degree two.

\begin{lemma}\label{anticlique}
$H$ contains an independent set $I$, which contains at least $\left\lceil\tfrac{|V(H_1)|}{2}\right\rceil - 1$ vertices in $H_1$ and at least $\left\lceil\tfrac{|V(H_2)|}{2}\right\rceil - 1$ vertices in $H_2$.
\end{lemma}
\begin{proof}
Since isolated vertices of $H$ may always be added to any independent set, suppose without loss of generality that there are none. The components of the graph $H$ are paths and even cycles. Let $p_1, p_2, \dots, p_{k'}$ be the paths of even length containing one more vertex of $H_1$ than of $H_2$ with $2\le |p_1|\le |p_2|\le \dots \le |p_{k'}|$ and also let $q_1, q_2, \dots, q_{k''}$ be the paths of even length containing one more vertex of $H_2$ than of $H_1$ with $2\le |q_1|\le |q_2|\le \dots \le |q_{k''}|$. Suppose without loss of generality that $k'\ge k''$. Since the property from the statement of the lemma is decreasing, one may add edges to $H$ to form cycles from the paths of odd length and also to form one long cycle by consecutively joining the endvertices of the paths $p_{k'-k''+1}, q_1, p_{k'-k''+2}, q_2, \dots, p_{k'}, q_{k''}, p_{k'-k''+1}$ in this order while keeping the sets $H_1$ and $H_2$ independent (see Figure~\ref{new fig 2}). Let the new bipartite graph be called $G$ with parts $H_1$ and $H_2$ and define $k = k'-k''$. Let $I$ be the empty set in the beginning. We consider two cases:

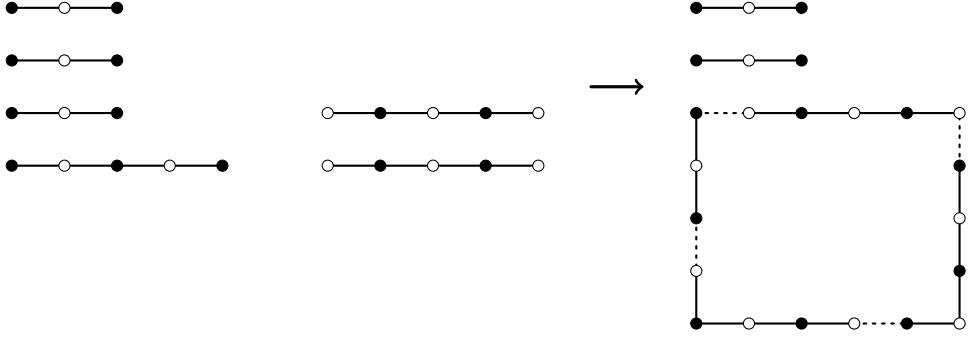
\begin{figure}
\centering
\definecolor{ffffff}{rgb}{1,1,1}
\begin{tikzpicture}[scale=0.7,line cap=round,line join=round,x=1cm,y=1cm]
\clip(-11,-2.2) rectangle (10,4.2);
\draw [line width=0.8pt] (-10,3)-- (-9,3);
\draw [line width=0.8pt] (-9,3)-- (-8,3);
\draw [line width=0.8pt] (-10,4)-- (-9,4);
\draw [line width=0.8pt] (-10,2)-- (-9,2);
\draw [line width=0.8pt] (-9,2)-- (-8,2);
\draw [line width=0.8pt] (-10,1)-- (-9,1);
\draw [line width=0.8pt] (-9,1)-- (-8,1);
\draw [line width=0.8pt] (-8,1)-- (-7,1);
\draw [line width=0.8pt] (-7,1)-- (-6,1);
\draw [line width=0.8pt] (-4,2)-- (-3,2);
\draw [line width=0.8pt] (-3,2)-- (-2,2);
\draw [line width=0.8pt] (-2,2)-- (-1,2);
\draw [line width=0.8pt] (-1,2)-- (0,2);
\draw [line width=0.8pt] (-4,1)-- (-3,1);
\draw [line width=0.8pt] (-3,1)-- (-2,1);
\draw [line width=0.8pt] (-2,1)-- (-1,1);
\draw [line width=0.8pt] (-1,1)-- (0,1);
\draw [line width=0.8pt] (-9,4)-- (-8,4);
\draw [->,line width=1.2pt] (1,2.5) -- (2,2.5);
\draw [line width=0.8pt] (3,4)-- (4,4);
\draw [line width=0.8pt] (4,4)-- (5,4);
\draw [line width=0.8pt] (3,3)-- (4,3);
\draw [line width=0.8pt] (4,3)-- (5,3);
\draw [line width=0.8pt] (3,2)-- (3,1);
\draw [line width=0.8pt] (3,1)-- (3,0);
\draw [line width=0.8pt,dash pattern=on 1pt off 2.5pt] (3,0)-- (3,-1);
\draw [line width=0.8pt] (3,-1)-- (3,-2);
\draw [line width=0.8pt] (3,-2)-- (4,-2);
\draw [line width=0.8pt] (4,-2)-- (5,-2);
\draw [line width=0.8pt] (5,-2)-- (6,-2);
\draw [line width=0.8pt,dash pattern=on 1pt off 2.5pt] (6,-2)-- (7,-2);
\draw [line width=0.8pt,dash pattern=on 1pt off 2.5pt] (3,2)-- (4,2);
\draw [line width=0.8pt] (4,2)-- (5,2);
\draw [line width=0.8pt] (5,2)-- (6,2);
\draw [line width=0.8pt] (6,2)-- (7,2);
\draw [line width=0.8pt] (7,2)-- (8,2);
\draw [line width=0.8pt] (7,-2)-- (8,-2);
\draw [line width=0.8pt] (8,-2)-- (8,-1);
\draw [line width=0.8pt] (8,-1)-- (8,0);
\draw [line width=0.8pt] (8,0)-- (8,1);
\draw [line width=0.8pt,dash pattern=on 1pt off 2.5pt] (8,1)-- (8,2);
\begin{scriptsize}
\draw [fill=black] (-10,3) circle (3pt);
\draw [fill=ffffff] (-9,3) circle (3pt);
\draw [fill=black] (-8,3) circle (3pt);
\draw [fill=black] (-10,4) circle (3pt);
\draw [fill=ffffff] (-9,4) circle (3pt);
\draw [fill=black] (-10,2) circle (3pt);
\draw [fill=ffffff] (-9,2) circle (3pt);
\draw [fill=black] (-8,2) circle (3pt);
\draw [fill=black] (-10,1) circle (3pt);
\draw [fill=ffffff] (-9,1) circle (3pt);
\draw [fill=black] (-8,1) circle (3pt);
\draw [fill=ffffff] (-7,1) circle (3pt);
\draw [fill=black] (-6,1) circle (3pt);
\draw [fill=ffffff] (-4,2) circle (3pt);
\draw [fill=black] (-3,2) circle (3pt);
\draw [fill=ffffff] (-2,2) circle (3pt);
\draw [fill=black] (-1,2) circle (3pt);
\draw [fill=ffffff] (0,2) circle (3pt);
\draw [fill=ffffff] (-4,1) circle (3pt);
\draw [fill=black] (-3,1) circle (3pt);
\draw [fill=ffffff] (-2,1) circle (3pt);
\draw [fill=black] (-1,1) circle (3pt);
\draw [fill=ffffff] (0,1) circle (3pt);
\draw [fill=black] (-8,4) circle (3pt);
\draw [fill=black] (3,4) circle (3pt);
\draw [fill=ffffff] (4,4) circle (3pt);
\draw [fill=black] (5,4) circle (3pt);
\draw [fill=black] (3,3) circle (3pt);
\draw [fill=ffffff] (4,3) circle (3pt);
\draw [fill=black] (5,3) circle (3pt);
\draw [fill=black] (3,2) circle (3pt);
\draw [fill=ffffff] (3,1) circle (3pt);
\draw [fill=black] (3,0) circle (3pt);
\draw [fill=ffffff] (3,-1) circle (3pt);
\draw [fill=black] (3,-2) circle (3pt);
\draw [fill=ffffff] (4,-2) circle (3pt);
\draw [fill=black] (5,-2) circle (3pt);
\draw [fill=ffffff] (6,-2) circle (3pt);
\draw [fill=black] (7,-2) circle (3pt);
\draw [fill=ffffff] (4,2) circle (3pt);
\draw [fill=black] (5,2) circle (3pt);
\draw [fill=ffffff] (6,2) circle (3pt);
\draw [fill=black] (7,2) circle (3pt);
\draw [fill=ffffff] (8,2) circle (3pt);
\draw [fill=ffffff] (8,-2) circle (3pt);
\draw [fill=black] (8,-1) circle (3pt);
\draw [fill=ffffff] (8,0) circle (3pt);
\draw [fill=black] (8,1) circle (3pt);
\end{scriptsize}
\end{tikzpicture}
\caption{Joining the paths of even length in the proof of Lemma~\ref{anticlique}. The vertices of $H_1$ are black, and the vertices of $H_2$ are white. The dashed edges are added to form cycles.}
\label{new fig 2}
\end{figure}

\begin{itemize}
    \item The number of vertices of $H_1$ in $\underset{1\le i\le k}{\bigcup} p_i$ is at most $\left\lceil\tfrac{|V(H_1)|}{2}\right\rceil - 1$.
    Add these to $I$ and then start exploring the cycles vertex by vertex in the following way: First, choose a cycle $c = v_0v_1\dots v_{2s-1}v_0$, a vertex (say $v_0$) in $H_1\cap c$ and a direction. Then, start adding to $I$ the vertices $v_2, v_4, \dots $ consecutively. If the vertex $v_{2s-2}$ is added to $I$, and $I$ still contains less than
    $\left\lceil\tfrac{|V(H_1)|}{2}\right\rceil - 1$ vertices of $H_1$, then choose another cycle and iterate. When this number of vertices of $H_1$ in $I$ is reached, then either jump over the next vertex of $H_1$ in the cycle we are just exploring and start in the same way adding vertices of $H_2$ to $I$, if possible, or go to another cycle and do the same (for $H_2$). This ensures that in total the number of vertices from $H_2$ added to $I$ is at least
    \begin{equation*}
        \tfrac{|V(G)|-2}{2} - \left(\left\lceil\tfrac{|V(H_1)|}{2}\right\rceil - 1\right)\ge \left\lceil\tfrac{|V(H_2)|}{2}\right\rceil - 1,
    \end{equation*}
    which is enough to conclude.
    \item The number of vertices of $H_1$ in $\underset{1\le i\le k}{\bigcup} p_i$ is more than $\left\lceil\tfrac{|V(H_1)|}{2}\right\rceil - 1$. Then, start exploring the paths of $p_1, p_2, \dots, p_k$ in this order. For a path $p$, add to $I$ one of the endvertices of $p$ and then continue adding vertices of $H_1$ in order until either we explore $p$ to the other endvertex (in which case we redo the exploration process with the next path) or the number of vertices of $H_1$ in $I$ reaches $\left\lceil\tfrac{|V(H_1)|}{2}\right\rceil - 1$. We add the vertices of $H_2$ outside the neighborhood of $I\cap H_1$ in $G$ to $I$ and prove that this independent set $I$ satisfies the condition of the lemma.
    Let $\ell_1, \ell_2, \dots, \ell_k$ be the lengths of the paths $p_1, p_2, \dots, p_k$ respectively, and suppose that there are exactly $\ell$ edges in cycles. Suppose that, for some $i\in [k]$, for every $j\le i-1$, all vertices from $H_1$ in the path $p_j$ are in $I$, and there are $b$ vertices from $H_1$ in the path $p_i$ in $I$. On the one hand, the number of vertices in $H_2$ outside the neighborhood of $H_1\cap I$ is
    exactly $\tfrac{1}{2}\left(\ell + \ell_{i+1} +\dots \ell_k + \max\{0, \ell_i - 2b\}\right)$. On the other hand, since at most half of the vertices in $H_1$ have been added to $I$,
    \begin{equation*}
        \sum_{1\le j\le i-1} \left(\dfrac{\ell_j}{2} + 1\right) + b \le \dfrac{1}{2}\left(\dfrac{\ell}{2} + \sum_{1\le j\le k} \left(\dfrac{\ell_j}{2} + 1\right)\right).
    \end{equation*}
    
    We deduce that 
    \begin{equation*}
        \sum_{1\le j\le i-1} \left(\ell_j+ 2\right) + 2b \le \ell + \sum_{i+1\le j\le k} (\ell_j+ 2) + \ell_i+2-2b.
    \end{equation*}
    
    Therefore 
    \begin{align*}
        & \dfrac{1}{2}(\ell + \ell_{i+1} +\dots + \ell_k + \max\{0, \ell_i - 2b\})\\ 
        \ge \hspace{0.4em}
        & \dfrac{1}{2}(\ell + \ell_{i+1} +\dots + \ell_k + \ell_i - 2b)\\
        \ge \hspace{0.4em}
        & \dfrac{1}{2}(\ell_1+\dots+\ell_{i-1}+2b+2(i-1)-2(k-i+1))\\
        = \hspace{0.4em}
        & \dfrac{1}{2}(\ell_1+\dots+\ell_{i-1}+2b+2(i-(k-i)-2)).
    \end{align*}
    
    However, there are at most $\tfrac{1}{2}(\ell_1+\dots+\ell_{i-1}+2b)$ vertices in the neighborhood of $H_1 \cap I$ that belong to $H_2$. Moreover, $i\ge k-i$, since by definition the number of vertices from $H_1$ in the union of the paths $p_1, p_2, \dots, p_i$ plus one must be at least as large as the number of vertices from $H_1$ in the union of the paths $p_{i+1}, \dots, p_k$ (recall that $2\le |p_1|\le |p_2|\le \dots \le |p_k|$). Therefore, the number of vertices from $H_2$ added to $I$ is at least the number of vertices from $H_2$ in the neighborhood of $H_1\cap I$ minus two.
\end{itemize}
This finishes the proof of the lemma.
\end{proof}

We will be interested in an upper bound on the number of 3-regular graphs, containing a bisection of size $\beta n$, coming from a first moment computation. As we shall see below, as in Section~\ref{section 3}, only the exponential order in this upper bound will be of any importance. Therefore, by Corollary~\ref{cor 5.3.75} it is sufficient to count the extensions of the skeleton of $G''_{ld, i}$, which contains the largest number of them.\par

To characterize this particular skeleton, we do a new transformation, which was already presented in Section~\ref{section 3}. We merge all cycles in this unlabeled graph into one very large cycle. This transformation changes neither the number of vertices of any degree in $G''_{ld, i}$ nor the number of edges in $G''_{ld, i}$. We conclude that the number of extensions of the newly obtained graph to $G_{3,i}$ remains the same as the number of extensions of the original graph $G''_{ld,i}$ itself. Moreover, as already observed, a constant number of edges contributes $\exp(o(n))$ to the counting given by the first moment. Therefore, after merging the cycles in $G''_{ld, i}$ one may delete one edge from the obtained very long cycle, and thus we are left only with paths and subdivisions of stars. By abuse of notation, we denote by $G''_{ld, i}$ the graph after the transformation as well.\par

Let $\beta'_1 n$ (respectively $\beta'_2 n$) be the number of vertices of degree two in $G[U_1]$ (respectively in $G[U_2]$), which subdivide the paths and the 3-stars in $G''_{ld, 1}$ (respectively in $G''_{ld, 2}$), and let $\beta n$ be the total number of vertices of degree two on both sides. Note that $\beta, \beta'_1$ and $\beta'_2$ are functions of $n$ with $\max \{(\beta - \beta'_1)n, (\beta - \beta'_2)n\} \le \tfrac{624 n}{\log^2 n} + 12$. For $i,j,\ell\in \mathbb N$ with $\ell\le j\le i$, let $x'_i$ be the number of paths containing no edge of weight two of length $i$ in $G''_{ld, 1}$, $x'_{j,i}$ be the number of paths with an edge of weight two in position $j\le \tfrac{i+1}{2}$ in $G''_{ld, 1}$ and of length $i$, and $y'_{i,j,l}$ be the number of stars with branches of length $i, j, \ell$ in $G''_{ld, 1}$. Let also $x''_i$ be the number of paths containing no edge of weight two of length $i$ in $G''_{ld, 2}$, $x''_{j,i}$ be the number of paths with edge of weight two in position $j\le \tfrac{i+1}{2}$ in $G''_{ld, 2}$ and of length $i$, and $y''_{i,j,l}$ be the number of stars with branches of length $i, j, l$ in $G''_{ld, 1}$.\par

Now, for any given $\beta$, we give an upper bound, up to $\exp(o(n))$, of the number of graphs that contain a bisection of type two of size $\beta n$. To this end, we find the skeleton of $G''_{ld, i}$, whose subdivision admits a maximal number of extensions, up to a $\exp(o(n))-$factor, to $G[U_i]$ by maximizing the first moment. We now explain the factors appearing in the counting in this first moment, as in the previous section.

\begin{enumerate}
    \item Choose the $n/2$ labels of the vertices in $U_1$ (and respectively in $U_2$) in 
    \begin{equation*}
        \binom{n}{n/2}
    \end{equation*}
    ways.
    \item Choose $\beta n$ labels for the vertices in $U_1$, which are going to be vertices of degree two in $G[U_1]$, and another $\beta n$ labels for the vertices in $U_2$, which are going to be vertices of degree two in $G[U_2]$, in
    \begin{equation*}
        \binom{n/2}{\beta n}^2 = \exp(o(n)) \binom{n/2}{\beta_1 n}\binom{n/2}{\beta_2 n}
    \end{equation*}
    ways. Furthermore, choose the labels of the vertices of degree two in the subdivision of $G''_{ld, 1}$ in $\binom{\beta n}{\beta'_1 n} = \exp(o(n))$ ways and choose labels of the the vertices in the subdivision of $G''_{ld, 2}$ in $\binom{\beta n}{\beta'_2 n} = \exp(o(n))$ ways. Here we use that for every constant $c > 0$, by Stirling's formula 
    \begin{equation*}
       \left\lfloor\dfrac{cn}{\log^2 n}\right\rfloor! = o\left(\left(\dfrac{cn}{\log^2 n}\right)^{\frac{cn}{\log^2 n}}\right) = \exp(o(n)).
    \end{equation*}
    \item Distribute the labels of the vertices of degree two on the edges of the unlabeled graph $G''_{ld, 1}$ in 
    \begin{equation*}
        ((\beta-\beta'_1)n)! (\beta'_1 n)! = \exp(o(n))(\beta'_1 n)!
    \end{equation*}
    ways. Also, distribute the labels of the vertices of degree two on the edges of the unlabeled graph $G''_{ld, 2}$ in 
    \begin{equation*}
        ((\beta-\beta'_2)n)! (\beta'_2 n)! = \exp(o(n))(\beta'_2 n)!
    \end{equation*}
    ways. 
    \item Divide by the product of the sizes of the automorphism groups of $G''_{ld, 1}$ and $G''_{ld, 2}$, which is given by
    \begin{align*}
        \prod_{i\ge 1} 2^{x'_i}(x'_i)!  &2^{x''_i}(x''_i)! \prod_{i\ge 1, (i+1)/2\ge j\ge 1} (x'_{j,i})!(x''_{j,i})! \prod_{i\ge 1, i\text{ odd}} 2^{x'_{\frac{i+1}{2}, i}} \prod_{i\ge 1, i\text{ odd}} 2^{x''_{\frac{i+1}{2}, i}}\\
        & \prod_{l\ge j\ge i\ge 1} (y'_{i,j,l})! \prod_{l > j\ge 1} 2^{y'_{j,j,l}} \prod_{j>i\ge 1} 2^{y'_{i,j,j}} \prod_{j\ge 1} 6^{y'_{j,j,j}}\\
        & \prod_{\ell\ge j\ge i\ge 1} (y''_{i,j,l})! \prod_{l > j\ge 1} 2^{y''_{j,j,l}} \prod_{j>i\ge 1} 2^{y''_{i,j,j}} \prod_{j\ge 1} 6^{y''_{j,j,j}}.
    \end{align*}
    \item\label{five} For $i=1,2$, we introduce the parameters $t_i$, the number of connected components in $G''_{ld, i}$, $k_i$, the number of edges of weight two in $G''_{ld, i}$, and $l_i$, the number of leaves in $G''_{ld, i}$. Thus we have
    \begin{align*}
        & \sum_{i\ge 1, (i+1)/2\ge j\ge 1} x'_{j,i} = k_1 n;\\ 
        & \sum_{i\ge 1, (i+1)/2\ge j\ge 1} x''_{j,i} = k_2 n;\\
        & \sum_{i\ge 1} x'_i + \sum_{i\ge 1, (i+1)/2\ge j\ge 1} x'_{j,i} + \sum_{l\ge j\ge i\ge 1} y'_{i,j,l} = t_1 n;\\
        & \sum_{i\ge 1} x''_i + \sum_{i\ge 1, (i+1)/2\ge j\ge 1} x''_{j,i} + \sum_{l\ge j\ge i\ge 1} y''_{i,j,l} = t_2 n;\\
        & \sum_{i\ge 1} ix'_i + \sum_{i\ge 1, (i+1)/2\ge j\ge 1} (i+1)x'_{j,i} + \sum_{l\ge j\ge i\ge 1} (i+j+l)y'_{i,j,l} = \beta'_1 n;\\
        & \sum_{i\ge 1} ix''_i + \sum_{i\ge 1, (i+1)/2\ge j\ge 1} (i+1)x''_{j,i} + \sum_{l\ge j\ge i\ge 1} (i+j+l)y''_{i,j,l} = \beta'_2 n;\\
        & \sum_{i\ge 1} 2x'_i + \sum_{i\ge 1, (i+1)/2\ge j\ge 1} 2x'_{j,i} + \sum_{l\ge j\ge i\ge 1} 3y'_{i,j,l} = \ell_1 n;\\
        & \sum_{i\ge 1} 2x''_i + \sum_{i\ge 1, (i+1)/2\ge j\ge 1} 2x''_{j,i} + \sum_{l\ge j\ge i\ge 1} 3y''_{i,j,l} = \ell_2 n.
    \end{align*}

\item Express the number of vertices of degree three in $G[U_1]$ participating in $G''_{ld, 1}$ (the same computation holds for $G''_{ld, 2}$ and $G[U_2]$, respectively). Note that this number is equal to the number of vertices of degree two in the subdivision of $G''_{ld, 1}$, plus a ``correction'' of one vertex per connected component without edge of weight two. In total this yields $(\beta'_1 + t_1 - k_1)n$ vertices in $G''_{ld, 1}$ and $(\beta'_2 + t_2 - k_2)n$ vertices in $G''_{ld, 2}$.

Thus, choose the labels of these vertices on the two sides in
\begin{equation*}
    \binom{(0.5-\beta)n}{(\beta'_1 + t_1 - k_1)n} \binom{(0.5-\beta)n}{(\beta'_2 + t_2 - k_2)n} = \exp(o(n)) \binom{(0.5-\beta'_1)n}{(\beta'_1 + t_1 - k_1)n} \binom{(0.5-\beta'_2)n}{(\beta'_2 + t_2 - k_2)n}
\end{equation*}
ways and distribute them in the skeleton in
\begin{equation*}
    ((\beta'_1 + t_1 - k_1)n)! ((\beta'_2 + t_2 - k_2)n)!
\end{equation*}
ways.
\item Choose the remaining edges completing the graphs $G[U_1]$ and $G[U_2]$ in 
\begin{align*}
& \left(\dfrac{((1.5 - 4\beta'_1 - \beta + 2k)n)!!}{2^{(\ell_1+\beta-\beta'_1)n}6^{(0.5 - \beta - \beta'_1 - (t_1-k_1))n}}\right) \left(\dfrac{((1.5 - 4\beta'_2 - \beta + 2k)n)!!}{2^{(\ell_2+\beta-\beta'_2)n}6^{(0.5 - \beta - \beta'_2 - (t_2-k_2))n}}\right)\\ = 
& \exp(o(n)) \left(\dfrac{((1.5 - 5\beta'_1 + 2k_1)n)!!}{2^{\ell_1 n}6^{(0.5 - 2\beta'_1 - (t_1-k_1))n}}\right) \left(\dfrac{((1.5 - 5\beta'_2 + 2k_2)n)!!}{2^{\ell_2 n}6^{(0.5 - 2\beta'_2 - (t_2-k_2))n}}\right)
\end{align*}
ways.
\item Multiply by $(\beta n)! = \exp(o(n)) \sqrt{(\beta'_1 n)!(\beta'_2 n)!}$ for the matching between the two parts.
\item Multiply by $6^n$ to count configurations instead of graphs.
\item Divide by the total number of $(3n-1)!!$ configurations to transform the counting into a probabilistic upper bound.
\item Up to now, we counted only bisections and did not take into consideration the fact that these may come from the same graph. Now, we make one step further, partially regrouping the bisections according to the 3-regular graph $G$ they come from.\par
Consider the vertices of degree two in $G[U_1]$ and $G[U_2]$, which have a neighbor of degree two on the same side. These vertices are exactly the ones which subdivide in $G[U_1]$ and $G[U_2]$ the edges of weight at least two in $G^+_{3,1}$ and $G^+_{3,2}$, respectively. We concentrate on the pairs of vertices subdividing edges of weight two in $G''_{ld, 1}$ and $G''_{ld, 2}$. Contracting each of these pairs to a single vertex and considering the graph induced by the edges in the cut, which are incident to at least one vertex in the contracted pairs, leads to a bipartite graph of maximal degree two, see Figure~\ref{fig 15}. We also know the number of vertices in each part of this graph, these are $k_1 n$ and $k_2 n$, respectively. By Lemma~\ref{anticlique} one may form an independent set in this graph with $\left\lceil \tfrac{k_1 n}{2}\right\rceil - 1$ vertices from the first part and $\left\lceil \tfrac{k_2 n}{2}\right\rceil - 1$ vertices from the second part. Now, notice that to form a bisection of $G$, it is sufficient to put the vertices in any $\left\lceil \tfrac{k_1 n}{2}\right\rceil - 1$ of the contracted pairs into $U_1$ and the remaining pairs into $U_2$ to form a minimal bisection of $G$. This makes a total of
\begin{equation*}
    \dbinom{\lceil k_1 n/2\rceil + \lceil k_2 n/2\rceil - 2}{\lceil k_1 n/2\rceil - 1}
\end{equation*}
choices, which correspond to different bisections of the same size coming from the same graph.
\end{enumerate}

\begin{figure}
\centering
\begin{tikzpicture}[scale = 1.75, line cap=round,line join=round,x=1cm,y=1cm]
\clip(-5.5367011984729775,-2.177318904147614) rectangle (4.561266072285648,2.5876594017416594);
\draw [rotate around={90:(-4.5,0.25)},line width=0.5pt] (-4.5,0.25) ellipse (1.5247548783981972cm and 0.8731422788979463cm);
\draw [line width=0.5pt] (-4.2,1.4)-- (-4.2,1.2);
\draw [line width=0.5pt] (-4.2,0.8)-- (-4.2,0.6);
\draw [line width=0.5pt] (-4.2,0.2)-- (-4.2,0);
\draw [line width=0.5pt] (-4.2,-0.4)-- (-4.2,-0.6);
\draw [line width=0.5pt] (-4.2,1.4)-- (-4.5,1.5);
\draw [line width=0.5pt] (-4.2,1.2)-- (-4.8,1.2);
\draw [line width=0.5pt] (-4.2,0.8)-- (-5,1);
\draw [line width=0.5pt] (-4.2,0.6)-- (-5,0.6);
\draw [line width=0.5pt] (-5,0.4)-- (-4.2,0.2);
\draw [line width=0.5pt] (-5,0)-- (-4.2,0);
\draw [line width=0.5pt] (-4.8,-0.4)-- (-4.2,-0.4);
\draw [line width=0.5pt] (-4.2,-0.6)-- (-4.5,-1);
\draw [rotate around={90:(-2.5,0.25)},line width=0.5pt] (-2.5,0.25) ellipse (1.5247548783981957cm and 0.8731422788979454cm);
\draw [line width=0.5pt] (-3,1)-- (-3,0.8);
\draw [line width=0.5pt] (-3,0.8)-- (-2.2,0.8);
\draw [line width=0.5pt] (-3,1)-- (-2.5,1.5);
\draw [line width=0.5pt] (-3,0.4)-- (-3,0.2);
\draw [line width=0.5pt] (-3,0.2)-- (-2.2,0);
\draw [line width=0.5pt] (-3,0.4)-- (-2.2,0.6);
\draw [line width=0.5pt] (-2.5,-1)-- (-3,-0.4);
\draw [line width=0.5pt] (-3,-0.4)-- (-3,-0.2);
\draw [line width=0.5pt] (-3,-0.2)-- (-2.2,-0.2);
\draw [line width=0.5pt] (-4.2,1.2)-- (-3,1);
\draw [line width=0.5pt] (-3,0.8)-- (-4.2,0.8);
\draw [line width=0.5pt] (-4.2,0.6)-- (-3,0.4);
\draw [line width=0.5pt] (-3,0.2)-- (-4.2,1.4);
\draw [line width=0.5pt] (-4.2,0)-- (-3,-0.2);
\draw [line width=0.5pt] (-3,-0.4)-- (-4.2,-0.4);
\draw [->,line width=1pt] (-1.4,0.4) -- (-0.4,0.4);
\draw [rotate around={90:(0.5,0.25)},line width=0.5pt] (0.5,0.25) ellipse (1.384329796997689cm and 0.594868882070379cm);
\draw [rotate around={90:(2,0.25)},line width=0.5pt] (2,0.25) ellipse (1.0405694150420952cm and 0.7213076372263418cm);
\draw [rotate around={90:(3.5,0.25)},line width=0.5pt] (3.5,0.25) ellipse (1.3843297969976904cm and 0.5948688820703797cm);
\draw [line width=0.5pt] (0.5,1.5)-- (1.7,1);
\draw [line width=0.5pt] (1.7,1)-- (1.7,0.8);
\draw [line width=0.5pt] (1.7,0.8)-- (0.4,1.2);
\draw [line width=0.5pt] (0.4,1)-- (1.7,0.6);
\draw [line width=0.5pt] (1.7,0.6)-- (1.7,0.4);
\draw [line width=0.5pt] (1.7,0.4)-- (0.4,0.5);
\draw [line width=0.5pt] (0.4,0.1)-- (1.7,0.2);
\draw [line width=0.5pt] (1.7,0.2)-- (1.7,0);
\draw [line width=0.5pt] (1.7,0)-- (0.4,-0.2);
\draw [line width=0.5pt] (1.7,-0.2)-- (1.7,-0.4);
\draw [line width=0.5pt] (1.7,-0.4)-- (0.5,-1);
\draw [line width=0.5pt] (1.7,-0.2)-- (0.4,-0.6);
\draw [line width=0.5pt] (3.5,1.5)-- (2.2,1);
\draw [line width=0.5pt] (2.2,1)-- (2.2,0.8);
\draw [line width=0.5pt] (2.2,0.8)-- (3.6,0.8);
\draw [line width=0.5pt] (3.8,0.6)-- (2.2,0.4);
\draw [line width=0.5pt] (2.2,0.4)-- (2.2,0.2);
\draw [line width=0.5pt] (2.2,0.2)-- (3.8,0);
\draw [line width=0.5pt] (3.6,-0.2)-- (2.2,-0.2);
\draw [line width=0.5pt] (2.2,-0.2)-- (2.2,-0.4);
\draw [line width=0.5pt] (2.2,-0.4)-- (3.5,-1);
\draw [rotate around={90:(1.7,0.9)},line width=0.5pt] (1.7,0.9) ellipse (0.16180339887498776cm and 0.12720196495140557cm);
\draw [rotate around={90:(2.2,0.9)},line width=0.5pt] (2.2,0.9) ellipse (0.16180339887499304cm and 0.12720196495140973cm);
\draw [rotate around={90:(1.7,0.5)},line width=0.5pt] (1.7,0.5) ellipse (0.16180339887499304cm and 0.12720196495140973cm);
\draw [rotate around={90:(2.2,0.3)},line width=0.5pt] (2.2,0.3) ellipse (0.16180339887498238cm and 0.12720196495140135cm);
\draw [rotate around={90:(1.7,0.1)},line width=0.5pt] (1.7,0.1) ellipse (0.16180339887499032cm and 0.1272019649514076cm);
\draw [rotate around={90:(2.2,-0.3)},line width=0.5pt] (2.2,-0.3) ellipse (0.16180339887498238cm and 0.12720196495140135cm);
\draw [rotate around={90:(1.7,-0.3)},line width=0.5pt] (1.7,-0.3) ellipse (0.16180339887499032cm and 0.1272019649514076cm);
\draw [line width=0.5pt] (1.7,1)-- (2.2,0.2);
\draw [line width=0.5pt] (2.2,0.4)-- (1.7,0.4);
\draw [line width=0.5pt] (1.7,0.6)-- (2.2,0.8);
\draw [line width=0.5pt] (1.7,0.8)-- (2.2,1);
\draw [line width=0.5pt] (1.7,0)-- (2.2,-0.2);
\draw [line width=0.5pt] (1.7,-0.2)-- (2.2,-0.4);
\begin{scriptsize}
\draw [fill=black] (-4.5,1.5) circle (0.8pt);
\draw [fill=black] (-4.5,-1) circle (0.8pt);
\draw [fill=black] (-4.2,1.4) circle (0.8pt);
\draw [fill=black] (-4.2,1.2) circle (0.8pt);
\draw [fill=black] (-4.2,0.8) circle (0.8pt);
\draw [fill=black] (-4.2,0.6) circle (0.8pt);
\draw [fill=black] (-4.2,0.2) circle (0.8pt);
\draw [fill=black] (-4.2,0) circle (0.8pt);
\draw [fill=black] (-4.2,-0.4) circle (0.8pt);
\draw [fill=black] (-4.2,-0.6) circle (0.8pt);
\draw [fill=black] (-4.8,1.2) circle (0.8pt);
\draw [fill=black] (-5,1) circle (0.8pt);
\draw [fill=black] (-5,0.6) circle (0.8pt);
\draw [fill=black] (-5,0.4) circle (0.8pt);
\draw [fill=black] (-5,0) circle (0.8pt);
\draw [fill=black] (-4.8,-0.4) circle (0.8pt);
\draw [fill=black] (-2.5,1.5) circle (0.8pt);
\draw [fill=black] (-2.5,-1) circle (0.8pt);
\draw [fill=black] (-3,1) circle (0.8pt);
\draw [fill=black] (-3,0.8) circle (0.8pt);
\draw [fill=black] (-2.2,0.8) circle (0.8pt);
\draw [fill=black] (-3,0.4) circle (0.8pt);
\draw [fill=black] (-3,0.2) circle (0.8pt);
\draw [fill=black] (-2.2,0) circle (0.8pt);
\draw [fill=black] (-2.2,0.6) circle (0.8pt);
\draw [fill=black] (-3,-0.4) circle (0.8pt);
\draw [fill=black] (-3,-0.2) circle (0.8pt);
\draw [fill=black] (-2.2,-0.2) circle (0.8pt);
\draw [fill=black] (0.5,1.5) circle (0.8pt);
\draw [fill=black] (0.5,-1) circle (0.8pt);
\draw [fill=black] (3.5,1.5) circle (0.8pt);
\draw [fill=black] (3.5,-1) circle (0.8pt);
\draw [fill=black] (1.7,1) circle (0.8pt);
\draw [fill=black] (1.7,0.8) circle (0.8pt);
\draw [fill=black] (0.4,1.2) circle (0.8pt);
\draw [fill=black] (0.4,1) circle (0.8pt);
\draw [fill=black] (1.7,0.6) circle (0.8pt);
\draw [fill=black] (1.7,0.4) circle (0.8pt);
\draw [fill=black] (0.4,0.5) circle (0.8pt);
\draw [fill=black] (0.4,0.1) circle (0.8pt);
\draw [fill=black] (1.7,0.2) circle (0.8pt);
\draw [fill=black] (1.7,0) circle (0.8pt);
\draw [fill=black] (0.4,-0.2) circle (0.8pt);
\draw [fill=black] (1.7,-0.2) circle (0.8pt);
\draw [fill=black] (1.7,-0.4) circle (0.8pt);
\draw [fill=black] (0.4,-0.6) circle (0.8pt);
\draw [fill=black] (2.2,1) circle (0.8pt);
\draw [fill=black] (2.2,0.8) circle (0.8pt);
\draw [fill=black] (3.6,0.8) circle (0.8pt);
\draw [fill=black] (3.8,0.6) circle (0.8pt);
\draw [fill=black] (2.2,0.4) circle (0.8pt);
\draw [fill=black] (2.2,0.2) circle (0.8pt);
\draw [fill=black] (3.8,0) circle (0.8pt);
\draw [fill=black] (3.6,-0.2) circle (0.8pt);
\draw [fill=black] (2.2,-0.2) circle (0.8pt);
\draw [fill=black] (2.2,-0.4) circle (0.8pt);
\end{scriptsize}
\end{tikzpicture}
\caption{The formation of the bipartite graph of maximal degree two from the pairs of vertices subdividing the edges of weight two in $G''_{ld, 1}$ and $G''_{ld, 2}$. The only edges in the cut on the left that are depicted, are the ones that participate in the bipartite graph.}
\label{fig 15}
\end{figure}
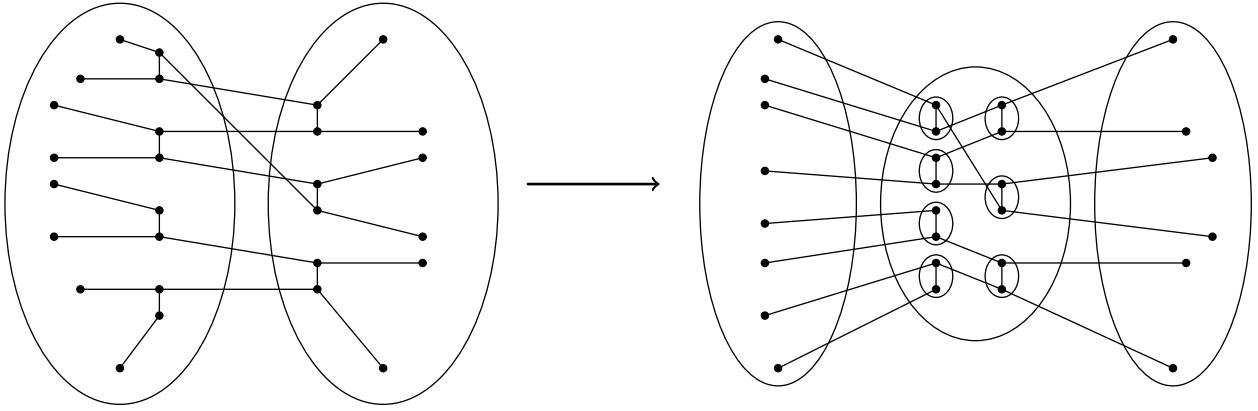

Using that $\ell_1 = 2t_1 + \sum_{l\ge j\ge i\ge 1} y_{i,j,l}$ and $\ell_2 = 2t_2 + \sum_{l\ge j\ge i\ge 1} y_{i,j,l}$, the final formula reads
\begin{align*}
& \exp(o(n)) \binom{n}{n/2} \binom{n/2}{\beta'_1 n} \binom{n/2}{\beta'_2 n} (\beta'_1 n)! (\beta'_2 n)!\\
& \times \left(\prod_{i\ge 1} 2^{x'_i}(x'_i)! \prod_{i\ge 1, (i+1)/2\ge j\ge 1} (x'_{j,i})! \prod_{i\ge 1, i\text{ odd}} 2^{x'_{\frac{i+1}{2}, i}} \prod_{l\ge j\ge i\ge 1} (y'_{i,j,l})! \prod_{l > j\ge 1} 2^{y'_{j,j,l}} \prod_{j>i\ge 1} 2^{y'_{i,j,j}} \prod_{j\ge 1} 6^{y'_{j,j,j}}\right)^{-1}\\
& \times \left(\prod_{i\ge 1} 2^{x''_i}(x''_i)! \prod_{i\ge 1, (i+1)/2\ge j\ge 1} (x''_{j,i})! \prod_{i\ge 1, i\text{ odd}} 2^{x''_{\frac{i+1}{2}, i}} \prod_{l\ge j\ge i\ge 1} (y''_{i,j,l})! \prod_{l > j\ge 1} 2^{y''_{j,j,l}} \prod_{j>i\ge 1} 2^{y''_{i,j,j}} \prod_{j\ge 1} 6^{y''_{j,j,j}}\right)^{-1}\\
& \times \binom{(0.5-\beta'_1)n}{(\beta'_1 + t_1 - k_1)n} \binom{(0.5-\beta'_2)n}{(\beta'_2 + t_2 - k_2)n} ((\beta'_1 + t_1 - k_1)n)!((\beta'_2 + t_2 - k_2)n)!\\
& \times \left(\dfrac{((1.5 - 5\beta'_1 + 2k_1)n)!!} {2^{\ell_1}6^{(0.5 - 2\beta'_1 - (t_1-k_1))n}}\right)
\left(\dfrac{((1.5 - 5\beta'_2 + 2k_2)n)!!} {2^{\ell_2}6^{(0.5 - 2\beta'_2 - (t_2-k_2))n}}\right)\sqrt{(\beta'_1 n)!(\beta'_2 n)!} \dfrac{6^n}{(3n-1)!!} \dbinom{\lfloor k_1 n/2\rfloor + \lfloor k_2 n/2\rfloor - 3}{\lfloor k_1 n/2\rfloor - 1}^{-1}
\end{align*}

Similarly to Section~\ref{section 3}, we use the method of Lagrange multipliers to minimize the product of the order of the automorphism group of $G''_{ld, 1}$ and $2^{\ell_1} = 2^{2t_1 + \sum_{l\ge j\ge i\ge 1} y'_{i,j,l}}$ in the first case,  and the product of the order of the automorphism group of $G''_{ld, 2}$ and $2^{\ell_2} = 2^{2t_2 + \sum_{l\ge j\ge i\ge 1} y''_{i,j,l}}$ in the second case, under the constraints given in point~\ref{five} of the counting procedure above. The calculation is the same for $(x'_i)_{i\ge 1}, (x'_{j,i})_{i\ge 1, (i+1)/2\ge j\ge 1}, (y'_{i,j,l})_{l\ge j\ge i\ge 1}$ and for $(x''_i)_{i\ge 1}, (x''_{j,i})_{i\ge 1, (i+1)/2\ge j\ge 1}, (y''_{i,j,l})_{l\ge j\ge i\ge 1}$, so we only do it for the first set of variables.\par 

Let 
\begin{equation*}
   x_i = \dfrac{x'_i}{n}, x_{j,i} = \dfrac{x'_{j,i}}{n} \text{ and } y_{i,j,l} = \dfrac{y'_{i,j,l}}{n}.
\end{equation*}

Stirling's formula and the constraints in point~\ref{five} above directly imply that one needs to maximize the expression
\begin{align*}
    & \prod_{i\ge 1} 2^{x_i}x_i^{x_i} \prod_{i\ge 1, (i+1)/2\ge j\ge 1} x_{j,i}^{x_{j,i}}\prod_{i\ge 1, i\text{ odd}} 2^{x_{\frac{i+1}{2}, i}} \prod_{\ell\ge j\ge i\ge 1} y_{i,j,l}^{y_{i,j,l}} \prod_{l > j\ge 1} 2^{y_{j,j,l}}\prod_{j>i\ge 1} 2^{y''_{i,j,j}}\prod_{j\ge 1} 6^{y''_{j,j,j}}\\
    &\times \prod_{i\ge 1}2^{2x_i} \prod_{i\ge 1, (i+1)/2\ge j\ge 1} 2^{2x_{j,i}}\prod_{l\ge j\ge i\ge 1} 2^{3y_{i,j,l}}
\end{align*}
as a function of $(x_i)_{i\ge 1}, (x_{j,i})_{i\ge 1, (i+1)/2\ge j\ge 1}, (y_{i,j,l})_{l\ge j\ge i\ge 1}$.\par

Let $f$ be a logarithm of the above expression, that is,
\begin{align*}
& f((x_i)_{i\ge 1}, (x_{j,i})_{i\ge 1, (i+1)/2\ge j\ge 1}, (y_{i,j,l})_{l\ge j\ge i\ge 1}) =\\
& \sum_{i\ge 1} x_i\ln(2) + x_i\ln(x_i) + \sum_{i\ge 1, (i+1)/2\ge j\ge 1} x_{j,i}\ln(x_{j,i}) + \sum_{i\ge 1, i\text{ odd}} x_{\frac{i+1}{2}, i}\ln(2) + \\
& +\sum_{l\ge j\ge i\ge 1} y_{i,j,l}\ln(y_{i,j,l}) + \sum_{l > j\ge 1} y_{j,j,l}\ln(2) + \sum_{j>i\ge 1} y_{i,j,j}\ln(2) + \sum_{j\ge 1} y_{j,j,j}\ln(6) +\\
& +\sum_{i\ge 1} 2x_i\ln(2) + \sum_{i\ge 1, (i+1)/2\ge j\ge 1} 2x_{j,i}\ln(2) + \sum_{l\ge j\ge i\ge 1} 3y_{i,j,l}\ln(2).
\end{align*}

Viewing $t_1, t_2, k_1, k_2, \beta'_1, \beta'_2$ as parameters, note that $f$ contains all terms depending on the variables of $x_i$, $x_{j,i}, y_{i,j,l}$. Moreover, $f$ is a strictly convex and infinitely differentiable function over its feasible domain. Therefore, by (\cite{Lag1}, Theorem 1) and \cite{Lag2} the method of Lagrange multipliers may be applied to find the global minimum of $f$ (which is its unique critical point by convexity) even in infinite-dimension. 

The optimization will be performed under the constraints
\begin{equation*}
\begin{cases}
\Lambda'_1: \sum_{i\ge 1, (i+1)/2\ge j\ge 1} x_{j,i} = k_1;\\
\Lambda'_2: \sum_{i\ge 1} x_i + \sum_{i\ge 1, (i+1)/2\ge j\ge 1} x_{j,i} + \sum_{l\ge j\ge i\ge 1} y_{i,j,l} = t_1;\\
\Lambda'_3: \sum_{i\ge 1} ix_i + \sum_{i\ge 1, (i+1)/2\ge j\ge 1} (i+1)x_{j,i} + \sum_{l\ge j\ge i\ge 1} (i+j+l)y_{i,j,l} = \beta'_1.
\end{cases}
\end{equation*}

The Lagrange objective function is 

\begin{align*}
& F((x_i)_{i\ge 1}, (x_{j,i})_{i\ge 1, (i+1)/2\ge j\ge 1}, (y_{i,j,l})_{l\ge j\ge i\ge 1})\\ =
& f((x_i)_{i\ge 1}, (x_{j,i})_{i\ge 1, (i+1)/2\ge j\ge 1}, (y_{i,j,l})_{l\ge j\ge i\ge 1})\\ -
& \lambda'_1 \left(\sum_{i\ge 1, (i+1)/2\ge j\ge 1} x_{j,i}-k_1\right)\\ -
&\lambda'_2 \left(\sum_{i\ge 1} x_i + \sum_{i\ge 1, (i+1)/2\ge j\ge 1} x_{j,i} + \sum_{l\ge j\ge i\ge 1} y_{i,j,l}-t_1\right)\\ -
& \lambda'_3 \left(\sum_{i\ge 1} ix_i + \sum_{i\ge 1, (i+1)/2\ge j\ge 1} (i+1)x_{j,i} + \sum_{l\ge j\ge i\ge 1} (i+j+l)y_{i,j,l}-\beta'_1\right).
\end{align*}

Direct calculation gives
\begin{align*}
& 3\ln(2) + 1 + \ln(x_i) - \lambda'_2 - i\lambda'_3 = 0 \Rightarrow x_i = \dfrac{\exp(\lambda'_2 + i\lambda'_3)}{8e},\\
\text{ for } j\neq \dfrac{i+1}{2}:\,& 2 \ln(2)+1 + \ln(x_{j,i}) - \lambda'_1 - \lambda'_2 - (i+1)\lambda'_3 = 0 \Rightarrow x_{j,i} = \dfrac{\exp(\lambda'_1 + \lambda'_2 + (i+1)\lambda'_3)}{4e},\\
\text{ for }j = \dfrac{i+1}{2}:\,& 1 + 3\ln(2) + \ln(x_{j,i}) - \lambda'_1 - \lambda'_2 - (i+1)\lambda'_3 = 0 \Rightarrow x_{j,i} = \dfrac{\exp(\lambda'_1 + \lambda'_2 + (i+1)\lambda'_3)}{8e},\\
\text{ for } l> j> i:\,& 1 + 3\ln(2) + \ln(y_{i,j,l}) - \lambda'_2 - (i+j+l)\lambda'_3 = 0 \Rightarrow y_{i,j,l} = \dfrac{\exp(\lambda'_2 + (i+j+l)\lambda'_3)}{8e},\\
\text{ for } l= j> i:\,& 1 + 4\ln(2) + \ln(y_{i,j,j}) - \lambda'_2 - (i+2j)\lambda'_3 = 0 \Rightarrow y_{i,j,j} = \dfrac{\exp(\lambda'_2 + (i+2j)\lambda'_3)}{16e},\\
\text{ for } l> j = i:\,& 1 + 4\ln(2) + \ln(y_{j,j,l}) - \lambda'_2 - (2j+l)\lambda'_3 = 0 \Rightarrow y_{j,j,l} = \dfrac{\exp(\lambda'_2 + (2j+l)\lambda'_3)}{16e},\\
\text{ for } l= j= i:\,& 1 + 3\ln(2) + \ln(6) + \ln(y_{j,j,j}) - \lambda'_2 - 3j\lambda'_3 = 0 \Rightarrow y_{j,j,j} = \dfrac{\exp(\lambda'_2 + 3j\lambda'_3)}{48e}.
\end{align*}

Inversing the system this time is more difficult, but at the same time we would like to have $\beta'_1$ as a parameter and not as a function of $\lambda'_1, \lambda'_2$ and $\lambda'_3$. There is an easy solution -- we keep $\lambda'_1, \lambda'_3$ and $\beta'_1$ as parameters and express $t_1, k_1$ and $\lambda'_2$ as functions of these parameters. This is not difficult since each of $x_i, x_{j, i}$ and $y_{i,j,l}$ can be represented as $\exp(\lambda'_2)g(\lambda'_1, \lambda'_3)$ for some function $g$ depending on the term. An analogous computation may be done for $t_2, k_2, \beta'_2$ with Lagrange multipliers $\lambda''_1, \lambda''_2, \lambda''_3$. A numerical optimization with Maple shows that, optimizing over $\lambda'_1, \lambda''_1, \lambda'_3, \lambda''_3$ in the interval $(\beta'_1, \beta'_2) \in [0.1, 0.103295]$, the maximum is attained at $\lambda'_1 = \lambda''_1 = 0.002428$, $\lambda'_3 = \lambda''_3 = -1.412768$, yielding the value of $0.999996$. We conclude that the proportion of 3-regular graphs of order $n$ with bisection of size at most $0.103295 n$ tends to zero exponentially fast with $n$ (the maple code for the maximization can be found on the authors' webpages). Together with the results of the previous section, the lower bound of Theorem~\ref{thm:main} follows.

\section{An upper bound}\label{section 6}
In this section we show a complementary asymptotically almost sure upper bound on the bisection width of a random 3-regular graph. The approach is very much inspired by the works of~\cite{Gerencser} and~\cite{Lyons}, together with combinatorial observations from the previous sections. In order to explain it, we need a few definitions: first, define a \emph{rooted} graph as a couple $(G, \rho)$, with $G$ some graph and $\rho$ a distinguished vertex of $G$ called \emph{root}. A finite rooted graph $(G, \rho)$ is said to be \emph{uniformly rooted} if $\rho$ is a vertex chosen uniformly at random among the vertices of $G$. Thus, one may consider $(G, \rho)$ as a probability space equipped with the uniform measure. A sequence of uniformly rooted graphs $(G_n, \rho_n)_{n \ge 1}$ is said to \emph{converge locally} (in the Benjamini-Schramm sense) if for every fixed $r > 0$ and every fixed graph $H$, $\lim_{n \to \infty} \mathbb{P}(B_{G_n}(\rho_n, r) \cong H)$ exists. Indeed, local convergence for sequences of graphs is a particular case of weak convergence of probability measures (of course, formally, one needs to embed the sequence of probability spaces into a common probability space and then argue on the basis of it). Since one often comes up with a simple description of the limit measure (let us call it $\mu$) in terms of the empirical frequency of the balls of any radius in a given infinite graph $G_{\infty}$, $\mu$ is usually abusively identified with $G_{\infty}$.\par

The following lemma is well known (and follows from the observation that there are a.a.s.\ at most $o(\log n)$ cycles of any constant length):
\begin{lemma}\label{BSconvergence}
Let $G_n$ be a random $3$-regular graph on $n$ vertices, uniformly rooted at $\rho_n$. The sequence $(G_n, \rho_n)_{n \ge 1}$ converges locally to the infinite $3$-regular tree $T_3$.
\end{lemma}

In order to use the previous lemma, we first define the concept of a \emph{factor of iid process}. Suppose that, for a (possibly infinite) graph $G$, a family of standard normal random variables $(Z_v)_{v\in V(G)}$ is assigned to the vertices of $G$. A \emph{factor of iid process} on a graph $G$ is a family of random variables $(X_v)_{v\in V(G)}$ such that:
\begin{enumerate}
    \item for every $v\in V(G)$, $X_v$ is a measurable function of the random variables $(Z_v)_{v\in V(G)}$, and
    \item the joint distribution of the family $(X_v)_{v\in V(G)}$ is invariant under permutation of the indices induced by the action of any automorphism on $V(G)$.
\end{enumerate}

Let $\mathcal{F}$ be the class of factors of iid processes $(\phi_v)_{v\in V(T_3)}$ on $T_3$ with values $\phi_v \in \{0,1\}$, for which $\mathbb{P}(\phi_{\rho}=0)=\frac{1}{2}$. The proof of Theorem 4.1 in \cite{Lyons} shows that
\begin{equation}\label{weaklimit}
\limsup_{n \to \infty} \{bw(G_n)/|V(G_n)|\} \le \inf_{\phi \in \mathcal{F}} \mathbb{E}(|\{v \sim \rho: \phi_v \neq \phi_{\rho})\}|)/2,
\end{equation}
where $(G_n, \rho_n)$ is any sequence of graphs converging locally to $T_3$. By Lemma~\ref{BSconvergence}, in particular, this also applies to the random 3-regular graph. Hence, for an upper bound on the scaled bisection width it suffices to find a suitable factor of iid process $\phi$ on $T_3$ having values in $\{0,1\}$, and for which the probability that the root obtains value $0$ is $1/2$.\par 

We say that a process $(X_v)_{v\in V(G)}$ is \emph{a linear factor of iid process $(Z_v)_{v\in V(G)}$} on a graph $G$ if there exist $\alpha_0, \alpha_1, \ldots$ such that $X_v = \sum_{u \in V(G)} \alpha_{d_G(v,u)} Z_u$ for all $v \in V(G)$. Note that $X_v$ is always a centered Gaussian random variable. We call a collection of random variables $(Y_v)_{v\in V(G)}$ a \emph{Gaussian process} on $G$ if they are jointly Gaussian and $Y_v$ is centered for every $v\in V(G)$. We say that a Gaussian process $(Y_v)_{v\in V(G)}$ is \emph{invariant} if the distribution of the family $(Y_v)_{v\in V(G)}$ is invariant under the action of any automorphism of $G$ on the index set $V(G)$.
We use the following theorem, proven in~\cite{Gerencser}:

\begin{theorem}\label{thm:Gaussianwave}
For any real number $\lambda$ with $|\lambda| \le 3$ there exists a non-trivial invariant Gaussian process $(Y_v)_{v\in V(T_3)}$ on $T_3$ that satisfies the eigenvector equation with eigenvalue $\lambda$, i.e.\ (with probability 1), for every vertex $v$
it holds
\begin{equation*}
\sum_{u \in N(v)} Y_u = \lambda Y_v.
\end{equation*}
Moreover, the joint distribution of such a process is unique under the additional condition that, for every $v\in V(T_3)$, the variance of $Y_v$ is 1. We will refer to this (essentially unique) process as the Gaussian wave function corresponding to the eigenvalue $\lambda$.
\end{theorem}

Let $((G_n, \rho_n))_{n\ge 1}$ be any sequence of uniformly rooted graphs converging locally to $T_3$. It was shown by Kesten (see~\cite{Kesten}) that the spectrum of the transition operator for the simple random walk on $T_3$ is contained in the interval  $[-2 \sqrt{2}, \, 2 \sqrt{2}]$. Moreover, Cs\'{o}ka, Gerencs\'{e}r, Harangi and Vir\'{a}g showed in~\cite{Gerencser} the following theorem:

\begin{theorem}\label{thm:Gaussianwave2}
For any real number $\lambda$ with $|\lambda| \le 2\sqrt{2}$, there exists a sequence of linear factors of iid processes $(X^{(n)}_v)_{v\in V(G_n)}$ that converges in distribution to the Gaussian wave function $(Y_v)_{v\in V(T_3)}$ corresponding to the eigenvalue $\lambda$.
\end{theorem}

As shown in~\cite{Gerencser}, the joint distribution of a Gaussian process on $T_3$ corresponding to the eigenvalue $\lambda$ with  $|\lambda| \le 2\sqrt{2}$ is completely determined by its covariances $(\covar(Y_u, Y_v))_{u, v \in V(T_3)}$ and satisfies $\covar(Y_u, Y_v) = \sigma_{d_{T_3}(u,v)}$, where 
\begin{equation}\label{covariance}
\sigma_0=1, 3\sigma_1-\lambda \sigma_0=0 \text{ and } 2\sigma_{k+1}-\lambda \sigma_k+\sigma_{k-1}=0 \text{ for every } k \ge 1,
\end{equation}
and so in particular $\sigma_1=\lambda/3$ and $\sigma_2=\frac{\lambda^2-3}{6}$. By the general theory of recursive sequences, the covariances decay exponentially to zero as $d_{T_3}(u,v) \to \infty$. Note that due to this decay of covariances together with Theorem~\ref{thm:Gaussianwave2} we may (and do) choose the process $(X^{(n)}_v)_{v \in V(G_n)}$ to be of the form
\begin{equation}\label{xncut}
    X^{(n)}_v = \sum_{i=0}^{\lfloor\log \log n\rfloor} \sum_{u\in V(G_n),\, d_{G_n}(u,v) = i} \sigma_i Z_u,
\end{equation}
where $(Z_v)_{v\in V(G_n)}$ is a family of independent standard normal random variables.\par
 
We consider the largest eigenvalue for which the theorems apply, that is, $\lambda=2\sqrt{2}$, as already done in~\cite{Lyons}. The idea is that the corresponding eigenvector has positive correlation between neighbors in $T_3$ (since $\lambda > 0$), and thus the cut between the set of vertices $\{v\hspace{0.2em}|\hspace{0.2em}Y_v \ge 0\}$ and the set of vertices $\{v\hspace{0.2em}|\hspace{0.2em}Y_v < 0\}$ has relatively few edges going across. More formally, consider the Gaussian wave $(Y_v)_{v\in V(T_3)}$ on $T_3$ given in Theorem~\ref{thm:Gaussianwave} associated to the eigenvalue $\lambda=2\sqrt{2}$. As before, let $((G_n, \rho_n))_{n\ge 1}$ be any sequence of rooted graphs converging locally to $T_3$. By Theorem~\ref{thm:Gaussianwave2}, there exists a linear factor of iid process $(X^{(n)}_v)_{v\in V(G_n)}$ on $G_n$ that converges in distribution to the Gaussian wave  $(Y_v)_{v\in V(T_3)}$. Since every variable in the Gaussian wave is centered, by setting for each vertex $v$ of $T_3$, $\phi_v=0$ in case $Y_v < 0$ and setting $\phi_v=1$ in case $Y_v \ge 0$, we see that $\phi_v \in \{0,1\}$ and $\mathbb{P}(\phi_{\rho}=0)=\frac12$. Hence, $\phi \in \mathcal{F}$, and by~\eqref{weaklimit}, 
\begin{equation*}
    \limsup_{n \to \infty} \{bw(G_n)/|V(G_n)|\} \le \mathbb{E}(|\{v \sim \rho: \phi_v \neq \phi_{\rho})\}|)/2.
\end{equation*} 
 
Now, consider the sequence $((G_n, \rho_n))_{n\ge 1}$. We first explain the approach of~\cite{Lyons} using the linear factor of iid process $(X^{(n)}_v)_{v\in V(G(n,3))}$ (denoted by $(X^{(n)}_v)$ in the sequel to simplify notation) described above. We consider events regarding cherries. Recall that a cherry $(v_1, v, v_2)$ consists of three vertices, $v, v_1, v_2 \in V(T_3)$, connected by edges $vv_1$ and $vv_2$. Here $v$ is said to be the center of the cherry, and $v_1$ and $v_2$ are its endpoints. We also call a cherry $(v_1, v, v_2)$ a \textit{border cherry} if the sign of $X_v$ is different from both the sign of $X_{v_1}$ and the sign of $X_{v_2}$.

Consider the event that a given cherry with center at $v$ is such that $X^{(n)}_v$ has the same sign as both $X^{(n)}_{v_1}$ and $X^{(n)}_{v_2}$. The probability of this event, in the limit as $n\to \infty$, is given by
\begin{align*}
& \lim_{n\to +\infty}\mathbb P(X^{(n)}_v \ge 0, X^{(n)}_{v_1} \ge 0, X^{(n)}_{v_2} \ge 0)+\mathbb P(X^{(n)}_v < 0, X^{(n)}_{v_1} < 0, X^{(n)}_{v_2} < 0)\\
=\hspace{0.3em}& \lim_{n\to +\infty}2\mathbb P(X^{(n)}_v \ge 0, X^{(n)}_{v_1} \ge 0, X^{(n)}_{v_2} \ge 0) = 2\left(\frac{1}{8} +\frac{1}{4\pi}  \left(2\arcsin{r_1}+\arcsin{r_2}\right) \right),  
\end{align*}
with 
\begin{equation*}
    r_1=\corr(Y_v,Y_{v_1})=\corr(Y_v,Y_{v_2})=\covar(Y_v,Y_{v_2})/\sqrt{\variance(Y_v)\variance(Y_{v_2})}
\end{equation*}
and
\begin{equation*}
    r_2=\corr(Y_{v_1},Y_{v_2})=\covar(Y_{v_1},Y_{v_2})/\sqrt{\variance(Y_{v_1})\variance(Y_{v_2})}.
\end{equation*}

Here, we abusively view $(v_1, v, v_2)$ as a cherry in $T_3$ as well. For the calculation, see for example (6.22) of~\cite{Kendall}, or also Subsection 4.1 in~\cite{Gerencser}. Plugging in the values of $\sigma_1=\frac{2\sqrt{2}}{3}$ and $\sigma_2=\frac56$ we get that
\begin{equation*}
\lim_{n\to +\infty} 2\mathbb P(X^{(n)}_v \ge 0, X^{(n)}_{v_1} \ge 0, X^{(n)}_{v_2} \ge 0) = 2\left(\frac18 +\frac{1}{4\pi}  \left( 2\arcsin{\frac{2\sqrt{2}}{3}}+\arcsin{\frac{5}{6}}\right) \right) \approx 0.798611.
\end{equation*}

Hence, we may conclude that, as $n \to \infty$, with probability tending to $0.798611$ none of the edges of a cherry participates in the cut with parts $\{v\in G(n,3)\hspace{0.2em}|\hspace{0.2em}X^{(n)}_v\ge 0\}$ and $\{v\in G(n,3)\hspace{0.2em}|\hspace{0.2em}X^{(n)}_v < 0\}$. Similarly, as $n \to \infty$, the probability of a cherry $(v_1, v, v_2)$ having exactly one endpoint with a sign, different from the sign of $X^{(n)}_v$, is given by
\begin{align*}
\lim_{n\to +\infty}& \mathbb P(X^{(n)}_v \ge 0, X^{(n)}_{v_1} \ge 0, X^{(n)}_{v_2} < 0)+\mathbb P(X^{(n)}_v < 0, X^{(n)}_{v_1} < 0, X^{(n)}_{v_2} \ge 0)\\ 
+\hspace{0.3em}
&\mathbb P(X^{(n)}_v \ge 0, X^{(n)}_{v_1} < 0, X^{(n)}_{v_2} \ge 0)+\mathbb P(X^{(n)}_v < 0, X^{(n)}_{v_1} \ge 0, X^{(n)}_{v_2} < 0)\\
=\hspace{0.3em}\lim_{n\to +\infty}& 4\mathbb P(X^{(n)}_v \ge 0, X^{(n)}_{v_1} \ge 0, X^{(n)}_{v_2} < 0)\\
=\hspace{3.15em}& 4\left(\frac18 +\frac{1}{4\pi}  \left(\arcsin{\frac{2\sqrt{2}}{3}}+\arcsin{\frac{-2\sqrt{2}}{3}}+\arcsin{\frac{-5}{6}}\right) \right) \approx 0.186429.
\end{align*}

By analogy, the probability to obtain a border cherry, as $n \to \infty$, tends to

\begin{align*}
\lim_{n\to +\infty}& \mathbb P(X^{(n)}_v \ge 0, X^{(n)}_{v_1} < 0, X^{(n)}_{v_2} < 0)+\mathbb P(X^{(n)}_v < 0, X^{(n)}_{v_1} \ge 0, X^{(n)}_{v_2} \ge 0)\\
=\hspace{0.3em}\lim_{n\to +\infty}& 2\mathbb P(X^{(n)}_v \ge 0, X^{(n)}_{v_1} < 0, X^{(n)}_{v_2} < 0)\\
=\hspace{3.15em}& 2\left(\frac18 +\frac{1}{4\pi}  \left(\arcsin{\frac{-2\sqrt{2}}{3}}+\arcsin{\frac{-2\sqrt{2}}{3}}+\arcsin{\frac{5}{6}}\right) \right) \approx 0.0149586.
\end{align*}

We define a partition $(V^{(n)}_1, V^{(n)}_2)$ of $V(G(n,3))$, where $V^{(n)}_1 = \{v\in V(G(n,3))\hspace{0.2em}|\hspace{0.2em}X_v\ge 0\}$ and $V^{(n)}_2 = V(G(n,3))\setminus V^{(n)}_1$.

\begin{observation}[Theorem 4.1 in \cite{Lyons}]\label{ob 6.4}
The minimum bisection in a random $3$-regular graph is of size a.a.s. at most
\begin{equation*}
    (0.186429 + 2 \times 0.0149586) \frac{3n}{4}\approx 0.16226n.
\end{equation*}
\end{observation}
\begin{proof}[Sketch of proof.]
By summing over all $3n$ cherries and taking into account that each edge is counted in four cherries, we get that, as $n \to \infty$, the expected size of the cut $(V^{(n)}_1, V^{(n)}_2)$ is equal to $(0.186429 + 2 \times 0.0149586) \frac{3n}{4}\approx 0.16226n$. While this cut is not necessarily a bisection, by~\eqref{xncut}, $(X^{(n)}_v)$ is a process depending on the set of standard normal variables corresponding to vertices at distance at most $\log \log n$ from $v$ in $G(n,3)$ only. Hence, since two uniformly chosen vertices in $G(n,3)$ are at distance $\Omega(\log n)$ from each other (which follows by a direct computation), by a standard second moment argument, a.a.s.\ $|V^{(n)}_1| - |V^{(n)}_2| = o(n)$. Concentration of the number of border cherries, as well as the number of cherries $(v_1, v, v_2)$ with $X^{(n)}_{v_1}$ and $X^{(n)}_{v_2}$ being of different signs, around their expected values follows from a similar second moment computation.
\end{proof}

We now improve on Lyons' result of Observation~\ref{ob 6.4}. 

\begin{observation}\label{ob 6.5}
The number of vertices $v$ in $G(n,3)$, which are centers of at least two border cherries, is a.a.s. $o(n)$ as $n\to +\infty$.
\end{observation}
\begin{proof}
A center $v$ of at least two border cherries with leaves $v_1,v_2,v_3$ in $G(n,3)$ must be attributed the value $X^{(n)}_v$ whose sign is different from each of $X^{(n)}_{v_1},X^{(n)}_{v_2},X^{(n)}_{v_3}$.
Since the distribution of the process $(X^{(n)}_v)_{v\in V(G(n,3))}$ approximates the one of the Gaussian wave on $T_3$ for the eigenvalue $\lambda = 2\sqrt{2}$, for a uniformly random vertex $v$ of $G(n,3)$ with neighbors $v_1$, $v_2$ and $v_3$ and for every $\varepsilon > 0$, a.a.s.
\begin{equation}\label{eq:approx}
    |2\sqrt{2}X^{(n)}_{v} - (X^{(n)}_{v_1}+X^{(n)}_{v_2}+X^{(n)}_{v_3})|\le \varepsilon.
\end{equation}
Thus, for $v$ to be a center of at least two border cherries, we must have that either $|X^{(n)}_{v}|\le \varepsilon$ or that~\eqref{eq:approx} does not hold. By choosing $\varepsilon$ sufficiently small, the probability of either event may be made arbitrarily close to 0 as $n\to \infty$, which shows that the expected number of centers of at least two border cherries in $G(n,3)$ is $o(n)$. The proof is completed by Markov's inequality.
\end{proof}




\begin{proof}[Proof of the upper bound of Theorem~\ref{thm:main}]
We construct the bipartite graph $H^{(n)}$ with parts $W^{(n)}_1\subseteq V^{(n)}_1$ and $W^{(n)}_2\subseteq V^{(n)}_2$. Its vertices are the centers of border cherries. Its edges are the ones crossing the cut $(V^{(n)}_1, V^{(n)}_2)$, whose endvertices are both centers of border cherries (see Figure~\ref{new fig 3}). Also, recall from the (sketch of) the proof of Observation~\ref{ob 6.4} that a.a.s.\ $|V^{(n)}_1| - |V^{(n)}_2| = o(n)$.

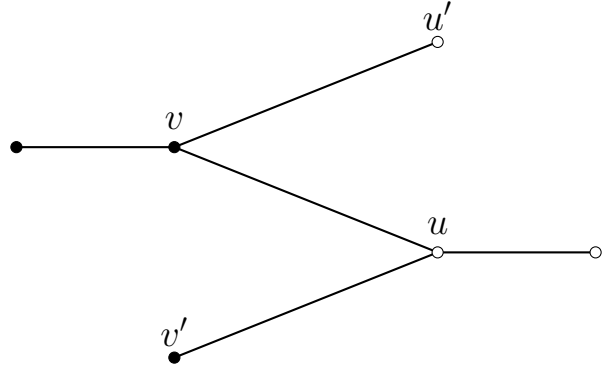
\begin{figure}
\centering
\begin{tikzpicture}[scale=0.7,line cap=round,line join=round,x=1cm,y=1cm]
\clip(0,-5) rectangle (15,3);
\draw [line width=0.8pt] (2,0)-- (5,0);
\draw [line width=0.8pt] (5,0)-- (10,2);
\draw [line width=0.8pt] (5,0)-- (10,-2);
\draw [line width=0.8pt] (10,-2)-- (13,-2);
\draw [line width=0.8pt] (5,-4)-- (10,-2);
\begin{scriptsize}
\draw [fill=black] (2,0) circle (3pt);
\draw [fill=black] (5,0) circle (3pt);
\draw [fill=white] (10,2) circle (3pt);
\draw [fill=white] (10,-2) circle (3pt);
\draw [fill=white] (13,-2) circle (3pt);
\draw [fill=black] (5,-4) circle (3pt);
\draw [fill=black] (5,0.5) node {\Large{$v$}};
\draw [fill=black] (5,-3.5) node {\Large{$v'$}};
\draw [fill=black] (10,-1.5) node {\Large{$u$}};
\draw [fill=black] (10,2.5) node {\Large{$u'$}};
\end{scriptsize}
\end{tikzpicture}    
\caption{The construction of the bipartite graph $H^{(n)}$. The vertices in black belong to $V^{(n)}_1$, and the vertices in white belong to $V^{(n)}_2$. Since $(u', v, u)$ is a cherry with black center and white endpoints, and $(v, u, v')$ is a cherry with white center and black endpoints, the vertex $v$ ($u$, respectively) is included in $H^{(n)}_1$ ($H^{(n)}_2$, respectively), and the edge $uv$ is included in $H^{(n)}$.} 
\label{new fig 3}
\end{figure}

Now, for a given cherry, the probability that it is a border cherry with center in $V^{(n)}_1$ is $\frac{0.0149595}{2}+o(1)$, and there are a.a.s.\ $3n/2 + o(n)$ cherries with centers in $V^{(n)}_1$. Thus, there are a.a.s.\ at least $0.0149595 \cdot \frac{3n}{2}+o(n) \approx 0.022438n+o(n)$ vertices in $W^{(n)}_1$ (concentration around the expected values follows as before). Notice that here, we use Observation~\ref{ob 6.5} to say that a.a.s.\ $|V(H^{(n)})|$ is, up to $o(n)$, equal to the total number of centers of border cherries. We conclude by Lemma~\ref{anticlique} that a.a.s.\ there exists an independent set $I_n$ having at least $\xi_n = \frac{0.022438n}{2}+o(n)=0.011219n+o(n)$ vertices of $W^{(n)}_1$ and $\xi_n$ vertices of $W^{(n)}_2$. Sending the vertices in $I_n\cap W^{(n)}_1$ to $V^{(n)}_2$ and sending back the vertices in $I_n\cap W^{(n)}_2$ to $V^{(n)}_1$ leads a.a.s.\ to a cut $(U^{(n)}_1, U^{(n)}_2)$ with $|U^{(n)}_1| = |U^{(n)}_2| = n/2-o(n)$ and size at most $0.16226n + o(n) - 2\xi_n = 0.139822n+o(n)$. After exchanging another $o(n)$ vertices, $(U^{(n)}_1, U^{(n)}_2)$ may be transformed into a bisection, which finishes the proof.
\end{proof}

We remark that the value obtained is exactly the same as the one given in~\cite{Gamarnik} for the number of edges not going through a maximum cut of a random 3-regular graph (the authors therein prefer to give only two digits after the decimal point). Their result is a corollary of the paper~\cite{Gerencser}. It uses a similar approach to our paper to find a maximum cut based on the eigenvalue $\lambda_{min} = -2\sqrt{2}$ of the transition operator of the random 3-regular graph. The numbers appearing in both calculations are the same. From a combinatorial point of view, the rough intuitive explanation of this phenomenon is the following: in \cite{Gamarnik} the authors pick two maximal independent sets on both sides of the cut (of equal sizes) and then add the remaining vertices in greedy order to the part where they have less neighbors. From an algorithmic point of view, this is equivalent to choosing arbitrarily the part of every vertex outside of the two initial independent sets and then switching vertices having at most one neighbor on the other side, as long as this is possible. We pick a minimum bisection given by the largest eigenvector and switch vertices having at least two neighbors on the other side, as long as this is possible. Despite the fact that this is far from a proof, as mentioned in the introduction, this phenomenon was observed before -- Zdeborov\'{a} and Boettcher conjectured in~\cite{Zdeborova} that the number of edges crossing a minimum bisection in a random 3-regular graph is a.a.s. (up to $o(n)$) equal to the number of edges not crossing a maximum cut in such a graph.\\

\textbf{Non-rigorous improvement.} The previous argument does not take into account that vertices of degree zero in the bipartite graph $(H^{(n)}_1, H^{(n)}_2)$ can be freely switched from one side to the other. Indeed, having once identified the number of these vertices on each side, they can be switched to reduce the cut, and in the remaining graph consisting of vertices of degree one or two, Lemma~\ref{anticlique} can be applied. In order to compute the probability that one vertex $v_4$ in, say, $V^{(n)}_2$ with neighbors $v_3, v_5$ in $V^{(n)}_1$ has degree zero in $H^{(n)}$, we must have for the other two neighbors $v_1, v_2$ of $v_3$ and for the other two neighbors $v_6, v_7$ of $v_5$ that $X^{(n)}_{v_1}\ge 0, X^{(n)}_{v_2}\ge 0, X^{(n)}_{v_6}\ge 0$ and $X^{(n)}_{v_7}\ge 0$ (see Figure~\ref{new fig 4}). Analogous computations hold if the vertex $v_4$ is in $V^{(n)}_1$ and all remaining vertices are in $V^{(n)}_2$.

\begin{figure}
\centering
\begin{tikzpicture}[scale=0.7,line cap=round,line join=round,x=1cm,y=1cm]
\clip(-5.5,-2.5) rectangle (5.5,5);
\draw [line width=0.8pt] (2,1)-- (5,1);
\draw [line width=0.8pt] (2,1)-- (-2,3);
\draw [line width=0.8pt] (2,1)-- (-2,-1);
\draw [line width=0.8pt] (-2,-1)-- (-5,0);
\draw [line width=0.8pt] (-2,-1)-- (-5,-2);
\draw [line width=0.8pt] (-2,3)-- (-5,4);
\draw [line width=0.8pt] (-2,3)-- (-5,2);
\begin{scriptsize}
\draw [fill=black] (2,1) circle (3pt);
\draw [fill=black] (5,1) circle (3pt);
\draw [fill=white] (-2,3) circle (3pt);
\draw [fill=white] (-2,-1) circle (3pt);
\draw [fill=white] (-5,0) circle (3pt);
\draw [fill=white] (-5,-2) circle (3pt);
\draw [fill=white] (-5,4) circle (3pt);
\draw [fill=white] (-5,2) circle (3pt);

\draw [fill=black] (2,1+0.5) node {\Large{$v_4$}};
\draw [fill=black] (-2,3+0.5) node {\Large{$v_5$}};
\draw [fill=black] (-2,-1+0.5) node {\Large{$v_3$}};
\draw [fill=black] (-5,0+0.5) node {\Large{$v_2$}};
\draw [fill=black] (-5,-2+0.5) node {\Large{$v_1$}};
\draw [fill=black] (-5,4+0.5) node {\Large{$v_7$}};
\draw [fill=black] (-5,2+0.5) node {\Large{$v_6$}};
\end{scriptsize}
\end{tikzpicture}
\caption{The vertices in black belong to $V^{(n)}_1$, and the vertices in white belong to $V^{(n)}_2$. In the figure $v_4$ is an isolated vertex in $H^{(n)}$, since neither $v_3$, nor $v_5$ participate in a border cherry.}
\label{new fig 4}
\end{figure}
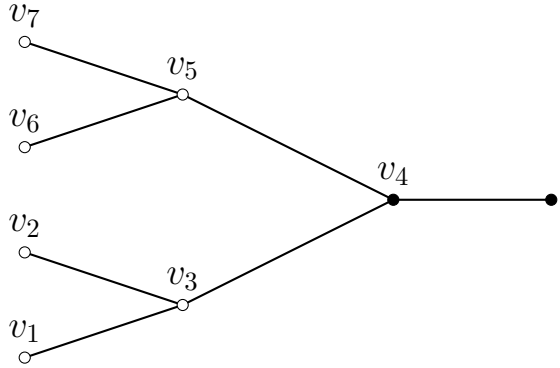

Now, let us compute
\begin{equation}\label{eq:prob}
\mathbb P (X^{(n)}_{v_1} > 0, X^{(n)}_{v_2} > 0, X^{(n)}_{v_3} > 0, -X^{(n)}_{v_4} > 0, X^{(n)}_{v_5} > 0, X^{(n)}_{v_6}> 0, X^{(n)}_{v_7} > 0).
\end{equation}
We proceed as before. Using~\eqref{covariance}, we obtain that the covariance matrix of the random vector 
\begin{equation*}
(X^{(n)}_{v_1}, X^{(n)}_{v_2}, X^{(n)}_{v_3}, -X^{(n)}_{v_4}, X^{(n)}_{v_5}, X^{(n)}_{v_6}, X^{(n)}_{v_7})^T
\end{equation*}
converges to the matrix $B$, given by
\begin{equation*}
\begin{pmatrix}
1 & 5/6 & 2\sqrt{2}/3 & -5/6 & 1/\sqrt{2} & 7/12 & 7/12 \\
5/6 & 1 & 2\sqrt{2}/3 & -5/6 & 1/\sqrt{2} & 7/12 & 7/12 \\
2\sqrt{2}/3 & 2\sqrt{2}/3 & 1 & -2\sqrt{2}/3 & 5/6 & 1/\sqrt{2} &
1/\sqrt{2} \\
-5/6 & -5/6 & -2\sqrt{2}/3 & 1 & -2\sqrt{2}/3 &-5/6 &-5/6 \\
1/\sqrt{2} & 1/\sqrt{2} & 5/6 & -2\sqrt{2}/3 & 1 & 2\sqrt{2}/3 & 2\sqrt{2}/3 \\
7/12 & 7/12 & 1/\sqrt{2} & -5/6 & 2\sqrt{2}/3 & 1 & 5/6 \\
7/12 & 7/12 & 1/\sqrt{2} & -5/6 & 2\sqrt{2}/3 & 5/6 & 1 
\end{pmatrix}  
\end{equation*}

Using the eigenvalue decomposition of $B$, we find the matrix $V$ satisfying $VV^T=B$ with approximate entries
\begin{equation*}
\begin{pmatrix}
0 & 0 & -0.2887 & 0 & -0.4082 & 0.0676 & 0.8634 \\
0 & 0 & 0.2887 & 0 & -0.4082 & 0.0676 & 0.8634 \\ 
0 & 0 & 0 & 0 & -0.2887 & -0.0215 & 0.9572 \\
0 & 0 & 0 & 0 & 0 & 0.1962 & -0.9806 \\
0 & 0 & 0 & 0 & 0.2887 & -0.0215 & 0.9572 \\
0 & 0 & 0 & -0.2887 & 0.4082 & 0.0676 & 0.8634 \\
0 & 0 & 0 & 0.2887 & 0.4082 & 0.0676 & 0.8634 
\end{pmatrix}.
\end{equation*}

In particular, applying $V$ on the left to the standard normal vector $(Z_{v_1}, Z_{v_2}, Z_{v_3}, Z_{v_4}, Z_{v_5}, Z_{v_6}, Z_{v_7})^T$ gives us the desired joint distribution of the seven-dimensional vector $(Y_{v_1}, Y_{v_2}, Y_{v_3}, -Y_{v_4}, Y_{v_5}, Y_{v_6}, Y_{v_7})^T$ (by abuse of notation we assume again that the vertices $(v_i)_{i\ge 1}$ are included in $G(n,3)$ as well as in $T_3$). Indeed, recall that the linear factors of iid $(X^{(n)}_v)_{v\in V(G(n,3))}$ converge in distribution to the Gaussian wave $(Y_v)_{v\in V(T_3)}$. More precisely, in order for all seven coordinates of this resulting vector to be positive, the following inequalities must hold:

\begin{eqnarray}
Z_{v_7} &>& -(V_{1,3}Z_{v_3}+V_{1,5}Z_{v_5}+V_{1,6}Z_{v_6})/V_{1,7} \label{eq1} \\ 
Z_{v_7} &>& -(V_{2,3}Z_{v_3}+V_{2,5}Z_{v_5}+V_{2,6}Z_{v_6})/V_{2,7} \label{eq2} \\ 
Z_{v_7} &>& -(V_{3,5}Z_{v_5}+V_{3,6}Z_{v_6})/V_{3,7} \label{eq6} \\
Z_{v_7} &\le& V_{4,6}Z_{v_6}/(-V_{4,7}) \label{eq7} \\ 
Z_{v_7} &>& -(V_{5,5}Z_{v_5}+V_{5,6}Z_{v_6})/V_{5,7} \label{eq5} \\ 
Z_{v_7} &>& -(V_{6,4}Z_{v_4}+V_{6,5}Z_{v_5}+V_{6,6}Z_{v_6})/V_{6,7} \label{eq3} \\ 
Z_{v_7} &>& -(V_{7,4}Z_{v_4}+V_{7,5}Z_{v_5}+V_{7,6}Z_{v_6})/V_{7,7} \label{eq4} 
\end{eqnarray}

Defining $LB$ to be the maximum of all lower bounds corresponding to inequalities~(\ref{eq1}, \ref{eq2}, \ref{eq6}, \ref{eq5}, \ref{eq3}, \ref{eq4}) for $X_{v_7}$, and  $UB=\max\{LB,\mbox{ the right hand side of  }(\ref{eq7})\}$, computing the desired probability~\eqref{eq:prob} is thus equivalent to computing the following integral:

\begin{equation*}
    \int_{-\infty}^{\infty}\int_{-\infty}^{\infty}\int_{-\infty}^{\infty}\int_{-\infty}^{\infty} \int_{LB}^{UB} e^{-\frac{x_3^2}{2}-\frac{x_4^2}{2}-\frac{x_5^2}{2}-\frac{x_6^2}{2}-\frac{x_7^2}{2}}\frac{1}{(2\pi)^{5/2}} dx_7 dx_6 dx_5 dx_4 dx_3.
\end{equation*}

Unfortunately, we are not able to compute this integral (not even numerically). Performing a Monte Carlo simulation with $3 \times 10^7$ iterations (see the authors' wegpages for the corresponding Python code) and checking in each round whether the new sample of $Z_{v_3}, Z_{v_4}, Z_{v_5}, Z_{v_6}, Z_{v_7}$ satisfies all inequalities, we obtain an estimated value of $0.0025247$. Therefore, the total number of border cherries having their centers in $H^{(n)}_2$ and of degree zero in $H^{(n)}$ is approximately $3\times 0.0025247n = 0.0075741n$ a.a.s. (concentration around the expected value follows as before). By symmetry, this clearly holds also for border cherries having their centers in $H^{(n)}_1$ and of degree zero in $H^{(n)}$. \par

We can therefore improve the previous argument in the following way: first switch the apprximately $2\times 0.0075741n$ centers of border cherries having degree zero in $H^{(n)}$ ($0.0075741n$ in both parts of the graph) to the other side of the cut (that is, change the associated value $\phi_v$ from $0$ to $1$ and vice versa), yielding a total a.a.s.\ gain of about $0.0151482 n$ edges. These centers of cherries ``disappear'' as vertices of $H^{(n)}$ since they are no longer centers of border cherries. We apply Lemma~\ref{anticlique} to the remaining bipartite graph with parts of size roughly $0.022438n-0.0075741n\approx 0.0148637n$ to obtain an independent set of size at least $2\times \frac{0.0148637n}{2}$. Switching these to the other side gives another gain of $0.0148637+o(n)$, thus yielding a total cut $(U^{(n)}_1, U^{(n)}_2)$ size of at most $0.16226n-0.016912n-0.0148637n\approx 0.132248n$ with $|U^{(n)}_1|- |U^{(n)}_2| = o(n)$ a.a.s. As above, this cut may be transformed into a bisection by switching the sides of at most $o(n)$ vertices of $G(n,3)$. We remark that this numerical value is also close to the one in Remark 4.1 given in~\cite{Gerencser} for numerical calculation of the number of edges, not included in a maximum cut of $G(n,3)$, thus giving another indication that the conjecture of~\cite{Zdeborova} holds true.\par

Clearly, taking into account different substructures (such as, for example, paths $v_1, v_2, \dots, v_k$ of three or more consecutive vertices, for which $(X_{v_i})_{1\le i\le k}$ have all the same sign, and each of $(v_i)_{1\le i\le k}$ having exactly one neighbor $(u_i)_{1\le i\le k}$, respectively, such that for every $i\in [k]$, $X_{v_i}$ and $X_{u_i}$ have different signs, or other structures that were observed in the proof of the lower bound), we could further improve the previous bound. Since we cannot evaluate the corresponding integrals numerically, we stopped as this point.

\bibliographystyle{plain}
\bibliography{References}

\end{document}